\documentclass[a4paper,reqno, 10pt]{amsart}

\usepackage{amsmath,amssymb,amsfonts,amsthm,mathrsfs}

\usepackage{verbatim,wasysym,cite}

\usepackage[latin1]{inputenc}
\usepackage{microtype}
\usepackage{color,enumitem,graphicx}
\usepackage[colorlinks=true,urlcolor=blue, citecolor=red,linkcolor=blue,
linktocpage,pdfpagelabels, bookmarksnumbered, bookmarksopen]{hyperref}
\usepackage[english]{babel}
\usepackage[symbol]{footmisc}
\renewcommand{\epsilon}{{\varepsilon}}
\usepackage{enumitem}

\numberwithin{equation}{section}
\newtheorem{theorem}{Theorem}[section]
\newtheorem{lemma}[theorem]{Lemma}
\newtheorem{remark}[theorem]{Remark}

\newtheorem{proposition}[theorem]{Proposition}
\newtheorem{corollary}[theorem]{Corollary}
\newtheorem{claim}[theorem]{Claim}

\newcommand{\R}{\mathbb R}
\newcommand{\N}{\mathcal N}

\def\({\left(}
\def\){\right)}
\def\<{\left\langle}
\def\>{\right\rangle}

\def\Sch{{\mathcal S}}

\def\F{\mathcal F}
\def\K{\mathcal K}
\def\TT{\mathcal T}
\def\Q{\mathcal Q}

\def\G{\mathcal G}
\def\L{\mathcal L}
\def\EE{\mathcal E}

\def\Lm{\rm L}

\DeclareMathOperator{\RE}{Re}
\DeclareMathOperator{\IM}{Im}

\newcommand{\qtq}[1]{\quad\text{#1}\quad}

\newcommand{\ve}[1]{\textbf{#1}}

\begin{document}

\title[Dynamic of threshold solutions]{Dynamics of the Energy-Critical Nonlinear Schr\"{o}dinger System in $\R^{4}$}

\author[Alex H. Ardila]{Alex H. Ardila}
\address{Department of Mathematics, Universidad del Valle, Colombia} 
\email{ardila@impa.br}

\begin{abstract} 
In this paper, we investigate the dynamics of radial solutions at threshold energy for a 3-component Schr\"{o}dinger system with cubic nonlinearity in four dimensions. The main difference from the cases previously addressed in the literature is that, in our system,  the kernel of the imaginary part $L_I$ of the linearized operator $-i\L=L_{R}+iL_{I}$ has dimension 2. 
To overcome this difficulty, we carry out a detailed study of the coercivity properties of these operators. We also introduce a new modulation parameter associated with the additional eigenfunction in the kernel of the operator $L_{I}$, which enables us to perform the modulation analysis and establish the uniqueness of exponentially decaying solutions to the linearized equation.
\end{abstract}


\maketitle

\medskip

\section{Introduction}
\label{sec:intro}

We consider the Cauchy problem for the following 3-component Schr\"{o}dinger system with cubic nonlinearity in four dimensions:

\begin{equation} \label{NLS}
\begin{cases} 
i \partial_t u_1 + \frac{1}{2m_1} \Delta u_1 + 2\overline{u}_1 u_2 u_3= 0, \\ 
i \partial_t u_2 + \frac{1}{2m_2} \Delta u_2 + u_1^2  \overline{u}_3 = 0, \\ 
i \partial_t u_3 + \frac{1}{2m_3} \Delta u_3 + u_1^{2} \overline{u}_2 = 0, \\ 
\end{cases}
 \end{equation}
where $\ve{u}=(u_1, u_2, u_3): \mathbb{R}\times \mathbb{R}^{4}\to \mathbb{C}^{3}$ and $m_1$, $m_2$, $m_3$ are positive coupling constants. Such systems with polynomial-type nonlinear terms arise in the study of laser-plasma interactions; For further details, see \cite{CCC, CoDiSa} and the references therein.

In what follows, we use the vector notation $\ve{u} = (u_1, u_2, u_{3})$, where  $\ve{u}$ is treated as a column vector. The local well-posedness of the Cauchy problem for \eqref{NLS} was established in \cite[Proposition 1.1]{OgaTsu2023}. We also refer to \cite{Tsu2025} for a detailed study of well-posedness for multicomponent nonlinear Schrödinger equations with Sobolev-critical nonlinearity.   Specifically, for initial data $\ve{u}_0 \in (\dot{H}^1(\mathbb{R}^4))^{3}$, there exists a unique solution  $\ve{u} \in C\big(I; (\dot{H}^1(\mathbb{R}^4))^3\big)$, defined on a maximal interval $I = (-T_{-}(\mathbf{u}_0), T_{+}(\mathbf{u}_0))$. Moreover, this solution conserves the energy $E(\ve{u}(t))=E(\ve{u}_{0})$ for all $t\in I$, where
\begin{align}\label{Ener1}
	E(\ve{u})=K(\ve{u})-2P(\ve{u}),
\end{align}
with
\begin{align}\label{HNN}
	K(\ve{u}):=\sum_{k=1}^{3}\tfrac{1}{2m_{k}}\|\nabla u_{k}\|_{L^{2}(\mathbb{R}^{4})}^{2}
\qtq{and}
P(\ve{u}):=\RE\int\limits_{\mathbb{R}^{4}}\overline{u}_{1}^{2}(x)u_{2}(x)u_{3}(x)dx.
\end{align}

The system \eqref{NLS} exhibits two fundamental symmetries: scaling invariance and phase rotation invariance. Specifically, if $\mathbf{u} = (u_1, u_2, u_3)$ is a solution to \eqref{NLS}, then the following are also solutions:
\begin{enumerate}[label=(\roman*)]
    \item {Scaling symmetry}: 
          $\lambda^{-1}\mathbf{u}(\lambda^{-2}t, \lambda^{-1}x)$ for any scaling parameter $\lambda > 0$;
    \item {Phase rotation symmetry}: 
          $\big(e^{i(\theta_1+\theta_2)}u_1(t,x), e^{2i\theta_1}u_2(t,x), e^{2i\theta_2}u_3(t,x)\big)$ for any phases $\theta_1, \theta_2 \in \mathbb{R}$.
\end{enumerate}

The scattering versus blow-up dichotomy for system \eqref{NLS} is investigated in \cite{OgaTsu2023, Tsu2025}. More precisely, the authors in \cite[Theorem 1.3]{OgaTsu2023} established the existence of ground states of the form $\Q = (Q_1, Q_2, Q_3)$, where 
\begin{align*}
Q_{1}(x) &= \left( \tfrac{1}{4 m_2 m_3} \right)^{\frac{1}{4}} Q(x), \quad
Q_{2}(x) = \tfrac{1}{2}\left( \tfrac{m_2}{m^{2}_1 m_3} \right)^{\tfrac{1}{4}} Q(x), \\
Q_{3}(x) &= \tfrac{1}{2} \left( \tfrac{m_3}{m^{2}_1 m_2} \right)^{\frac{1}{4}} Q(x),
\end{align*}
with $Q(x) = (1 + |x|^2/8)^{-1}\in \dot{H}^1(\R^4)$. Note that $Q$ is the positive solution to the nonlinear elliptic equation
\begin{align}\label{groundS}
\Delta Q + Q^{3} = 0.
\end{align}
The uniqueness of the ground state $\Q = (Q_1, Q_2, Q_3)$ (modulo symmetries) is proved in Proposition~\ref{UBS} below.

In~\cite[Theorem~1.4]{OgaTsu2023}, the authors established a classification of radial solutions to \eqref{NLS} with energy below the ground state threshold $E(\Q)$. Under the mass resonance condition $2m_1 + m_2 = m_3$, for any radial initial data $\mathbf{u}_0 \in (\dot{H}^1(\mathbb{R}^4))^3$ satisfying $E(\mathbf{u}_0) < E(\Q)$, the corresponding solution $\mathbf{u}(t)$ exhibits a sharp dichotomy: either (i) global existence and scattering when $K(\mathbf{u}_0) < K(\Q)$, or (ii) finite-time blow-up when $K(\mathbf{u}_0) > K(\Q)$, provided $\mathbf{u}_0$ additionally satisfies either $|x|\mathbf{u}_0 \in (L^2(\mathbb{R}^4))^3$ or $\mathbf{u}_0 \in (H^1(\mathbb{R}^4))^3$. Notice that this  classification  depends on the mass resonance condition; see \cite[Appendix]{OgaTsu2023} for further discussion of this assumption.
We recall that a solution $\mathbf{u}(t)$ of \eqref{NLS} scatters in $(\dot{H}^1(\mathbb{R}^4))^3$ if there exist $(u_{1\pm}, u_{2\pm}, u_{3\pm}) \in (\dot{H}^1(\mathbb{R}^4))^3$ such that  
\[
\lim_{t \to \pm \infty} \|u_k(t) - e^{\frac{it}{2m_k} \Delta} u_{k\pm}\|_{\dot{H}^1(\mathbb{R}^4)} = 0 \quad \text{for } k=1, 2, 3.
\]
An analogous result for multicomponent nonlinear Schr\"odinger equations with Sobolev-critical nonlinearity can be found in \cite[Theorem~1.4]{Tsu2025}.

In this paper, we investigate the behavior of solutions precisely at the energy threshold $E(\Q)$. More specifically, we establish the following results. First, we construct two special solutions that will enable us to classify the threshold dynamics.

\begin{theorem}\label{Gcharc}
Fix $m_1, m_2, m_3 > 0$. Under the mass resonance condition $2m_1 + m_2 = m_3$, the system \eqref{NLS} admits two special radial solutions $\G^+(t)$ and $\G^-(t)$ with the following properties:
\begin{enumerate}[label=(\roman*)]
    \item For the solution $\G^+$:
    \begin{itemize}
        \item Energy threshold: $E(\G^+(t)) = E(\Q)$;
        \item Global existence in positive time: $T_+(\G^+) = +\infty$;
         \item Supercritical condition: $K(\G^+(0)) > K(\Q)$.
    \end{itemize}
    
    \item For the solution $\G^-$:
    \begin{itemize}
        \item Energy threshold: $E(\G^-(t)) = E(\Q)$;
        \item Subcritical condition: $K(\G^-(0)) < K(\Q)$;
        \item Global existence in positive and negative time: $T_+(\G^-) = +\infty$ and $T_-(\G^-) = +\infty$;
        \item Scattering behavior: $\G^-(t)$ scatters as $t \to -\infty$.
    \end{itemize}
\end{enumerate}
Moreover,
\[
\lim_{t \to +\infty}\G^{-}(t)=\Q \quad\text{in $(\dot{H}^{1}(\R^{4}))^{3}$.}
\]
\end{theorem}

Our second result provides a  classification of solution behaviors at the energy threshold $E(\Q)$. More precisely, 

\begin{theorem}\label{TH22}
Fix $m_1, m_2, m_3 > 0$ satisfying the mass resonance condition $2m_1 + m_2 = m_3$. Let $\mathbf{u}(t)$ be the solution to \eqref{NLS} with radial initial data $\mathbf{u}_0 \in (\dot{H}^1(\mathbb{R}^4))^3$ such that $E(\mathbf{u}_0) = E(\Q)$. Then the following classification holds:
\begin{enumerate}[label=(\roman*)]
    \item Subcritical case. If $K(\mathbf{u}_0) < K(\Q)$, then 
    \begin{itemize}
        \item The solution $\mathbf{u}(t)$ is global in time;
        \item Either $\mathbf{u}$ coincides with $\G^-$ modulo the symmetries of the equation or $\mathbf{u}(t)$ scatters in both time directions;
    \end{itemize}
    
    \item If $K(\mathbf{u}_0) = K(\Q)$, then  $\mathbf{u}=\Q$ modulo symmetries of the equation.

    \item Supercritical  case. If $K(\mathbf{u}_0) > K(\Q)$ with $\mathbf{u}_0 \in (L^2(\mathbb{R}^4))^3$, then either $\mathbf{u}$ coincides with $\G^+$ modulo symmetries of the equation or the solution blows up in finite time.
\end{enumerate}
\end{theorem}

It is worth emphasizing that the coupling condition $2m_1 + m_2 = m_3$ plays a fundamental role in the analysis of the dynamics of \eqref{NLS}. This condition is necessary for deriving the virial identity presented in Lemma~\ref{VirialIden}, which in turn is essential for establishing the exponential convergence of solutions $\mathbf{u}(t)$ to the ground state $\mathcal{Q}$ (modulo the symmetries of the equation) at the energy threshold. For further details, we refer to the proofs of Propositions~\ref{CompacDeca} and~\ref{SupercriQ}.

To prove Theorem~\ref{TH22}, we closely follow the argument developed by T. Duyckaerts and F. Merle \cite{DuyMerle2009}. To this end, we define the ground state orbit $\mathcal{B}$ associated to $\Q$ as:
\[
\mathcal{B} := \left\{ \Q_{[\theta_1, \theta_2, \lambda]} : \theta_1, \theta_2 \in \mathbb{R}, \lambda > 0 \right\},
\]
where
\[
\Q_{[\theta_1, \theta_2, \lambda]} := \left( e^{i(\theta_1+\theta_2)} \lambda^{-1} Q_1(\lambda^{-1}x), \ e^{2i\theta_1} \lambda^{-1} Q_2(\lambda^{-1}x), \ e^{2i\theta_2} \lambda^{-1} Q_3(\lambda^{-1}x) \right).
\]
We then show that any solution $\mathbf{u}(t)$ of \eqref{NLS} with initial data $\mathbf{u}_0 \in (\dot{H}^1(\mathbb{R}^4))^3$ satisfying the conditions of Theorem~\ref{TH22} must exhibit exactly one of the following seven behaviors:
\begin{enumerate}
    \item Scattering in both time directions ($t \to \pm\infty$);
    \item Trapped by $\mathcal{B}$ as $t \to +\infty$ and scattering as $t \to -\infty$;
    \item Trapped by $\mathcal{B}$ as $t \to -\infty$ and scattering as $t \to +\infty$;
    \item Finite-time blow-up in both time directions;
    \item Trapped by $\mathcal{B}$ as $t \to +\infty$ and finite-time blow-up for $t < 0$;
    \item Trapped by $\mathcal{B}$ as $t \to -\infty$ and finite-time blow-up for $t > 0$;
		\item The initial data $\mathbf{u}_0$ belongs to the orbit $\mathcal{B}$.
\end{enumerate}
Here, ``trapped by $\mathcal{B}$'' means that the solution remains within an $\mathcal{O}(\varepsilon)$-neighborhood of $\mathcal{B}$ in the $(\dot{H}^1(\mathbb{R}^4))^3$ norm after some time (or before some time). Later, using the special solutions  $\G^{\pm}$, we characterize all possible solutions exhibiting the asymptotic behaviors (2), (3), (5), and (6), proving their uniqueness up to symmetries of the system. This yields Theorem~\ref{TH22} as a direct consequence.

Recent years have witnessed significant advances in the analysis of solution behavior for systems of nonlinear Schrödinger equations with polynomial-type nonlinearities. Substantial progress has been made in understanding both the local and global dynamics of these systems. We can mention some recent works in this direction: the existence of ground states and well-posedness results have been established in \cite{NayaOzaTana, LiHayashi2014, Zhang2016, Masaki2022}, while orbital stability and instability properties have been investigated in \cite{COCOOh, SCO, AA2018, AeDiFo2021, FukayaHayashiInui2024}. The dynamics below the mass-energy threshold have been analyzed in \cite{MEXu, GaoMengXuZheng, HaInuiNi, OgaTsu2023, Tsu2025, MasakiTsukuda2023}, with critical threshold behavior examined in \cite{ArdilaCelyMeng, CAPA2022, NogueraPastor2022}.

The main difficulty presented by the system \eqref{NLS} stems from the two degrees of freedom in the phase rotation symmetry, which leads to $\dim \ker(L_{I}) = 2$, where $L_{I}$ is the imaginary part of the linearized operator $-i\mathcal{L} = L_{R} + iL_{I}$ (this operator can be found in Section~\ref{EqLinea}). To the best of our knowledge, in all previous works studying energy threshold dynamics for the NLS, the kernel of the imaginary part has dimension 1; See, for example, \cite{DuyMerle2009, CamposFarahRoudenko} for the classical energy-critical NLS case; \cite{MIWUGU2015, LiLiuTangXu2025} for the energy-critical Hartree equation; \cite{ArdilaCelyMeng} for the energy-critical NLS system with quadratic interaction; \cite{YangZengZhang2022} for the energy-critical NLS with inverse square potential; and \cite{LiuYangZhang2024} for the energy-critical inhomogeneous NLS, among others.

To overcome this difficulty, we carry out a detailed study of the coercivity properties of these operators. Furthermore, We introduce a new modulation parameter associated with the additional eigenfunction in the kernel of the operator $L_{I}$. By studying the decay of solutions to the linearized equation and following the arguments developed in \cite{DuyMerle2009, MIWUGU2015}, we obtain all seven aforementioned behaviors and establish the uniqueness (modulo symmetries) of solutions satisfying the threshold scenarios (2), (3), (5), and (6).



In the rest of the introduction, let us briefly describe the organization of the
paper and the strategy of proof for Theorem~\ref{Gcharc} and Theorem~\ref{TH22}. In Section~\ref{S0:preli}, we introduce the notation used throughout the text and revisit the Cauchy problem. 
In Section~\ref{S1:Vari}, we characterize the functions that achieve equality in the Gagliardo-Nirenberg inequality \eqref{GNI}. We show that these are precisely the translations, dilations, and phase rotations of $\mathcal{Q}$. This characterization plays a crucial role in the modulation analysis and in understanding the dynamic behavior of the solution at the energy threshold.

Furthermore, we establish the virial identity. This identity is a key element for proving the exponential convergence of the solution $\mathbf{u}(t)$ to the ground states $\Q$ at the energy threshold, as shown in Propositions~\ref{CompacDeca} and~\ref{SupercriQ}. Note that to derive the virial identity, it is necessary to assume the coupling condition $2m_1 + m_2 = m_3$. This section also presents several variational characterizations of $\Q$ which will be useful for the subsequent modulation analysis.

In Section~\ref{EqLinea}, we study the coercive properties of the linearized operators $L_I$ and $L_R$, which arise from linearizing the Schr\"{o}dinger system around the ground state $\Q$. The main results of this section are Lemmas~\ref{CoerLi11} and~\ref{Coer11}, which establish that, under suitable orthogonality conditions, $L_I$ and $L_R$ are coercive. This coercivity is essential for the modulation analysis.

Unlike the scalar case, where the kernel of the imaginary part of the linearized operator is one-dimensional, in this system the kernel of $L_I$ is two-dimensional due to the system's two phase invariances. To address this difficulty and establish coercivity, we transform $L_I$ and $L_R$ via a change of variables (cf. proof of Lemma~\ref{Apx22}).  This transformation allows us to diagonalize the operators into blocks involving well-known scalar operators. The coercivity of these scalar operators is already established in the theory for the scalar case; from this fact, we can derive the coercivity and spectral properties of $L_I$ and $L_R$, which will be used throughout this work.

In Section~\ref{S:Modula}, we establish the modulation analysis for radial solutions near the ground state $\Q$. The central result, Proposition~\ref{ModilationFree}, shows that any threshold solution $\mathbf{u}(t)$  can be uniquely decomposed as $\mathbf{u}_{[\eta(t),\theta(t),\mu(t)]}(t) = (1+\alpha(t))\Q + \mathbf{h}(t)$, where the parameters $\eta(t)$, $\theta(t)$, and $\mu(t)$ satisfy the estimates \eqref{EstimateOne} and \eqref{EstimateFree}. Note that we introduce two phase parameters $\eta(t)$ and $\theta(t)$, due to the two-dimensional kernel of $L_I$. This decomposition provides a precise description of the evolution of $\mathbf{u}(t)$ near $\Q$.

In Sections~\ref{S: SubThres} and~\ref{S: SuperThres}, we study solutions with initial data satisfying parts (i) and (iii) of Theorem~\ref{TH22}. The main techniques involve using a virial argument and a concentration-compactness approach adapted to the system \eqref{NLS} to establish the exponential decay \eqref{PIN} and~\eqref{DoQ11} of $\delta(t)$ for large positive time. This decay, combined with modulational stability, implies the exponential convergence in the positive time direction to $\Q$ (up to scaling and phase rotation). In contrast to the scalar case, obtaining this exponential convergence to $\Q$ requires careful consideration of both phase parameters ($\eta(t)$ and $\theta(t)$) associated with the additional symmetries of the system \eqref{NLS}.

In Section~\ref{SpecLine}, we establish the spectral properties of the linearized operator $\mathcal{L}$ around $\mathcal{Q}$, which are derived from the spectral analysis of the component operators $L_I$ and $L_R$. We introduce a quadratic form $\mathcal{F}$ associated with $\mathcal{L}$ and characterize two subspaces $G^\perp \cap \dot{H}^1_{\text{rad}}$ and $\tilde{G}^\perp \cap \dot{H}^1_{\text{rad}}$ within $\dot{H}^1$ where $\mathcal{F}$ remains positive (coercive), effectively avoiding the neutral and negative directions of the linearized dynamics. These spectral results are fundamental for the subsequent construction and uniqueness proof of the special radial solutions $\G^{\pm}(t)$ in Sections~\ref{S:Existence} and \ref{S:uniq}.

Section~\ref{S:Existence} is devoted to proving Theorem~\ref{Gcharc}. Specifically, using the spectral properties of the real eigenvalues of the linearized operator $\mathcal{L}$ and applying a fixed-point argument, we construct the radial solutions $\mathcal{G}^{\pm}(t)$ established in Theorem~\ref{Gcharc}.

In Section~\ref{S:uniq}, we utilize the positivity of the quadratic form $\mathcal{F}$ over $G^\perp \cap \dot{H}^1_{\text{rad}}$ to study the exponential decay properties of solutions to the linearized equation. In contrast to the scalar case, here we must introduce two coordinate functions associated with the two eigenfunctions spanning the kernel of $L_I$. For these coordinate functions, we establish specific exponential decay estimates, which in turn enable us to derive exponential decay for solutions of the linearized equation (see \eqref{bqBp} for details). Finally, we apply these exponential decay results to prove the uniqueness of the special solutions. Furthermore, with the uniqueness of special solutions established, in Section~\ref{S:proof} we provide the proof of Theorem~\ref{TH22}.

In Appendix~\ref{S:A2}, we demonstrate that the linearized operator $\mathcal{L}$ possesses at least one negative eigenvalue. This spectral information is crucial for both the construction and uniqueness proof of the special solutions in Sections~\ref{S:uniq} and \ref{S:Existence}.

\section{Notation and Local theory}\label{S0:preli}
For any $s \geq 0$, we denote $\dot{H}^{s}(\mathbb{R}^{4}:\mathbb{C}) \times \dot{H}^{s}(\mathbb{R}^{4}:\mathbb{C}) \times \dot{H}^{s}(\mathbb{R}^{4}:\mathbb{C})$ by $(\dot{H}^{s}(\mathbb{R}^{4}:\mathbb{C}))^{3}$, equipped with the standard norm. Similarly, we write $(H^{s}(\mathbb{R}^{4}:\mathbb{C}))^{3}$ to denote $H^{s}(\mathbb{R}^{4}:\mathbb{C}) \times H^{s}(\mathbb{R}^{4}:\mathbb{C}) \times H^{s}(\mathbb{R}^{4}:\mathbb{C})$.  

For a time interval $I$, we use the following notation:
\begin{equation}\label{SFU}
\begin{aligned}
&S(I) = \big(L^{6}_{t}L^{6}_{x}(I \times \mathbb{R}^{4})\big)^{3}, \quad
Z(I) = \big(L^{6}_{t}L^{\frac{12}{5}}_{x}(I \times \mathbb{R}^{4})\big)^{3}, \\
&N(I) = \big(L^{2}_{t}L^{\frac{4}{3}}_{x}(I \times \mathbb{R}^{4})\big)^{3}, \quad \mathcal{S} := \big(\mathcal{S}(\mathbb{R}^{4})\big)^{3},
\end{aligned}
\end{equation}
where $\mathcal{S}(\mathbb{R}^{4})$ denotes the Schwartz space. Furthermore, when no confusion arises, we simply write  
\[\dot{H}^{s} := (\dot{H}^{s}(\mathbb{R}^{4}:\mathbb{C}))^{3} \qtq{and} L^{p} := (L^{p}(\mathbb{R}^{4}:\mathbb{C}))^{3}.\]

We recall the Sobolev inequality in $\mathbb{R}^4$:
\begin{align}\label{SobL}
	\|f\|_{L^4(\mathbb{R}^4)} \leq G_4 \|\nabla f\|_{L^2(\mathbb{R}^4)},
\end{align}

for $f \in \dot{H}^{1}(\mathbb{R}^4)$, where $G_4$ is the best Sobolev constant.

By {solution}  to \eqref{NLS}, we mean a function $\mathbf{u} \in C_t(I, \dot{H}^1_x(\mathbb{R}^4))$ defined on an interval $I \ni 0$ that satisfies the {Duhamel formula}:
	\[
	\mathbf{u}(t) = U(t)\mathbf{u}_0 + i\int_0^t U(t - \tau)F(\mathbf{u}(\tau))\,d\tau, \quad \text{for } t \in I,
	\]
	where
	\[
	U(t) =
	\begin{pmatrix}
		e^{\frac{1}{2m_1}it\Delta} & 0 & 0 \\
		0 & e^{\frac{1}{2m_2}it\Delta} & 0 \\
		0 & 0 & e^{\frac{1}{2m_3}it\Delta}
	\end{pmatrix},
	\quad
	F(\mathbf{u}) := 
	\begin{pmatrix}
		2\overline{u}_1 u_2 u_3 \\
		u_1^2 \overline{u}_3 \\
		u_1^2 \overline{u}_2
	\end{pmatrix}.
	\]
	
	The solution $\mathbf{u}$ to the system on an interval $I \ni t_0$ satisfies the following {Strichartz estimates} (cf. \cite{OgaTsu2023, Tsu2025}):
	\[
	\left\| \int_{t_0}^t U(t - s)F(\mathbf{u}(s))\,ds \right\|_{Z(I)} \leq C\|F(\mathbf{u})\|_{N(I)},
	\]
	and
	\begin{align}\label{Estriz}
			\|\mathbf{u}\|_{Z(I)} \leq C \left( \|\mathbf{u}(t_0)\|_{L^2(\mathbb{R}^4)} + \|F(\mathbf{u})\|_{N(I)} \right).
	\end{align}

	\subsection*{Local theory}
The following results can be found in \cite{OgaTsu2023, Tsu2025}.

\begin{proposition}\label{LCP}
Fix $\ve{u}_{0} \in \dot{H}^{1}$. Then the following hold:
\begin{enumerate}[label=\rm{(\roman*)}]
    \item There exist $T_{+}(\ve{u}_{0}) > 0$, $T_{-}(\ve{u}_{0}) > 0$, and a unique solution  
    $\ve{u}: (-T_{-}(\ve{u}_{0}), T_{+}(\ve{u}_{0})) \times \mathbb{R}^{4} \to \mathbb{C}$ to \eqref{NLS} with initial data $\ve{u}(0) = \ve{u}_{0}$.
    
    \item Finite blow-up criterion. If $T_{+} = T_{+}(\ve{u}_{0}) < +\infty$, then $\|\ve{u}\|_{L^{6}_{t,x}((0, T_{+}) \times \mathbb{R}^{4})} = +\infty$. An analogous statement holds for negative time.
\end{enumerate}
\end{proposition}

\begin{proposition}[Sufficient condition for scattering]
Let $\ve{u}(t)$ be a global $\dot{H}^{1}$ solution in positive time ($T_{+} = +\infty$). If $\ve{u}$ remains uniformly bounded in $L_{t,x}^{6}$, i.e.,
\[
\|\ve{u}\|_{L^{6}_{t,x}([0, +\infty) \times \mathbb{R}^{4})} < \infty,
\]
then $\ve{u}$ scatters in $\dot{H}^{1}$.
\end{proposition}

We also have the following stability property:

\begin{lemma}[Long-time perturbation theory]\label{ConSNLS}
Let $I \subset \mathbb{R}$ be a time interval containing $0$, and let $\tilde{\ve{u}}$ be a solution to \eqref{NLS} on $I$. 
Assume that for some $L > 0$,
\[
\sup_{t \in I} \|\tilde{\ve{u}}(t)\|_{\dot{H}^{1}} \leq L \quad \text{and} \quad 
\|\tilde{\ve{u}}\|_{L_{t,x}^{6}(I \times \mathbb{R}^{4})} \leq L.
\]
There exists $\epsilon_0(L) > 0$ such that if
\[
\|\ve{u}_{0} - \tilde{\ve{u}}_{0}\|_{\dot{H}^{1}} \leq \epsilon
\]
for $0 < \epsilon < \epsilon_{0}(L)$, then there exists a unique solution $\ve{u}$ to \eqref{NLS} with initial data $\ve{u}_{0}$ such that
\[
\sup_{t \in I} \|\ve{u}(t) - \tilde{\ve{u}}(t)\|_{\dot{H}^{1}} \leq C(L)\epsilon
\quad \text{and} \quad
\|\ve{u}\|_{L_{t,x}^{6}(I \times \mathbb{R}^{4})} \leq C(L).
\]
\end{lemma}

Finally, the following result characterizes the solution dynamics below the energy threshold. For the proof, we refer to \cite[Theorem 1.4]{OgaTsu2023} and  \cite[Theorem 1.4]{Tsu2025}.

\begin{theorem}[Sub-threshold dynamics: scattering vs. blow-up]\label{Th1}
Let $m_1, m_2, m_3 > 0$ satisfy the mass resonance condition $2m_1 + m_2 = m_3$. Consider the solution $\ve{u}(t)$ to \eqref{NLS} with initial data $\ve{u}_0 \in \dot{H}^1$. Then the following dynamics hold:
\begin{enumerate}[label=\rm{(\roman*)}]
    \item (Global existence and scattering) If $\ve{u}_0\in \dot{H}^1 $ is radially symmetric and satisfies $E(\ve{u}_0) < E(\Q)$ and $\|\nabla \ve{u}_0\|_{L^2} < \|\nabla \Q\|_{L^2}$, then $\ve{u}(t)$ exists globally in time and scatters in $\dot{H}^1$ as $t \to \pm \infty$.
    
    \item (Finite-time blow-up) If $\ve{u}_0 \in \dot{H}^1$ satisfies $E(\ve{u}_0) < E(\Q)$ and $\|\nabla \ve{u}_0\|_{L^2} > \|\nabla \Q\|_{L^2}$, and either $\ve{u}_0 \in L^{2}$  is radial or  $|x|\ve{u}_0 \in L^2$,
        then the solution $\ve{u}(t)$ blows up in finite time.
 \end{enumerate}
\end{theorem}

We recall the following Strauss lemma \cite{Strauss1977}.

\begin{lemma}\label{STRau}
There is a constant $C > 0$ such that, for any radial function $f$ in $H^1(\mathbb{R}^4)$
and any $R > 0$,
\[
\|f\|_{L^\infty_{\{|x|\geq R\}}} \leq \frac{C}{R^{\frac{3}{2}}} \|f\|_{L^2}^{\frac{1}{2}} \|\nabla f\|_{L^2}^{\frac{1}{2}}.
\]
\end{lemma}

\section{Variational Analysis}\label{S1:Vari}

Following \cite{OgaTsu2023, Tsu2025}, we say that a function $\ve{u}=(u_{1},u_{2},u_{3}):\mathbb{R}^{4}\rightarrow\mathbb{C}^{3}$ is  a ground state if it satisfies the variational problem:
\begin{equation}\label{GSta}
E(\ve{u})=\inf\left\{E(\ve{v}): \ve{v}\in\dot{H}^{1}\setminus\{0\},\;\text{and}\; \N(\ve{v})=0\right\},
\end{equation}
where $\N$ denotes the Nehari functional $\N(\ve{v}):=H(\ve{v})-4P(\ve{v})$.

We have the following Gagliardo-Nirenberg type inequality. The proof can be found in \cite[Theorem 1.3]{OgaTsu2023}.
\begin{proposition}\label{TheGN}
For any $\ve{u}\in \dot{H}^{1}$, we have
\begin{equation}\label{GNI}
\left|P(\ve{u})\right|\leq G_{S}[K(\ve{u})]^{2},
\end{equation}
where $G_{S}$ is a positive constant given by
\[
G_{S}=\sqrt[4]{\tfrac{m_{1}\sqrt{m_{2}m_{3}}}{2}}G_{4}
\]
with $G_{4}$ being the best Sobolev constant in dimension 4.
\end{proposition}
Next, we characterize the functions that satisfy the equality in \eqref{GNI}. We follow \cite[Section 3]{HajaiejStuart2004}. Suppose that 
$\left|P(\ve{u})\right|= G_{S}[K(\ve{u})]^{2}$ with $\ve{u}\neq 0$. Notice that
\begin{equation}\label{gradi}
\|\nabla |f|\|^{2}_{L^{2}}\leq \|\nabla f\|^{2}_{L^{2}} \quad \text{for } f\in \dot{H}^{1}(\R^{4}).
\end{equation}
Combining \eqref{GNI} and \eqref{gradi}, we obtain (we set $|\ve{u}|=(|u_{1}|,|u_{2}|,|u_{3}|)$)
\begin{equation}\label{ineGc}
\left|P(\ve{u})\right|\leq \left|P(|\ve{u}|)\right|\leq G_{S}[K(|\ve{u}|)]^{2}\leq G_{S}[K(\ve{u})]^{2}.
\end{equation}
We set $\varphi_{j}=|u_{j}|\geq 0$ for $j=1$, $2$, $3$. Equation \eqref{ineGc} implies that $\left|P(|\ve{u}|)\right|= G_{S}[K(|\ve{u}|)]^{2}$, and thus $(\varphi_{1}, \varphi_{2}, \varphi_{3})$ minimizes the variational problem \eqref{GSta} (see \cite[Proposition 3.2]{OgaTsu2023}). Then $(\varphi_{1}, \varphi_{2}, \varphi_{3})$ satisfies the stationary problem (the Euler-Lagrange equation):
\begin{equation}\label{s1A}
\left\{\begin{aligned}
&-\tfrac{1}{2m_{1}}\Delta \varphi_{1}=2\varphi_{1}\varphi_{2}\varphi_{3},\\
&-\tfrac{1}{2m_{2}}\Delta \varphi_{2}=\varphi_{1}^{2}\varphi_{3},\\
&-\tfrac{1}{2m_{3}}\Delta \varphi_{3}=\varphi_{1}^{2}\varphi_{2}.
\end{aligned}\right.
\end{equation}
By using the change of coefficients,
\[
W(x) =(W_{1}(x),W_{2}(x),W_{3}(x))
=\left(\sqrt[4]{4m_{2}m_{3}}\, \varphi_{1}(x),\,2\sqrt[4]{\tfrac{m_{1}^{2}m_{3}}{m_{2}}}\,\varphi_{2}(x),\,2\sqrt[4]{\tfrac{m_{1}^{2}m_{2}}{m_{3}}}\,\varphi_{3}(x)\right),
\]
 the system \eqref{s1A} can be transformed into the system:
\begin{equation}\label{NWu}
\begin{cases}
&-\Delta W_{1}=W_{1}W_{2}W_{3},\\
&-\Delta W_{2}=W_{1}^{2}W_{3},\\
&-\Delta W_{3}=W_{1}^{2}W_{2}.
\end{cases}
\end{equation}
Note that $W_{j}(x)\geq 0$ for all $x\in \R^{4}$ and for $j=1$, $2$, $3$. By standard elliptic regularity theory, it is clear that $W_{j}\in C^{2}(\R^{4})$ for $j=1$, $2$, $3$ (see e.g. \cite[Lemma 2.2]{SerraBadi}). In addition, an application of the Comparison Principle \cite[Corollary 2.8]{Lib2011} shows that $W_{i}(x)>0$ for all $x\in \R^{4}$ and for $j=1$, $2$, $3$. In \cite[Page 5]{OgaTsu2023}, it is shown that the solution to the system \eqref{NWu} with $W_{i}(x)>0$ is unique up to translation and dilation and is given by $(Q, Q, Q)$ (see the definition of $Q$ in \eqref{groundS}). In particular, we see that (up to translation and dilation)
\[
(\varphi_{1}, \varphi_{2}, \varphi_{3})
=\left(\sqrt[4]{\tfrac{1}{4m_{2}m_{3}}}Q,\tfrac{1}{2}\sqrt[4]{\tfrac{m_{2}}{m_{1}^{2}m_{3}}}Q,\tfrac{1}{2}\sqrt[4]{\tfrac{m_{3}}{m_{1}^{2}m_{2}}}Q\right).
\]
Next, from \eqref{ineGc} we get that $K(|\ve{u}|)=K(\ve{u})$. Thus, by \eqref{gradi}, we conclude 
\begin{align}\label{IdeGra}
	\|\nabla |u_{j}|\|^{2}_{L^{2}}= \|\nabla u_{j}\|^{2}_{L^{2}} \quad \text{for } j=1, 2, 3.
\end{align}
We claim that $u_{j}(x)=e^{i \theta_{j}}\varphi_{j}(x)$ with $\theta_{j}\in \R$ for $j=1$, $2$, $3$. Indeed,
we set $w(x) := \tfrac{u_{j}(x)}{\varphi_{j}(x)}$ (recall that $\varphi_{j}>0$). Since $|w|^2 = 1$, it follows that $\text{Re}(\overline{w} \nabla w) = 0$ and
\[
\nabla u_{j} = (\nabla \varphi_{j}) w + \varphi_{j} \nabla w = w (\nabla \varphi_{j} +\varphi_{j}\overline{w} \nabla w).
\]
Therefore, we infer that
\[
|\nabla u_{j}|^2 = |\nabla \varphi_{j}|^2 + \varphi_{j}^2 |\nabla w|^2.
\]
By \eqref{IdeGra} we obtain
\[
\int_{\mathbb{R}^4} \varphi_{j}^2 |\nabla w|^2 \, dx = 0.
\]
Since $\varphi_{j}>0$, we get $|\nabla w| = 0$. Thus, $w$ is constant with $|w| = 1$, and we have that there exists $\theta_{j} \in \mathbb{R}$ such that $u_{j} = e^{i\theta} \varphi_{j}(x)$. This proves the claim. 

Finally, note that $\ve{u} = (u_{1}, u_{2}, u_{3})$ also satisfies the stationary problem associated with \eqref{NLS}. Indeed, $\ve{u}$ is a minimizer of the variational problem \eqref{GSta}. Therefore, the phases $\theta_{j}$ satisfy the identity: $2\theta_{1} = \theta_{2} + \theta_{3}$ (cf. \eqref{s1A}).

We obtain the following result:
\begin{proposition}\label{UBS}
Let $\ve{u}\in \dot{H}^{1}$. Then $\ve{u}$ satisfies the equality in \eqref{GNI} if, and only if, there exist $\alpha>0$, $\lambda>0$, $x_{0}\in \R^{4}$, and $\theta_{1}$, $\theta_{2}\in \R$ such that 
\[
\ve{u}(x) = \left( \alpha \, e^{i(\theta_{1}+\theta_{2})} Q_{1}\left(\lambda^{-1}(x+x_{0})\right), \alpha\,e^{2i\theta_{1}} Q_{2}\left(\lambda^{-1}(x+x_{0})\right), \alpha \, e^{2i\theta_{2}} Q_{3}\left(\lambda^{-1}(x+x_{0})\right) \right).
\]
\end{proposition}


We need the following bubble decomposition. The proof follows the same lines as the scalar case; see \cite[Section 4.2]{KiiVisan2008} for more details.

\begin{theorem}\label{BubD}
Let $\ve{f}_n$ be a bounded radial sequence in $\dot{H}^1_x$. Then there exist $J^\ast \in \{0, 1, 2, \ldots\} \cup \{\infty\}$, 
$\left\{{\Phi}^{j}\right\}_{j=1}^{J^\ast} \subseteq \dot{H}^1_x$ and $\left\{\lambda_n^j\right\}_{j=1}^{J^\ast} \subseteq (0, \infty)$ so that along some subsequence in $n$ one may write
\[
\ve{f}_n(x) = \sum_{j=1}^J (\lambda_n^j)^{-2}{\Phi}^j \left( \tfrac{x}{\lambda_n^j} \right) + \ve{r}_n^J(x) \quad \text{for all } 0 \leq J \leq J^\ast
\]
with the following  properties:
\begin{align}\label{ArN}
&\limsup_{J \to J^\ast} \limsup_{n \to \infty} \| \ve{r}_n^J \|_{L^{4}_x} = 0,\\\label{AK}
&\sup_J \limsup_{n \to \infty} \left| K(\ve{f}_{n}) - \left( K(\ve{r}_{n})+ \sum_{j=1}^J K({\Phi}^j) \right) \right| = 0,\\\label{OGla}
&\lim_{n \to \infty}\,\, \frac{\lambda_n^j}{\lambda_n^{j'}} + \frac{\lambda_n^{j'}}{\lambda_n^j} = \infty \quad \text{for all } j \neq j'.
\end{align}
\end{theorem}

Using H\"{o}lder's inequality, \eqref{ArN} and the orthogonalization of the parameters $\lambda_n^j$ given in \eqref{OGla}, we easily deduce the following result.

\begin{corollary}
Under the conditions of Theorem~\ref{BubD}, we have that
\begin{equation}\label{AsP}
\limsup_{J \to J^\ast} \limsup_{n \to \infty} | P(\ve{f}_{n}) -\sum_{j=1}^J P({\Phi}^j)| = 0.
\end{equation}
\end{corollary}

Next we define the quantity
\begin{equation}\label{deltaK}
\delta(\ve{f}):=|K(\ve{f})-K(\Q)|.
\end{equation}

\begin{proposition}\label{Aprox}
Let $\ve{u} \in \dot{H}^1$ be radial with  $E(\ve{u}) = E(\Q)$. Then there exists a function $\varepsilon = \varepsilon(\rho)$, such that
\[
\inf_{\theta_{1}\in \mathbb{R}, \theta_{2} \in \mathbb{R}, \lambda > 0} \|\ve{u}_{[\theta_{1}, \theta_{2}, \lambda]} - \Q\|_{\dot{H}^1} \leq \varepsilon(\delta(\ve{u})), \quad \lim_{\rho \to 0} \varepsilon(\rho) = 0,
\]
where
\begin{equation}\label{DefUR}
\ve{u}_{[\theta_{1}, \theta_{2}, \lambda]}
=(
e^{i(\theta_{1}+\theta_{2})} \lambda^{-1} u_{1}(\lambda^{-1}x),
e^{2i\theta_{1}} \lambda^{-1} u_{2}(\lambda^{-1}x), 
e^{2i\theta_{2}} \lambda^{-1} u_{3}(\lambda^{-1}x)).
\end{equation}
\end{proposition}

The proof of Proposition~\ref{Aprox} is an immediate consequence of the following lemma.

\begin{lemma}\label{AproxLim23}
Let $\left\{\ve{u}^{n}\right\}^{\infty}_{n=1}$ be a sequence in $\dot{H}_{\text{rad}}^1$ such that $E(\ve{u}^{n}) = E(\Q)$. If $K(\ve{u}^{n}) \to K(\Q)$ then, up to a subsequence, there exist $\theta_1, \theta_2 \in \mathbb{R}/2\pi\mathbb{Z}$ and $\{\mu_n\} \subset (0, +\infty)$ 
such that
\begin{align}\label{CQQ}
	\ve{u}^{n}_{[\theta_{1}, \theta_{2}, \mu_{n}]}\to \Q \qtq{in $\dot{H}^1$ as $n\to +\infty$.}
\end{align}
\end{lemma}

\begin{proof}
By Theorem~\ref{BubD} we can write
\begin{align}\label{Desa1}
	\ve{u}^{n} = \sum_{j=1}^J (\lambda_n^j)^{-2}{\Phi}^j \left( \tfrac{x}{\lambda_n^j} \right) + \ve{r}_n^J(x)
\end{align}
Since $E(\ve{u}^{n}) = E(\Q)$, we get
\[
2P(\ve{u}^{n}) = K(\ve{u}^{n}) - E(\ve{u}^{n}) \to K(\Q) - E(\Q) = 2P(\Q).
\]
Therefore, by \eqref{AsP} we obtain
\[
\sum_{j=1}^{J^\ast} P({\Phi}^j) = P(\Q).
\]
Moreover, \eqref{AK} implies that
\[
\sum_{j=1}^{J^{\ast}} K(\Phi^j) \leq K(\Q).
\]
Thus, from the sharp Gagliardo-Nirenberg inequality \eqref{GNI}
\[
P(\Q) = \sum_{j=1}^{J^\ast} P(\Phi^j) \leq G_{S} \sum_{j=1}^{J^\ast} K(\Phi^j)^{2} \leq G_{s} \left[\sum_{j=1}^{J^\ast} K(\Phi^j)\right]^{2}
\leq [K(\Q)]^{2}.
\]
As $P(\Q) >0$, Proposition~\ref{UBS} implies that $J^{\ast}=1$ and $\Phi^1=\Q_{[\theta_{1}, \theta_{2}, \lambda_{0}]}$ for some 
$\theta_{1}$, $\theta_{2}$ and $\lambda_{0}>0$. On the other hand, by \eqref{AK} (recall that $K(\ve{u}^{n}) \to K(\Q)$) we see that
$K(\ve{r}^{n})\to 0$ as $n\to \infty$.
Since $\|\cdot\|_{\dot{H}^{1}}$ is equivalent to the norm induced by $K(\cdot)$, from \eqref{Desa1} we obtain \eqref{CQQ}.
This completes the proof.
\end{proof}

We observe that the following Pohozaev identity holds:  
\begin{align}\label{IP4}
	K(\Q) = 4P(\Q).  
\end{align}

We conclude this section with the following result. The proof follows from the Gagliardo-Nirenberg \eqref{GNI} inequality and proceeds along the same lines as in \cite[Claim 2.6]{DuyMerle2009}.

\begin{lemma}\label{ConvexIn}
If $\ve{f}\in \dot{H}^{1}$ and $K(\ve{f})\leq K(\Q)$, then
\[
K(\ve{f})E(\Q) \leq K(\Q)E(\ve{f}).
\]
\end{lemma}

\subsection{Virial identities}
For $R>1$, we consider the functions
\[
w_{R}(x)=R^{2}\phi\(\tfrac{x}{R}\)
\quad \text{and}\quad w_{\infty}(x)=|x|^{2},
\]
where $\phi$ is a real-valued radial function satisfying
\[
\phi(x)=
\begin{cases}
|x|^{2},& \quad |x|\leq 1,\\
0,& \quad |x|\geq 2,
\end{cases}
\quad \text{with}\quad 
|\partial^{\alpha}\phi(x)|\lesssim |x|^{2-|\alpha|}.
\]
Let $\ve{u}=(u,v,g)$ be a solution to equation \eqref{NLS}. We define the function 
\[
V(t):=\int_{\R^{4}}\left(m_{1}|u(t,x)|^{2}+m_{2}|v(t,x)|^{2}+m_{3}|g(t,x)|^{2}\right)w_{R}(x)\,dx.
\]
We also consider the localized virial functional (for $\ve{u}=(u,v,g)$)
\[
I_{R}[\ve{u}]=2\IM\int_{\R^{4}} \nabla w_{R}(x) \cdot 
\left(\nabla u(t) \overline{u(t)}+ \nabla {v(t)} \overline{v(t)}+\nabla g(t) \overline{g(t)}\right) dx.
\]

The following result will be needed; see \cite[Lemma 2.2]{OgaTsu2023} for more details.

\begin{lemma}\label{VirialIden}
Let $R\in [1, \infty]$. Assume the constants $m_1$, $m_2$ and $m_3$ satisfy the mass resonance condition $2m_1 + m_2 = m_3$.  Suppose $\ve{u}(t)$ solves \eqref{NLS}.
 Then
\begin{align}\label{F1LocalVirial}
	\tfrac{d}{d t}V(t)&=I_{R}[\ve{u}(t)],\\ \label{LocalVirial}
	\tfrac{d}{d t}I_{R}[\ve{u}]&=F_{R}[\ve{u}(t)],
\end{align}
where
\begin{align*}
F_{R}[\ve{u}]&:=\int_{\R^{4}}\left(-\tfrac{1}{4} \Delta \Delta w_{R}\right)\left(\tfrac{1}{m_{1}}|u|^{2}+\tfrac{1}{m_{2}}|v|^{2}+\tfrac{1}{m_{3}}|g|^{2}\right)\,dx\\
&-2\RE\int_{\R^{4}}\Delta[w_{R}(x)]\overline{u}(x)^{2}v(x)g(x)dx\\
&\quad +\RE\int_{\R^{4}} \left[\tfrac{1}{m_{1}}\overline{u_{j}} u_{k} +\tfrac{1}{m_{2}}\overline{v_{j}} v_{k} + \tfrac{1}{m_{3}}\overline{g_{j}} g_{k}\right]\partial_{jk}[w_{R}(x)]dx.
\end{align*}
In particular, when $R=\infty$, we obtain $F_{\infty}[\ve{u}]=4[K(\ve{u})-4P(\ve{u})]$.
\end{lemma}

Given the specifications of the weight function $w_{R}$ defined above (with $\phi(r)=\phi(|x|)$), we see that
\begin{align*}
&\RE\int_{\R^{4}} \left[\tfrac{1}{m_{1}}\overline{u_{j}} u_{k} + \tfrac{1}{m_{2}}\overline{v_{j}} v_{k} + \tfrac{1}{m_{3}}\overline{g_{j}} g_{k}\right]\partial_{jk}[w_{R}(x)]dx=\\
&\RE\int_{\R^{4}}\left[\tfrac{1}{m_{1}}|\nabla u|^{2} + \tfrac{1}{m_{2}}|\nabla v|^{2} + \tfrac{1}{m_{3}}|\nabla g|^{2}\right]\partial^{2}_{r}w_{R}dx.
\end{align*}


 As a consequence of Lemma~\ref{VirialIden}, we obtain the following results.

\begin{lemma}\label{Virialzero}
Let $R\in [1, \infty]$, $\theta\in \R$ and $\lambda>0$. Then
\[
I_{R}[\Q_{[\theta_{1}, \theta_{2}, \lambda]}]=0.
\]
\end{lemma}

\begin{lemma}\label{VirialModulate}
Let $\ve{u}$ be a solution of \eqref{NLS} defined on the interval $I$. Consider $R\in [1, \infty]$, and functions $\chi: I\to \R$, $\theta_{1}: I\to \R$, $\theta_{2}: I\to \R$, and $\lambda: I\to \R^{\ast}$. Then for all $t\in I$,
\begin{align}\nonumber
    \tfrac{d}{d t}I_{R}[\ve{u}] &= F_{\infty}[\ve{u}(t)] \\ \label{Modu11}
                           &\quad + F_{R}[\ve{u}(t)] - F_{\infty}[\ve{u}(t)] \\ \label{Modu22}
                           &\quad - \chi(t)\big\{F_{R}[\Q_{[\theta_{1}(t), \theta_{2}(t), \lambda(t)]}] - 
													F_{\infty}[\Q_{[\theta_{1}(t), \theta_{2}(t), \lambda(t)]}]\big\}.
\end{align}
\end{lemma}


\section{Linearized Equation}\label{EqLinea}

Let $\ve{u}(t)$ be a solution to \eqref{NLS}. Define $\ve{h}=(h_1, h_2, h_3)$ via
\[
\ve{h}(t,x):=\ve{u}(t,x)-\Q(x),
\]
where $\Q(x)=(Q_1, Q_2, Q_3)$ is the ground state. Recall that the functions $Q_1$, $Q_2$ and $Q_3$ are given by 
\begin{equation}\label{q1q3}
\begin{aligned}
Q_{1}(x) &= \left( \tfrac{1}{4 m_2 m_3} \right)^{\!\frac{1}{4}} Q(x), \\
Q_{2}(x) &= \tfrac{1}{2}\left( \tfrac{m_2}{m^{2}_1 m_3} \right)^{\!\frac{1}{4}} Q(x), \\
Q_{3}(x) &= \tfrac{1}{2} \left( \tfrac{m_3}{m^{2}_1 m_2} \right)^{\!\frac{1}{4}} Q(x)
\end{aligned}
\end{equation}
with $Q(x)= \tfrac{1}{(1 + |x|^2 / 8)}$. Note that since $\ve{u}(t)$ is a solution to \eqref{NLS} and $\Q$ satisfies the elliptic equation \eqref{s1A}, we have that $\ve{h}$ satisfies the nonlinear Schr\"odinger equation

\[
\begin{cases}
i \partial_t h_1 + \frac{1}{2m_1} \Delta h_1 + N_{1}(\ve{h}) = 0, \\
i \partial_t h_2 + \frac{1}{2m_2} \Delta h_2 + N_{2}(\ve{h}) = 0, \\
i \partial_t h_3 + \frac{1}{2m_3} \Delta h_3 + N_{3}(\ve{h}) = 0,
\end{cases}
\]
where
\begin{align*}
N_{1}(\ve{h}) &:= 2\overline{h}_1 h_2 h_3 + 2\overline{h}_1 h_2 Q_3 + 2\overline{h}_1 Q_2 h_3 + 2\overline{h}_1 Q_2 Q_3 \\
&\quad + 2\overline{Q}_1 h_2 h_3 + 2\overline{Q}_1 h_2 Q_3 + 2\overline{Q}_1 Q_2 h_3,\\
N_{2}(\ve{h}) &:= h_1^2 \overline{h}_3 + h_1^2 \overline{Q}_3 + 2h_1 Q_1 \overline{h}_3 + 2h_1 Q_1 \overline{Q}_3 + Q_1^2 \overline{h}_3, \\
N_{3}(\ve{h}) &:= h_1^2 \overline{h}_2 + h_1^2 \overline{Q}_2 + 2h_1 Q_1 \overline{h}_2 + 2h_1 Q_1 \overline{Q}_2 + Q_1^2 \overline{h}_2.
\end{align*}

Equivalently, $\ve{h}$ satisfies the equation
\begin{equation}\label{Decomh}
\partial_{t}\ve{h} + \mathcal{L} \ve{h} = iR\ve{h}, \quad \text{where} \quad
\mathcal{L} := \begin{pmatrix}
0 & -L_{I} \\
L_{R} & 0
\end{pmatrix},
\end{equation}
with
\[
R(\ve{h}) = \begin{pmatrix}
2\overline{h}_1 h_2 h_3 + 2\overline{h}_1 h_2 Q_3 + 2\overline{h}_1 Q_2 h_3 + 2\overline{Q}_1 h_2 h_3, \\
h_1^2 \overline{h}_3 + h_1^2 \overline{Q}_3 + 2h_1 Q_1 \overline{h}_3, \\
h_1^2 \overline{h}_2 + h_1^2 \overline{Q}_2 + 2h_1 Q_1 \overline{h}_2
\end{pmatrix}.
\]

Furthermore, the operators $L_{I}$ and $L_{R}$ are given by
\[
L_{R} := 
\begin{pmatrix}
-\tfrac{1}{2m_1} \Delta & 0 & 0 \\
0 & -\tfrac{1}{2m_2} \Delta & 0 \\
0 & 0 & -\tfrac{1}{2m_3} \Delta
\end{pmatrix}
+
\begin{pmatrix}
-2 Q_2 Q_3 & -2Q_1 Q_3 & -2Q_1 Q_2 \\
-2 Q_1 Q_3 & 0 & -Q_1^2 \\
-2 Q_1 Q_2 & -Q_1^2 & 0
\end{pmatrix}
\]
and
\[
L_{I} := 
\begin{pmatrix}
-\tfrac{1}{2m_1} \Delta & 0 & 0 \\
0 & -\tfrac{1}{2m_2} \Delta & 0 \\
0 & 0 & -\tfrac{1}{2m_3} \Delta
\end{pmatrix}
+
\begin{pmatrix}
2 Q_2 Q_3 & -2Q_1 Q_3 & -2Q_1 Q_2 \\
-2 Q_1 Q_3 & 0 & Q_1^2 \\
-2 Q_1 Q_2 & Q_1^2 & 0
\end{pmatrix}.
\]

Notice also that we can write \eqref{Decomh} as a Schr\"odinger equation (recall that $\ve{h}=(h_{1}, h_{2}, h_{3})$):
\[
(i\partial_{t}h_{1}, i\partial_{t}h_{2}, i\partial_{t}h_{3}) + \left( \tfrac{1}{2m_1} \Delta h_1, \tfrac{1}{2m_2} \Delta h_2, \tfrac{1}{2m_3} \Delta h_3 \right) + K(\ve{h}) = -R(\ve{h}),
\]
where
\[
K(\ve{h}) = \begin{pmatrix}
2\overline{h}_1 Q_2 Q_3 + 2\overline{Q}_1 h_2 Q_3 + 2\overline{Q}_1 Q_2 h_3, \\
2h_1 Q_1 \overline{Q}_3 + Q_1^2 \overline{h}_3, \\
2h_1 Q_1 \overline{Q}_2 + Q_1^2 \overline{h}_2
\end{pmatrix}.
\]

Substituting the functions $Q_1$, $Q_2$, and $Q_3$ (cf.~\eqref{q1q3}) into the operators $L_{R}$ and $L_{I}$ and simplifying, we obtain
\[
L_{I} = 
\begin{pmatrix} 
 -\dfrac{1}{2m_1} \Delta & 0 & 0 \\ 
 0 & -\dfrac{1}{2m_2} \Delta & 0 \\ 
 0 & 0 & -\dfrac{1}{2m_3} \Delta 
\end{pmatrix} 
+ 
\begin{pmatrix} 
 \dfrac{1}{2 m_1} & -\dfrac{1}{\sqrt[4]{4 m_1^{2} m_2^{2}}} & -\dfrac{1}{\sqrt[4]{4 m_1^{2} m_3^{2}}} \\ 
 -\dfrac{1}{\sqrt[4]{4 m_1^{2} m_2^{2}}} & 0 & \dfrac{1}{\sqrt{4 m_2 m_3}} \\ 
 -\dfrac{1}{\sqrt[4]{4 m_1^{2} m_3^{2}}} & \dfrac{1}{\sqrt{4 m_2 m_3}} & 0 
\end{pmatrix} Q^2.
\]
and
\[
L_{R} = 
\begin{pmatrix} 
 -\dfrac{1}{2m_1} \Delta & 0 & 0 \\ 
 0 & -\dfrac{1}{2m_2} \Delta & 0 \\ 
 0 & 0 & -\dfrac{1}{2m_3} \Delta 
\end{pmatrix} 
+ 
\begin{pmatrix} 
 -\dfrac{1}{2 m_1} & -\dfrac{1}{\sqrt[4]{4 m_1^{2} m_2^{2}}} & -\dfrac{1}{\sqrt[4]{4 m_1^{2} m_3^{2}}} \\ 
 -\dfrac{1}{\sqrt[4]{4 m_1^{2} m_2^{2}}} & 0 & -\dfrac{1}{\sqrt{4 m_2 m_3}} \\ 
 -\dfrac{1}{\sqrt[4]{4 m_1^{2} m_3^{2}}} & -\dfrac{1}{\sqrt{4 m_2 m_3}} & 0 
\end{pmatrix} Q^2.
\]
Next, we will study the coercivity of the operators $L_{R}$ and $L_{I}$. For the following results, we introduce:
\begin{equation*}
\Phi_{1} = \left(\tfrac{Q}{\sqrt{3m_1}}, \tfrac{-Q}{\sqrt{12m_2}}, \tfrac{Q}{\sqrt{12m_3}}\right) \quad \text{and} \quad
\Phi_{2} = \left(\tfrac{Q}{\sqrt{12m_1}}, \tfrac{Q}{\sqrt{12m_2}}, \tfrac{-Q}{\sqrt{12m_3}}\right).
\end{equation*}

\begin{lemma}\label{Apx22}
There exists $C>0$, depending on $m_1$, $m_2$, $m_3$, and the best Sobolev constant in dimension~4, such that for every $\ve{v}\in (\dot{H}^{1}(\R^{4}: \R))^{3}$ satisfying
\begin{align}\label{Ort223}
    (\ve{v}, \Phi_{1})_{\dot{H}^{1}} = (\ve{v}, \Phi_{2})_{\dot{H}^{1}} = 0, 
\end{align}
then we have
\[
\langle L_{I}\ve{v},\ve{v}\rangle \geq C \|\ve{v}\|^{2}_{\dot{H}^{1}}.
\]
\end{lemma}

\begin{proof}
Consider the operator $\mathcal{A}$ given by
\[
\mathcal{A} = \begin{pmatrix}
-\Delta & 0 & 0 \\
0 & -\Delta & 0 \\
0 & 0 & -\Delta
\end{pmatrix}
+\begin{pmatrix}
1 & -\sqrt{2} & -\sqrt{2} \\
-\sqrt{2} & 0 & 1 \\
-\sqrt{2} & 1 & 0
\end{pmatrix}Q^2.
\]
For $\gamma\in \R$ we define $L_{\gamma}v = -\Delta v - \gamma Q^{2}v$ for $v\in \dot{H}^{1}(\R^{4})$. Then $\mathcal{A}$ can be diagonalized as follows:
\[
\mathcal{A} = P\begin{pmatrix}
L_{1} & 0 & 0 \\
0 & L_{1} & 0 \\
0 & 0 & L_{-3}
\end{pmatrix}P^{\ast}, \quad \text{where} \quad
P = \begin{pmatrix}
\frac{1}{\sqrt{3}} & \frac{1}{\sqrt{3}} & \frac{1}{\sqrt{5}} \\
\frac{\sqrt{2}}{\sqrt{3}} & 0 & \frac{-\sqrt{2}}{\sqrt{5}} \\
0 & \frac{\sqrt{2}}{\sqrt{3}} & \frac{-\sqrt{2}}{\sqrt{5}}
\end{pmatrix}.
\]
Notice that 
\[
\langle \mathcal{A}\ve{v},\ve{v}\rangle = \langle L_{1}w_{1}, w_{1} \rangle + \langle L_{1}w_{2}, w_{2} \rangle + \langle L_{-3}w_{3}, w_{3} \rangle,
\]
where $\ve{w}=P^{\ast}\ve{v}$. Now, we define the transformation 
\[
\Gamma(\ve{v}) = \Gamma(u, v, g) := (\sqrt{2m_1}u, \sqrt{2m_2}v, \sqrt{2m_3}g).
\]
Then, it is easy to verify that 
\begin{align*}
\langle L_{I}\ve{v}, \ve{v} \rangle &= \langle L_{I}\Gamma(\Gamma^{-1}\ve{v}), \Gamma(\Gamma^{-1}\ve{v}) \rangle \\
&= \langle \mathcal{A}\Gamma^{-1}\ve{v}, \Gamma^{-1}\ve{v} \rangle \\
&= \langle L_{1}\tilde{w}_{1}, \tilde{w}_{1} \rangle + \langle L_{1}\tilde{w}_{2}, \tilde{w}_{2} \rangle + \langle L_{-3}\tilde{w}_{3}, \tilde{w}_{3} \rangle,
\end{align*}
where $(\tilde{w}_{1}, \tilde{w}_{2}, \tilde{w}_{3}) = P^{\ast}\Gamma^{-1}\ve{v}$.

Since (cf. \eqref{Ort223})
\[
(\tilde{w}_{1}, Q)_{\dot{H}^{1}} = (\tilde{w}_{2}, Q)_{\dot{H}^{1}} = 0,
\]
\cite[Claim 3.5]{DuyMerle2009} implies that there exists a constant $C_{1}$, depending on the Sobolev constant in dimension~4,  such that
\[
\langle L_{1}\tilde{w}_{1}, \tilde{w}_{1} \rangle + \langle L_{1}\tilde{w}_{2}, \tilde{w}_{2} \rangle + \langle L_{-3}w_{3}, w_{3} \rangle\geq C_{1}\left[\|\tilde{w}_{1}\|^{2}_{\dot{H}^{1}} + \|\tilde{w}_{2}\|^{2}_{\dot{H}^{1}}+ \|\tilde{w}_{3}\|^{2}_{\dot{H}^{1}}\right].
\]
Thus, if $\ve{v}\neq 0$ and satisfies \eqref{Ort223}, we have
\[
\langle L_{I}\ve{v}, \ve{v}\rangle \geq C_{1} \|P^{\ast}\Gamma^{-1}\ve{v}\|^{2}_{\dot{H}^{1}} = \|\Gamma^{-1}\ve{v}\|^{2}_{\dot{H}^{1}} \geq C_{2}\|\ve{v}\|^{2}_{\dot{H}^{1}},
\]
where $C_2$ depends on $m_1$, $m_2$, $m_3$, and the Sobolev constant in dimension~4.
This completes the proof of the lemma.
\end{proof}

Before stating the next result, we define the following vectors:
\begin{align*}
\Pi_{1} &= \left( \tfrac{Q}{2\sqrt{m_1}}, \tfrac{Q}{2\sqrt{2m_2}}, \tfrac{Q}{2\sqrt{2m_3}} \right), \quad 
\Pi_{2} = \left( \tfrac{\Lambda Q}{2\sqrt{m_1}}, \tfrac{\Lambda Q}{2\sqrt{2m_2}}, \tfrac{\Lambda Q}{2\sqrt{2m_3}} \right), \\
\Psi_{j} &= \left( \tfrac{\partial_{j}Q}{2\sqrt{m_1}}, \tfrac{\partial_{j}Q}{2\sqrt{2m_2}}, \tfrac{\partial_{j}Q}{2\sqrt{2m_3}} \right)
\quad \text{for } 1 \leq j \leq 4,
\end{align*}
where $\Lambda Q$ denotes the scaling derivative $\Lambda Q = 2Q + x\cdot\nabla Q$  and $\partial_j Q$ are the spatial derivatives.

\begin{lemma}\label{Apx11}
There exists $C>0$,   depending on $m_1$, $m_2$, $m_3$, and the best Sobolev constant in dimension~4, such that for every $\ve{v}\in (\dot{H}^{1}(\R^{4}: \R))^{3}$ satisfying
\begin{align}\label{Ort11}
    (\ve{v}, \Pi_{1})_{\dot{H}^{1}} = (\ve{v}, \Pi_{2})_{\dot{H}^{1}} = (\ve{v}, \Psi_{j})_{\dot{H}^{1}} = 0
\end{align}
for $1\leq j\leq 4$, then we have
\[
\langle L_{R}\ve{v},\ve{v}\rangle \geq C \|\ve{v}\|^{2}_{\dot{H}^{1}}.
\]
\end{lemma}

\begin{proof}
The proof follows similar arguments to Lemma~\ref{Apx22}. Consider the operator $\mathcal{B}$ given by 
\[
\mathcal{B} = \begin{pmatrix}
-\Delta & 0 & 0 \\
0 & -\Delta & 0 \\
0 & 0 & -\Delta
\end{pmatrix}
+ \begin{pmatrix}
-1 & -\sqrt{2} & -\sqrt{2} \\
-\sqrt{2} & 0 & -1 \\
-\sqrt{2} & -1 & 0
\end{pmatrix}Q^2.
\]
Note that $\mathcal{B}$ can be diagonalized as follows:
\[
\mathcal{B} = C\begin{pmatrix}
L_{-1} & 0 & 0 \\
0 & L_{-1} & 0 \\
0 & 0 & L_{3}
\end{pmatrix}
C^{\ast}, \quad \text{where} \quad
C = \begin{pmatrix}
-\frac{\sqrt{10}}{5} & -\frac{2\sqrt{15}}{15} & \frac{\sqrt{2}}{2} \\
\frac{2\sqrt{5}}{5} & -\frac{\sqrt{30}}{15} & \frac{1}{2} \\
-\frac{2\sqrt{5}}{5} & \frac{\sqrt{30}}{15} & \frac{1}{2}
\end{pmatrix}.
\]
Observe that
\[
\langle \mathcal{B}\ve{v},\ve{v}\rangle = \langle L_{-1}w_{1}, w_{1} \rangle + \langle L_{-1}w_{2}, w_{2} \rangle + \langle L_{3}w_{3}, w_{3} \rangle,
\]
where $\ve{w}=C^{\ast}\ve{v}$. Defining the transformation 
\[
\Gamma(\ve{v}) = \Gamma(u, v, g) := (\sqrt{2m_1}u, \sqrt{2m_2}v, \sqrt{2m_3}g),
\]
we obtain
\begin{align*}
\langle L_{R}\ve{v}, \ve{v} \rangle &= \langle L_{R}\Gamma(\Gamma^{-1}\ve{v}), \Gamma(\Gamma^{-1}\ve{v}) \rangle \\
&= \langle \mathcal{B}\Gamma^{-1}\ve{v}, \Gamma^{-1}\ve{v} \rangle \\
&= \langle L_{-1}\tilde{w}_{1}, \tilde{w}_{1} \rangle + \langle L_{-1}\tilde{w}_{2}, \tilde{w}_{2} \rangle + \langle L_{3}\tilde{w}_{3}, \tilde{w}_{3} \rangle.
\end{align*}
From \eqref{Ort11} we deduce that
\[
(\tilde{w}_{3}, Q)_{\dot{H}^{1}} = (\tilde{w}_{3}, \Lambda Q)_{\dot{H}^{1}} = (\tilde{w}_{3}, \partial_{j}Q)_{\dot{H}^{1}} = 0,
\]
for $1\leq j\leq 4$. Therefore, there exists $C_3>0$ such that (see \cite[Lemma 3.5]{CamposFarahRoudenko})
\[
\langle L_{3}\tilde{w}_{3}, \tilde{w}_{3} \rangle \geq C_{3}\|\tilde{w}_{3}\|^{2}_{\dot{H}^{1}}.
\]
Consequently, if $\ve{v}\neq 0$ and satisfies \eqref{Ort11}, we have
\[
\langle L_{R}\ve{v}, \ve{v}\rangle \gtrsim \|C^{\ast}\Gamma^{-1}\ve{v}\|^{2}_{\dot{H}^{1}} = \|\Gamma^{-1}\ve{v}\|^{2}_{\dot{H}^{1}} \gtrsim \|\ve{v}\|^{2}_{\dot{H}^{1}},
\]
which completes the proof.
\end{proof}

We denote by $\mathcal{F}(\ve{u}, \ve{v})$ the bilinear symmetric form
\begin{equation}\label{Quadratic}
\mathcal{F}(\ve{u}, \ve{v}) := \tfrac{1}{2}\langle L_{R} \RE\ve{u}, \RE\ve{v}\rangle + \tfrac{1}{2}\langle L_{I}\IM\ve{u}, \IM\ve{v}\rangle,
\end{equation}
and we write $\mathcal{F}(\ve{u}) := \mathcal{F}(\ve{u},\ve{u})$.

In what follows, we consider the functions 
\begin{equation}
\begin{aligned}
\Q_{p} &:= (Q_{1}, 2Q_{2}, 0), \\
\Q_{q} &:= (2Q_{1}, -Q_{2}, 5Q_{3}), \\
\partial_{j} \Q &:= (\partial_{j} Q_{1}, \partial_{j} Q_{2}, \partial_{j} Q_{3}) \quad \text{for } j=1,\ldots,4, \\
\Lambda \Q &:= (\Lambda Q_{1}, \Lambda Q_{2}, \Lambda Q_{3}) \quad \text{with } \Lambda Q_{j} = 2Q_{j} + x\cdot\nabla Q_{j} \in L^{2}.
\end{aligned}
\label{DEFQt}
\end{equation}
Note that $\Q \in \text{Span}\left\{\Q_{p}, \Q_{q}\right\}$ and $(\Q_{p}, \Q_{q})_{\dot{H}^{1}}=0$. By direct calculation, we obtain:
\[
L_{R}(\partial_{j} \Q) = 0, \quad L_{R}(\Lambda\Q) = 0,
\]
and
\[
L_{I}(\Q_{p}) = 0, \quad L_{I}(\Q_{q}) = 0.
\]
In particular, we have that
\[
\mathcal{L} (\partial_{j}\Q) = \mathcal{L} (\Lambda \Q) = \mathcal{L} (i\Q_{p}) = \mathcal{L} (i\Q_{q}) = 0.
\]

\begin{lemma}\label{CoerLi11}
There exists a positive constant $C$,  depending on $m_1$, $m_2$, $m_3$, and the best Sobolev constant in dimension~4,  such that for every $\ve{v}\in (\dot{H}^{1}(\R^{4}: \R))^{3}$ satisfying
\begin{align}\label{SOrt1111}
(\ve{v}, \Q_{p})_{\dot{H}^{1}} = (\ve{v}, \Q_{q})_{\dot{H}^{1}} = 0,
\end{align}
then we have
\begin{align}\label{Fine22}
\langle L_{I}\ve{v},\ve{v}\rangle \geq C \|\ve{v}\|^{2}_{\dot{H}^{1}}.
\end{align}
\end{lemma}

\begin{proof}
It suffices to show that for all $\ve{v}\in \dot{H}^{1}$ satisfying \eqref{SOrt1111} we have 
$\mathcal{F}_{-}(\ve{v}) := \langle L_{I}\ve{v},\ve{v}\rangle > 0$. Indeed, since the quadratic form $\mathcal{F}_{-}(\cdot)$ is a compact perturbation of $K(\ve{u})$, 
if $\mathcal{F}_{-}(\ve{v})>0$, a standard argument shows \eqref{Fine22}.

Suppose by contradiction that there exists $\ve{g}\in \dot{H}^{1}\setminus\left\{0\right\}$ such that
\begin{align}\label{contra1}
(\ve{g}, \Q_{p})_{\dot{H}^{1}} = (\ve{g}, \Q_{q})_{\dot{H}^{1}} = 0,
\end{align}
and $\mathcal{F}_{-}(\ve{g}) \leq 0$. Recall that $L_I(\Q_{q}) = L_I(\Q_{p}) = 0$. Since $\mathcal{F}_{-}(\Q_{q}, \ve{h}) = \mathcal{F}_{-}(\Q_{p}, \ve{h}) = 0$ for all $\ve{h}\in \dot{H}^{1}$, we see that
$\mathcal{F}_{-}(\ve{v}) \leq 0$ for $\ve{v}\in E$, where $E = \text{Span}\left\{\ve{g}, \Q_{q}, \Q_{p}\right\}$. Moreover, by \eqref{contra1} we infer that $E$ is a subspace of dimension $3$, which contradicts Lemma~\ref{Apx22}. 
\end{proof}

\begin{remark}\label{KerLI}
As $L_I \Q_{p}=L_I \Q_{q}=0$, from Lemma~\ref{CoerLi11} we get  $L_{I}\geq 0$ and $\text{Ker}(L_I)=\text{span}\left\{\Q_{p}, \Q_{q}\right\}$. 
\end{remark}

\begin{lemma}\label{Coer11}
There exists a positive constant $C>0$,  depending on $m_1$, $m_2$, $m_3$, and the best Sobolev constant in dimension~4,  such that for every $\ve{v}\in (\dot{H}^{1}(\R^{4}: \R))^{3}$ satisfying
\begin{align}\label{SOrt11}
    \mathcal{F}(\ve{v}, \Q) = (\ve{v}, \Lambda \Q)_{\dot{H}^{1}} = (\ve{v}, \partial_{j}\Q)_{\dot{H}^{1}} = 0
\end{align}
for $1\leq j\leq 4$, then we have
\begin{align}\label{Fine}
    \langle L_{R}\ve{v},\ve{v}\rangle \geq C \|\ve{v}\|^{2}_{\dot{H}^{1}}.
\end{align}
\end{lemma}

\begin{proof}
Following the approach of Lemma~\ref{CoerLi11}, we show that if $\ve{v}$ satisfies \eqref{SOrt11}, then $\mathcal{F}(\ve{v}) := \tfrac{1}{2}\langle L_{R}\ve{v},\ve{v}\rangle > 0$. 

Suppose by contradiction that there exists $\ve{g}\in \dot{H}^{1}\setminus\{0\}$ such that
\begin{align}\label{Ort22}
    \mathcal{F}(\ve{g}, \Q) = (\ve{g}, \Lambda \Q)_{\dot{H}^{1}} = (\ve{g}, \partial_{j}\Q)_{\dot{H}^{1}} = 0,
\end{align}
and $\mathcal{F}(\ve{g}) \leq 0$. Since $\mathcal{F}(\Q) < 0$, it is straightforward to show that 
\[
E = \text{span}\left\{\ve{g}, \Q, \Lambda \Q, \partial_{1}\Q, \partial_{2}\Q, \partial_{3}\Q, \partial_{4}\Q    \right\}
\] 
is a subspace of dimension $7$ where $\mathcal{F}(\ve{u}) \leq 0$ for all $\ve{u}\in E$. However, Lemma~\ref{Apx11} establishes that $\mathcal{F}(\ve{u}) = \tfrac{1}{2}\langle L_{R}\ve{u},\ve{u}\rangle$ is positive definite on a subspace of co-dimension $6$, leading to a contradiction.
\end{proof}

By Lemmas~\ref{Coer11} and \ref{CoerLi11}, we get the following proposition.

\begin{proposition}\label{CoerQuadra}
There exists  a positive constant $C>0$,  depending on $m_1$, $m_2$, $m_3$, and the best Sobolev constant in dimension~4,  such that for every $\ve{h}\in G^{\bot}$, we have
\[
\F(\ve{h})\geq C \|\ve{h}\|^{2}_{\dot{H}^{1}},
\]
where
\begin{align*}
G^{\bot} := \big\{ \ve{h} \in \dot{H}^{1} \, \big| \, &\F(\Q, \ve{h}) = (i\Q_{p}, \ve{h})_{\dot{H}^{1}} = (i\Q_{q}, \ve{h})_{\dot{H}^{1}} \\
&= (\Lambda \Q, \ve{h})_{\dot{H}^{1}} = (\partial_{j} \Q, \ve{h})_{\dot{H}^{1}} = 0 : j = 1, \ldots, 4\big\}.
\end{align*}
\end{proposition}

\section{Modulation analysis}\label{S:Modula}

We recall the quantity (cf. \eqref{deltaK})
\[
\delta(\mathbf{f}) := |K(\mathbf{f}) - K({\Q})|.
\]
Consider a radial solution $\ve{u}(t)$ to \eqref{NLS} with initial data $\ve{u}_{0}$ in $\dot{H}^{1}$ satisfying
\[
E(\ve{u}) = E(\Q),
\]
and define the quantity 
\[
\delta(t) :=\delta(\mathbf{u}(t)) = \left| K(\ve{u}(t))- K(\Q) \right|.
\]
Let $\delta_{0} > 0$ be a small parameter, and define the open set
\[
I_{0} = \left\{ t \in [0, \infty) : \delta(t) < \delta_{0} \right\}.
\]
We now state and prove the following proposition.

\begin{proposition}\label{ModilationFree}
For $\delta_{0} > 0$ sufficiently small, there exist functions 
\[
\eta: I_{0} \to \R, \quad \theta: I_{0} \to \R, \quad \mu: I_{0} \to \R^{\ast}, \quad \alpha: I_{0} \to \R, \quad \text{and} \quad \ve{h}: I_0 \to \dot{H}^1
\]
such that, for all $t \in I_0$, the radial solution $\ve{u}$ can be decomposed as
\begin{equation}\label{DecomUFree}
\ve{u}_{[\eta(t), \theta(t), \mu(t)]}(t) = (1 + \alpha(t))\Q + \ve{h}(t),
\end{equation}
where the following estimates hold:
\begin{equation}\label{EstimateOne}
|\alpha(t)| \sim \|\ve{h}(t)\|_{\dot{H}^{1}} \sim \delta(t),
\end{equation}
and
\begin{equation}\label{EstimateFree}
|\eta^{\prime}(t)| + |\theta^{\prime}(t)| + |\alpha^{\prime}(t)| + \frac{|\mu^{\prime}(t)|}{|\mu(t)|} \lesssim \mu^{2}(t)\delta(t).
\end{equation}
\end{proposition}

For the proof of the proposition, we need the following result:


\begin{lemma}\label{ExistenceFree}
There exists $\delta_0 > 0$ such that for all radial $\mathbf{u}$ in $\dot{H}^1$ 
satisfying $E(\mathbf{u}) = E(\Q)$ and $\delta(\mathbf{u})< \delta_0$, there exist 
$(\eta, \theta, \mu) \in \mathbb{R} \times \mathbb{R} \times (0, +\infty)$ with
\[
\mathbf{u}_{[\eta, \theta, \mu]} \perp i\Q_{p}, \quad \mathbf{u}_{[\eta, \theta, \mu]} \perp i{\Q}_{q}, \quad
\mathbf{u}_{[\eta, \theta, \mu]} \perp \Lambda\Q,
\]
where $\Q_{p} = (Q_{1}, 2Q_{2}, 0)$, $\Q_{q} = (2Q_{1}, -Q_{2}, 5Q_{3})$, and 
$\Lambda\Q = (\Lambda Q_{1}, \Lambda Q_{2}, \Lambda Q_{3})$ with 
$\Lambda Q_{j} = 2Q_{j} + x \cdot \nabla Q_{j}$ (cf. \eqref{DEFQt}). 
The parameters $(\eta, \theta, \mu)$ are unique in 
$\mathbb{R}/2\pi\mathbb{Z} \times \mathbb{R}/2\pi\mathbb{Z} \times \mathbb{R}_+$, 
and the mapping $\mathbf{u} \mapsto (\eta, \theta, \mu)$ is $C^1$.
\end{lemma}

\begin{proof}
By Proposition~\ref{Aprox}, we can choose ${\eta_{1}}$, ${\theta_{1}}$, and $\mu_{1}$ such that
\begin{align}\label{aprx1}
	\mathbf{u}_{[{\eta_{1}}, {\theta_{1}}, \mu_{1}]} = \Q + g \quad \text{with} \quad \|g\|_{\dot{H}^{1}} \leq \epsilon(\delta(\mathbf{u})),
\end{align}
for $\delta(\mathbf{u})$ sufficiently small. Now, consider the functional
\begin{align*}
J(\eta, \theta, \mu, \mathbf{u}) &= (J_1(\eta, \theta, \mu, \mathbf{u}), 
J_2(\eta, \theta, \mu, \mathbf{u}), J_3(\eta, \theta, \mu, \mathbf{u})) \\
&= \left( (\mathbf{u}_{[\eta, \theta, \mu]}, i\Q_{p})_{\dot{H}^1}, 
(\mathbf{u}_{[\eta, \theta, \mu]}, i{\Q}_{q})_{\dot{H}^1}, 
(\mathbf{u}_{[\eta, \theta, \mu]}, \Lambda{\Q})_{\dot{H}^1} \right).
\end{align*}
Let
\[
H(\theta_{1}, \theta_{2} , \mu, \mathbf{u}):=J(\theta_{1}-\tfrac{1}{2}\theta_{2}, \tfrac{5}{2}\theta_{2}, \mu, \mathbf{u}).
\]
Since $\Q$, $\Q_p$, $\Q_q$ are real-valued and $(\Q, \Lambda\Q)_{\dot{H}^1}=0$, we have
\[
	H(0, 0, 1, \Q) = J(0, 0, 1, \Q)= \left( 
	(\Q, i\Q_p)_{\dot{H}^1}, 
	(\Q, i\Q_q)_{\dot{H}^1},
	(\Q, \Lambda\Q)_{\dot{H}^1}
	\right)=(0,0,0).
\]
On the other hand, a direct calculation shows that 
\begin{align*}
\frac{\partial H}{\partial (\theta_{1}, \theta_{2} , \mu)} (0, 0, 1, \Q)=
 \begin{bmatrix}
(i\Q_{p}, i\Q_p)_{\dot{H}^1} & (i\Q_{q}, i\Q_p)_{\dot{H}^1} & -(\Lambda \Q, i\Q_p)_{\dot{H}^1} \\
(i\Q_{p}, i\Q_q)_{\dot{H}^1} & (i\Q_{q}, i\Q_q)_{\dot{H}^1} & -(\Lambda \Q, i\Q_q)_{\dot{H}^1} \\
(i\Q_{p}, \Lambda \Q)_{\dot{H}^1} & (i\Q_{q}, \Lambda \Q)_{\dot{H}^1} & -(\Lambda \Q, \Lambda \Q)_{\dot{H}^1}
\end{bmatrix}.
\end{align*}
Therefore, using $(\Q_{q}, \Q_p)_{\dot{H}^1}=0$, we see that 
\[
\left|\det\left(\frac{\partial H}{\partial (\theta_{1}, \theta_{2} , \mu)} (0, 0, 1, \Q) \right)\right| = 
\|\nabla \Q_{p}\|^{2}_{L^{2}}\|\nabla \Q_{q}\|^{2}_{L^{2}}\|\nabla \Lambda\Q\|^{2}_{L^{2}} \neq 0.
\]

Hence, by the implicit function theorem, there exist $\epsilon_0, \gamma_0 > 0$ such that for any $\mathbf{h} \in \dot{H}^1$ with $\|\mathbf{h} - \Q\|_{\dot{H}^1} < \epsilon_0$, there exists a unique $(\tilde{\theta_{1}}(\mathbf{h}), \tilde{\theta_{2}}(\mathbf{h}), \tilde{\mu}(\mathbf{h}))$ (a $C^1$ function of $\mathbf{h}$) satisfying $|\tilde{\theta_{1}}|+|\tilde{\theta_{2}}| + |\tilde{\mu} - 1|\ll  \gamma_0$ and
\[
H(\tilde{\theta_{1}}, \tilde{\theta_{2}}, \tilde{\mu}, \mathbf{h}) = J(\tilde{\theta_{1}}-\tfrac{1}{2}\tilde{\theta_{2}}, \tfrac{5}{2}\tilde{\theta_{2}},\tilde{\mu}, \mathbf{h}) = 0.
\]
Defining $\eta_{0}=\tilde{\theta_{1}}-\tfrac{1}{2}\tilde{\theta_{2}}$ and $\theta_{0}=\tfrac{5}{2}\tilde{\theta_{2}}$, we obtain $|{\eta_{0}}|+|{\theta_{0}}| + |\tilde{\mu} - 1| < \gamma_0$ and
\[
J(\eta_{0}, \theta_{0},\tilde{\mu}, \mathbf{h}) = 0.
\]
Thus, from \eqref{aprx1}, we find that there exists a unique $(\tilde{\eta}_0, \tilde{\theta}_0, \tilde{\mu}_0)$  such that
\[
J(\tilde{\eta}_0, \tilde{\theta}_0, \tilde{\mu}_0, \mathbf{u}_{[{\eta_{1}}, {\theta_{1}}, \mu_{1}]}) = 0.
\]
Using the group properties of the transformation $\mathbf{u} \mapsto \mathbf{u}_{[\eta, \theta, \mu]}$, this is equivalent to
\[
J(\tilde{\eta}_0+\eta_{1}, \tilde{\theta}_0+\theta_{1}, \tilde{\mu}_0\mu_1, \mathbf{u}) = 0.
\]
This completes the proof by taking the final parameters to be $\eta = \tilde{\eta}_0 + \eta_1$, $\theta = \tilde{\theta}_0 + \theta_1$, and $\mu = \tilde{\mu}_0\mu_1$.
\end{proof}


Let $u$ be a radial solution to \eqref{NLS} and $I_{0}$ be a time interval such that
\[
\delta(t)=\delta(\mathbf{u}(t))< \delta_0 \text{ for all } t \in I_{0},
\]
where $\delta_{0}$ is given by the previous lemma. For each $t \in I_{0}$, we choose the parameters $(\eta(t), \theta(t), \mu(t))$ according to Lemma~\ref{ExistenceFree}, and we express the solution $\ve{u}$ in the form
\begin{equation}\label{DefH}
\ve{u}_{[\eta(t), \theta(t), \mu(t)]}(t) = (1 + \alpha(t))\Q + \ve{h}(t) \quad \text{for all} \quad t \in I_0,
\end{equation}
where the modulation parameter $\alpha(t)$ is given by (cf. \eqref{NegativeLQ})
\[
\alpha(t) + 1 = \tfrac{1}{\F(\Q, \Q)} \F(\Q, \ve{u}_{[\eta(t), \theta(t), \mu(t)]}).
\]
The function $\ve{h}(t)$ satisfies the following orthogonality conditions:
\begin{equation}\label{h: Cond}
\ve{h} \perp \text{span}\left\{\nabla \Q, i\Q_{p}, i\Q_{q}, \Lambda\Q \right\} \quad \text{and} \quad \F(\Q, \ve{h}) = 0.
\end{equation}
Observe that the linearized operator $L_{R}$ applied to $\Q$ yields
\[
L_{R}(\Q) = \left(\tfrac{1}{m_{1}}\Delta Q_{1}, \tfrac{1}{m_{2}}\Delta Q_{2}, \tfrac{1}{m_{3}}\Delta Q_{3}\right).
\]
Consequently, from the orthogonality conditions in \eqref{h: Cond}, we deduce that
\begin{align}\label{OrQh}
	\left(\big(\tfrac{1}{m_{1}}Q_{1}, \tfrac{1}{m_{2}} Q_{2}, \tfrac{1}{m_{3}} Q_{3}\big), \ve{h}\right)_{\dot{H}^{1}} = 0.
\end{align}
Note also that
\begin{align}\label{NegativeLQ}
	\F(\Q, \Q)=\F(\Q)=-K(\Q)<0.
\end{align}

\begin{lemma}\label{BoundI}
Taking a smaller $\delta_{0}$, if necessary, for all $t \in I_0$, we have
\begin{equation}\label{DeltaBound}
\delta(t) \sim |\alpha(t)| \sim \|\ve{h}(t)\|_{\dot{H}^{1}}.
\end{equation}
\end{lemma}

\begin{proof}
Let $\ve{v} = \ve{u}_{[\eta(t), \theta(t), \mu(t)]}(t) - \Q = \ve{h} + \alpha(t)\Q$. From \eqref{OrQh}, we obtain
\begin{align}\label{EquV}
K(\ve{v}) = \alpha^{2}K(\Q) + K(\ve{h}).
\end{align}
Note that $K(\ve{v}) \sim \|\ve{v}\|^{2}_{\dot{H}^{1}}$ is small when $\delta(t)$ is small.

By a Taylor expansion, we have
\[
E(\ve{v} + \Q) - E(\Q) = \left\langle E^{\prime}(\Q), \ve{v} \right\rangle + \F(\ve{v}) + o(\|\ve{v}\|^{3}_{\dot{H}^{1}}).
\]
Since $E(\Q + \ve{v})= E(\Q)$ and $E^{\prime}(\Q) = 0$, it follows that
\begin{align}
|\F(\ve{v})| \lesssim \|\ve{v}\|^{3}_{\dot{H}^{1}}.
\end{align}
Moreover, since $\F(\Q) < 0$ (cf. \eqref{NegativeLQ}), we can write
\[
\F(\ve{v}) = \F(\ve{h}) + \alpha^{2}\F(\Q) = \F(\ve{h}) - \alpha^{2}|\F(\Q)|.
\]
This implies that $|\F(\ve{h}) - \alpha^{2}|\F(\Q)|| \leq C\|\ve{v}\|^{3}_{\dot{H}^{1}}$. Additionally, by Proposition~\ref{CoerQuadra} (cf. \eqref{h: Cond}), we deduce that $\|\ve{h}\|^{2}_{\dot{H}^{1}} \sim \F(\ve{h})$. Therefore,
\begin{align}\label{TwoI}
\alpha^{2} \leq O(\|\ve{h}\|^{2}_{\dot{H}^{1}} + \|\ve{v}\|^{3}_{\dot{H}^{1}})
\quad \text{and} \quad 
\|\ve{h}\|^{2}_{\dot{H}^{1}} \leq O(\alpha^{2} + \|\ve{v}\|^{3}_{\dot{H}^{1}}).
\end{align}

Since $K(\ve{v}) \sim \|\ve{v}\|^{2}_{\dot{H}^{1}}$, combining \eqref{EquV} and \eqref{TwoI}, we obtain, for $\delta_{0}$ sufficiently small,
\[
|\alpha| \sim \|\ve{h}\|_{\dot{H}^{1}} \sim \|\ve{v}\|_{\dot{H}^{1}}.
\]
Finally, as 
\[
\delta(t) = |K(\ve{u}) - K(\Q)| = |K(\ve{v}) - \alpha|\F(\Q)||,
\]
we conclude that $\delta(t) \sim |\alpha|$. This completes the proof.
\end{proof}

\begin{lemma}\label{BoundII}
Let $(\eta(t), \theta(t), \mu(t))$ be as in Lemma~\ref{ExistenceFree} and $\ve{h}(t)$ and $\alpha(t)$ be as in \eqref{DefH}. Then, we have
\begin{align}
\label{alfaE}
|\eta^{\prime}(t)|+|\theta^{\prime}(t)|+|\alpha^{\prime}(t)|+\frac{|\mu^{\prime}(t)|}{|\mu(t)|}
\lesssim\mu^{2}(t)\delta(t),
\end{align}
\end{lemma}
for $\delta_{0}$ small enough.
\begin{proof}
We define ${\delta}^{\ast}(t) := |\eta'(t)| + |\theta'(t)| + \left|\tfrac{\mu'(t)}{\mu(t)}\right| + \mu^2(t)\delta(t)$ and
\[ 
\ve{v}(t, y) := \ve{u}_{\left[\eta(t),\theta(t), \mu(t)\right]}(t, y).
\]
A straightforward calculation shows that the equation \eqref{NLS} takes the form:
\begin{align}\label{Line}
i\begin{pmatrix}
\partial_{t}v_{1} \\
\partial_{t}v_{2} \\
\partial_{t}v_{3}
\end{pmatrix}
&+\mu^{2}(t)\begin{pmatrix}
  \tfrac{1}{2m_{1}}\Delta v_{1} \\
  \tfrac{1}{2m_{2}}\Delta v_{2} \\
  \tfrac{1}{2m_{3}}\Delta v_{3}
\end{pmatrix}
+
\begin{pmatrix}
(\eta^{\prime}(t) + \theta^{\prime}(t))v_{1} \\
2\eta^{\prime}(t)v_{2} \\
2\theta^{\prime}(t)v_{3}
\end{pmatrix} \\
&+\tfrac{\mu^{\prime}(t)}{\mu(t)}\begin{pmatrix}
2 + y \cdot \nabla v_1 \\
2 + y \cdot \nabla v_2 \\
2 + y \cdot \nabla v_3
\end{pmatrix}
+\mu^{4}(t)
\begin{pmatrix}
\overline{v_{1}}v_{2}v_{3} \\
{v_{1}}^{2}\overline{v}_{2} \\
{v_{1}}^{2}\overline{v}_{3}
\end{pmatrix}=\ve{0}.
\end{align}
Moreover, since $\ve{v}=(1+\alpha(t))\Q+\ve{h}$, equation \eqref{Line} shows that $\ve{h}$ satisfies, for $t \in I_{0}$,  
\begin{align}
i\partial_{t}\ve{h} &+ \mu^{2}(t)\begin{pmatrix}
  \tfrac{1}{2m_{1}}\Delta h_{1} \\
  \tfrac{1}{2m_{2}}\Delta h_{2} \\
  \tfrac{1}{2m_{3}}\Delta h_{3}
\end{pmatrix}
+ i\alpha^{\prime}(t)\Q + (\eta^{\prime}(t)+\tfrac{1}{5}\theta^{\prime}(t))\Q_{p} + \tfrac{2}{5}\theta^{\prime}(t)\Q_{q} \nonumber \\
&+ i\tfrac{\mu^{\prime}(t)}{\mu(t)}\Lambda\Q = O\left(\mu^2(t)\delta(t) + \delta(t){\delta}^{\ast}(t)\right) \quad \text{in} \quad \dot{H}^1.
\end{align}
Using \eqref{h: Cond}, we obtain 
\[\partial_{t}\ve{h} \bot \text{span}\left\{\nabla \Q, i\Q_{p}, i\Q_{q}, \Lambda\Q \right\}\] 
and $\F(\Q, \partial_{t}\ve{h}) = 0$ for $t \in I_{0}$. Then, Lemma~\ref{BoundI} implies (recall that $(\Q_{p}, \Q_{q})_{\dot{H}^1}=0$)
\[
|\alpha'(t)| = O\left(\mu^2(t)\delta(t) + \delta(t){\delta}^{\ast}(t)\right), \quad 
\left|\tfrac{\mu^{\prime}(t)}{\mu(t)}\right| = O\left(\mu^2(t)\delta(t) + \delta(t){\delta}^{\ast}(t)\right)
\]
and  
\[|\theta'(t)| = O\left(\mu^2(t)\delta(t) + \delta(t){\delta}^{\ast}(t)\right)\quad
|\eta'(t)| = O\left(\mu^2(t)\delta(t) + \delta(t){\delta}^{\ast}(t)\right).
\]
By a continuity argument, we obtain the result for $\delta_{0}$ sufficiently small.
\end{proof}

\section{Convergence for subcritical threshold solution}\label{S: SubThres}

Henceforth, we assume the constants $m_1$, $m_2$, and $m_3$ satisfy the mass resonance condition $2m_1 + m_2 = m_3$ and, in particular, that the conclusions of Lemma~\ref{VirialIden} hold.

This section is devoted to proving the following result:

\begin{proposition}\label{CompacDeca}
Let $\ve{u}$ be a radial solution of \eqref{NLS} on the interval $I=(T_{-}, T_{+})$ satisfying 
\begin{align}\label{PropCon11}
E(\ve{u}_{0})=E(\Q) \qtq{and} K(\ve{u}_{0})<K(\Q).
\end{align}
Then the solution is global, i.e., $I=\R$. Moreover, if 
\begin{align}\label{ScaTa}
\|\ve{u}\|_{L^{6}_{t, x}((0, \infty)\times \R^{4})}=\infty,
\end{align}
then there exist parameters $\theta_{1}, \theta_{2}\in \R$, $\lambda>0$, and constants $c>0$, $C>0$ such that
\[
\|\ve{u}(t)-\Q_{[\theta_{1}, \theta_{2}, \lambda]}\|_{\dot{H}^{1}}\leq Ce^{-c t}
\qtq{for all $t\geq 0$.}
\]
An analogous result holds for negative times.
\end{proposition}

As a consequence of the previous proposition, we obtain the following corollary:

\begin{corollary}\label{ClassC}
There exists no radial solution to equation \eqref{NLS} satisfying both \eqref{PropCon11} and 
\begin{equation}\label{Infinity10}
\|\ve{u}\|_{L^{6}_{t, x}((0, \infty)\times \R^{4})}=
\|\ve{u}\|_{L^{6}_{t, x}((-\infty,0)\times \R^{4})}=
\infty.
\end{equation}
\end{corollary}

We will first prove Proposition~\ref{CompacDeca}, followed by Corollary~\ref{ClassC}.

We begin with the following lemma in the spirit of Kenig and Merle's work \cite{KenigMerle2006}. The proof follows along similar lines to \cite[Proposition 5.3]{OgaTsu2023}.

\begin{lemma}[Compactness]\label{Compacness11}
Let $\ve{u}(t)$ be a radial solution of \eqref{NLS} with maximal existence interval $I=[0, T_{+})$ that satisfies both \eqref{PropCon11} and 
\begin{align}\label{NscaF}
    \|\ve{u}\|_{L^{6}_{t, x}((0, T_{+})\times \R^{4})}=\infty.
\end{align}
Then there exists a scaling parameter $\lambda(t): [0, T_{+}) \to (0, +\infty)$ such that
\begin{equation}\label{CompactX}
\left\{\ve{u}_{[\lambda(t)]}: t\in [0, T_{+})\right\} \quad
\text{is pre-compact in $\dot{H}^{1}$},
\end{equation}
where $\ve{u}_{[\lambda(t)]}(t,x) := \lambda(t)^{-1}u(\lambda(t)^{-2}t, \lambda(t)^{-1}x)$.
\end{lemma}

\begin{lemma}[Global solution]\label{GlobalW}
Let $\ve{u}(t)$ be a radial solution of \eqref{NLS} defined on its maximal interval of existence $I=(T_{-}, T_{+})$. If the initial data satisfies
\begin{equation}\label{leqleq}
E(\ve{u}_{0}) \leq E(\Q) \qtq{and} K(\ve{u}_{0}) \leq K(\Q),
\end{equation}
then the solution extends globally in time, i.e., $I=\R$.
\end{lemma}
\begin{proof}
We consider three cases:

\textbf{Case (i).} Suppose that $K(\ve{u}_{0})=K(\Q)$. Lemma~\ref{ConvexIn} implies that $E(\ve{u})=E(\Q)$. Then the variational characterization given in Proposition~\ref{TheGN} shows that $\ve{u}_{0}=\Q_{[\theta_{1},\theta_{2}, \lambda_{0}]}$.

\textbf{Case (ii).} Suppose that $K(\ve{u}_{0})<K(\Q)$ and $E(\ve{u})<E(\Q)$. Theorem~\ref{Th1} shows that the solution $\ve{u}$ is global.

\textbf{Case (iii).} Suppose that $K(\ve{u}_{0})<K(\Q)$ and $E(\ve{u})=E(\Q)$. If $\|\ve{u}\|_{L^{6}_{t, x}(I\times \R^{4})}<\infty$,
then by the finite blow-up criterion, we conclude that $\ve{u}$ is a global solution.

On the other hand, if $\|\ve{u}\|_{L^{6}_{t, x}([0, T_{+})\times \R^{6})}=\infty$, Lemma~\ref{Compacness11} implies that there exists a function
$\lambda(t)$ such that $\left\{\ve{u}_{[\lambda(t)]}: t\in [0, T_{+})\right\}$ 
is pre-compact in $\dot{H}^{1}$. 

Suppose, for contradiction, that $T_{+}< +\infty$. 
By compactness and following the same argument as in Case 1 of \cite[Proposition 5.3]{KenigMerle2006}, we obtain
\begin{align}\label{LamdaI}
	\lim_{t\to T_{+}}\lambda(t)=+\infty.
\end{align}

Now, for $R>0$, define (we set $\ve{u}:=(u,v,g)$)
\[
z_{R}(t)=\int_{\R^{4}}[m_{1}|u(t,x)|^{2}+m_{2}|v(t,x)|^{2}+m_{3}|g(t,x)|^{2}]\xi\left( \tfrac{x}{R} \right)dx
\quad\text{for } t\in [0, T_{+}),
\]
where $\xi=1$ if $|x|\leq 1$ and $\xi=0$ if $|x|\geq 2$. From (cf. \eqref{F1LocalVirial}) 
\[
z^{\prime}_{R}(t)=\tfrac{2}{R}\int_{\R^{4}}(\overline{u}\nabla u + \overline{v}\nabla v+\overline{g}\nabla g )\cdot (\nabla\xi)\left( \tfrac{x}{R} \right),
\]
and by Hardy's inequality together with $K(\ve{u(t)})\leq K(\Q)$, we obtain $|z^{\prime}_{R}(t)|\leq C_{0}$, where $C_{0}$ is a constant independent of $R$. Applying the fundamental theorem of calculus on $[t, T]\subset [0, T_{+})$, we have
\begin{align}\label{InteF}
	|z_{R}(t)-z_{R}(T)|\leq C_{0}|t-T|.
\end{align}

By the compactness property \eqref{CompactX}, we see that for any $\rho>0$,
\begin{equation}\label{CritialE}
\int_{|x|\geq{\rho}}|u(t,x)|^{4}+|v(t,x)|^{4} +|g(t,x)|^{4}\,dx \to 0 \quad\text{as } t\to T_{+}.
\end{equation}

Combining \eqref{LamdaI}, \eqref{CritialE}, and taking the limit $t\to T_{+}$,
we obtain
\[
\lim_{t\to T_{+}}z_{R}(t)=0.
\]

From \eqref{InteF}, we have $|z_{R}(t)|\leq  C_{0}|t-T_{+}|$. Letting $R\to+\infty$, we conclude that $\ve{u}(t)\in L^{2}$ and
$\|\ve{u}(t)\|^{2}_{L^{2}}\leq C_{0}|t-T_{+}|$. In particular, this implies $\ve{u}_{0}=0$, which contradicts $E(\ve{u})= E(\Q)>0$. Therefore, $T_{+}=+\infty$. 
\end{proof}

\begin{lemma}[Convergence in the ergodic mean]\label{ZeroVirial}
Suppose $\ve{u}$ is a radial solution of \eqref{NLS} satisfying assumptions \eqref{PropCon11} and \eqref{ScaTa}. Then 
\begin{align}\label{Cdelta}
    \lim_{T\to +\infty}\frac{1}{T}\int^{T}_{0}\delta(t)dt=0,
\end{align}
where $\delta(t)=K(\Q)-K(\ve{u}(t))$.
\end{lemma}

\begin{proof} 
 Since $|\nabla w_{R}| \lesssim \tfrac{R^{2}}{|x|}$, Hardy's inequality implies
\[
|I_{R}[\ve{u}](t)| \leq C_{\ast} R^{2}
\]
for some constant $C_{\ast}>0$.

Given $\epsilon>0$ and choosing $R>0$ (to be determined later), we write (cf. Lemma~\ref{VirialIden})
\begin{align*}
    \frac{d}{d t}I_{R}[\ve{u}] &= F_{\infty}[\ve{u}(t)] \\
                          &+ F_{R}[\ve{u}(t)] - F_{\infty}[\ve{u}(t)].
\end{align*}

Using the relations $K(\Q)=4P(\Q)$ and $2E(\ve{u})=K(\Q)$, we obtain
\[
F_{\infty}[\ve{u}(t)] = 4[K(\ve{u})-4P(\ve{u})] = 4[K(\Q)-K(\ve{u})] = 4\delta(t).
\]
Thus,
\[
\frac{d}{d t}I_{R}[\ve{u}] = 4\delta(t) + [F_{R}[\ve{u}(t)] - F_{\infty}[\ve{u}(t)]].
\]

Next, observe that (We set \ve{u}:=(u,v,g))
\begin{align*}
F_{R}[\ve{u}(t)] - F_{\infty}[\ve{u}(t)] 
&= \int_{|x|\geq R}\left(-\tfrac{1}{4} \Delta \Delta w_{R}\right)\left(\tfrac{1}{m_{1}}|u|^{2}+\tfrac{1}{m_{2}}|v|^{2}+\tfrac{1}{m_{3}}|g|^{2}\right)\,dx \\
&- 2\RE\int_{|x|\geq R}\Delta[w_{R}(x)]\overline{u}(x)^{2}v(x)g(x)dx \\
&-2\RE\int_{|x|\geq R}\left[\tfrac{1}{m_{1}}|\nabla u|^{2} + \tfrac{1}{m_{2}}|\nabla v|^{2} + \tfrac{1}{m_{3}}|\nabla g|^{2} - 8\overline{u}^{2}vg\right]dx \\
&+ \RE\int_{|x|\geq R} \left[\tfrac{1}{m_{1}}\overline{u_{j}} u_{k} + \tfrac{1}{m_{2}}\overline{v_{j}} v_{k} + \tfrac{1}{m_{3}}\overline{g_{j}} g_{k}\right]\partial_{jk}[w_{R}(x)]dx.
\end{align*}

By compactness in $\dot{H}^{1}$, there exists $C_{\epsilon}>0$ such that
\[
\sup_{t\geq 0}\int_{|x|>\tfrac{C_{\epsilon}}{\lambda(t)}}\left[|\nabla u|^{2} + |\nabla v|^{2} +|\nabla g|^{2} + |u|^{4} + |v|^{4} + |g|^{4}\right](t,x)dx \ll \epsilon.
\]

Using the conditions on the weight $w_{R}$ specified in Lemma~\ref{VirialIden} and applying H\"older's inequality, we obtain for $R \geq \tfrac{C_{\epsilon}}{\lambda(t)}$,
\[
|F_{R}[\ve{u}(t)] - F_{\infty}[\ve{u}(t)]| \leq \epsilon.
\]

\begin{claim}\label{Claim11}
\begin{align}\label{Claim1}
    \lim_{t\to +\infty}\sqrt{t}\lambda(t) = +\infty.
\end{align}
\end{claim}

Assuming the claim holds, there exists $t_{0}\geq 0$ such that for all $t \geq t_{0}$ we have
\[
\lambda(t) \geq \tfrac{M_{0}}{\sqrt{t}},
\]
where we choose $M_{0}$ satisfying
\[
M_{0}\epsilon_{0} \geq C_{\epsilon} \quad \text{with} \quad \epsilon^{2}_{0} := \tfrac{\epsilon}{2C_{\ast}}.
\]

Setting $R := \epsilon_{0}\sqrt{T}$ for $T \geq t_{0}$, we find that for $t \in [t_{0}, T]$,
\[
R \geq \epsilon_{0}\sqrt{T}\frac{M_{0}}{\sqrt{t}\lambda(t)} = \frac{\sqrt{T}}{\sqrt{t_{0}}}\frac{M_{0}\epsilon_{0}}{\lambda(t)} \geq \frac{C_{\epsilon}}{\lambda(t)}.
\]

Combining the above estimates and applying the fundamental theorem of calculus on $[t_{0}, T]$, we conclude
\[
\tfrac{4}{T}\int^{T}_{t}\delta(t)\,dt \leq 2C_{\ast}\tfrac{R^{2}}{T} + \epsilon\tfrac{(T-t_{0})}{T} \leq 2\epsilon.
\]

Finally, taking the limit $T\to +\infty$ followed by $\epsilon\to 0$, we obtain
\[
\lim_{T\to +\infty}\tfrac{1}{T}\int^{T}_{0}\delta(t)dt = 0.
\]

To complete the proof, it remains to verify the claim.
\begin{proof}[Proof of Claim~\ref{Claim11}]
Suppose by contradiction that \eqref{Claim1} does not hold. Then there exists $s \in [0, +\infty)$ such that
$\lim_{t_{n}\to+\infty}\sqrt{t_{n}}\lambda(t_{n}) = s$. Consequently,
\begin{align}\label{LamZer}
    \lim_{t_{n}\to+\infty}\lambda(t_{n}) = 0.
\end{align}

Define
\[
\ve{w}_{n}(\tau, y) = \lambda(t_{n})^{-1}\ve{u}\left( t_{n} + \tfrac{\tau}{\lambda(t_{n})^{2}}, \tfrac{y}{\lambda(t_{n})}\right).
\]

By compactness, there exists $\ve{w}_{0} \in \dot{H}^{1}$ such that $\ve{w}_{n}(0) \to \ve{w}_{0}$ in $\dot{H}^{1}$ as $n \to \infty$. Since $E(\ve{u}_{0}) = E(\Q)$ and $K(\ve{u}(t_{n})) < K(\Q)$, it follows that $E(\ve{w}_{0}) = E(\Q)$ and $K(\ve{w}_{0}) \leq K(\Q)$. Lemma~\ref{GlobalW} then implies that the solution $\ve{w}(t)$ to \eqref{NLS} with initial data $\ve{w}_{0}$ is global and satisfies $E(\ve{w}(t)) = E(\Q)$ for all $t \in \R$.

Now, since $-\sqrt{t_{n}}\lambda(t_{n}) \to -s$, stability theory (cf. Lemma~\ref{ConSNLS}) yields
\[
\lambda(t_{n})^{-1}\ve{u}_{0}\left(\tfrac{y}{\lambda(t_{n})}\right) = \ve{w}_{n}(-t_{n}\lambda(t_{n})^{2}, y) \to \ve{w}(-s^{2},y).
\]

However, by \eqref{LamZer} we have
\[
\lambda(t_{n})^{-1}\ve{u}_{0}\left(\tfrac{y}{\lambda(t_{n})}\right) \rightharpoonup 0 \quad \text{in $\dot{H}^{1}$},
\]
which contradicts $E(\ve{w}(-s^{2})) = E(\Q) > 0$. This completes the proof of the claim.
\end{proof}
\end{proof}

As a direct consequence of Lemma~\ref{ZeroVirial}, we obtain the following result.
\begin{lemma}\label{ZeroVirial11}
Let $\ve{u}$ be a radial solution of \eqref{NLS} satisfying the assumptions of Proposition~\ref{CompacDeca}. 
Then there exists a sequence $t_{n}\to\infty$ so that 
\[
\lim_{n\to +\infty}\delta(t_{n})=0.
\]
\end{lemma}

Let $\ve{u}(t,x) = (u(t,x), v(t,x), g(t,x))$ be a solution of \eqref{NLS}. Consider $\delta_0 > 0$ and the modulation parameters $\eta(t)$, $\theta(t)$, $\mu(t)$, and $\alpha(t)$ given by Lemma~\ref{ExistenceFree}, which are defined for all $t \in I_0$.  

The decomposition \eqref{DecomUFree} and the estimate \eqref{EstimateOne} imply the existence of a constant $C_0 > 0$ such that:  for all  $t \in I_{0}$
\begin{align*}  
\int_{\mu(t)\leq|x|\leq 2\mu(t)}\left[|\nabla u(t,x)|^{2} + |\nabla v(t,x)|^{2} + |\nabla g(t,x)|^{2}\right]dx  
\geq \int_{1\leq|x|\leq 2}|\nabla Q_{1}|^{2} - C_{0}\delta(t). 
\end{align*}  
Taking $\delta_{0} > 0$ sufficiently small, there exists $\epsilon > 0$ for which  
\begin{align*}  
\int_{\frac{\mu(t)}{\lambda(t)}\leq|x|\leq\frac{2\mu(t)}{\lambda(t)}} \tfrac{1}{\lambda(t)^{4}} 
\left[\left|\nabla u\left(t,\tfrac{x}{\lambda(t)}\right)\right|^{2} + \left|\nabla v\left(t,\tfrac{x}{\lambda(t)}\right)\right|^{2}+ \left|\nabla g\left(t,\tfrac{x}{\lambda(t)}\right)\right|^{2}\right]dx \geq \epsilon  
\end{align*}  
for all $t \in I_{0}$. Since $\left\{\ve{u}_{[\lambda(t)]}: t \in [0, +\infty)\right\}$ is pre-compact in $\dot{H}^{1}$, we deduce that $|\mu(t)| \sim |\lambda(t)|$ for $t \in I_{0}$.  

Therefore, we may adjust $\lambda(t)$ so that $\left\{\ve{u}_{[\lambda(t)]}: t \in [0, +\infty)\right\}$ remains pre-compact in $\dot{H}^{1}$ with  
\begin{align}\label{Kcp}  
    \lambda(t) = \mu(t) \quad \text{for all } t \in I_{0}.  
\end{align}


\begin{lemma}\label{Lemma11}
There exists a constant $C=C(\delta_{1})>0$ such that for any interval $[t_{1}, t_{2}]\subset [0, \infty)$, 
\begin{equation}\label{BoundT1}
\int^{t_{2}}_{t_{1}}\delta(t)dt\leq C\sup_{t\in[t_{1},t_{2}]}\frac{1}{\lambda(t)^{2}}
\left\{\delta(t_{1})+\delta(t_{2})\right\}.
\end{equation}
\end{lemma}
\begin{proof}
Let $R>1$ be a constant to be determined later. We establish the localized virial identities (cf. Lemma \ref{VirialModulate}) with $\chi(t)$ satisfying
\[
\chi(t)=
\begin{cases}
1& \text{if } \delta(t)<\delta_{0}, \\
0& \text{if } \delta(t)\geq \delta_{0}.
\end{cases}
\]
From Lemma \ref{VirialModulate} (recalling that $F_{\infty}[\ve{u}(t)]=4\delta(t)$), we obtain
\begin{equation}\label{VirilaX}
\frac{d}{dt}I_{R}[\ve{u}(t)]=F_{\infty}[\ve{u}(t)]+\EE(t)=4\delta(t)+\EE(t),
\end{equation}
where
\begin{equation}\label{Error11}
\EE(t)=
\begin{cases}
F_{R}[\ve{u}(t)]-F_{\infty}[\ve{u}(t)]& \text{if } \delta(t)\geq \delta_{0}, \\
F_{R}[\ve{u}(t)]-F_{\infty}[\ve{u}(t)]-\K[\ve{u}(t)]& \text{if } \delta(t)< \delta_{0},
\end{cases}
\end{equation}
with
\begin{equation}\label{Error22}
\K(t)=F_{R}[\Q_{[-\eta(t), -\theta(t), \lambda(t)^{-1}]}]-F_{\infty}[\Q_{[-\eta(t), -\theta(t), \lambda(t)^{-1}]}].
\end{equation}

We now assume the following claims temporarily to complete the proof.

\textbf{Claim I.} For $R>1$, we have
\begin{align}\label{EstimateV11}
	|I_{R}[\ve{u}(t_{j})]|\lesssim \frac{R^{2}}{\delta_{0}}\delta(t_{j}) \quad& \text{if } \delta(t_{j})\geq \delta_{0} \text{ for } j=1, 2,\\
	\label{EstimateV22}
	|I_{R}[\ve{u}(t_{j})]|\lesssim R^{2} \delta(t_{j}) \quad &\text{if } \delta(t_{j})< \delta_{0} \text{ for } j=1, 2.
\end{align}

\textbf{Claim II.} Given $\epsilon>0$, there exists $\rho_{\epsilon}=\rho(\epsilon)>0$ such that if $R=\rho_{\epsilon}\sup_{t\in [t_{1}, t_{2}]}\tfrac{1}{\lambda(t)}$, then
\begin{align}\label{EstimateE11}
	|\EE(t)|\leq \frac{\epsilon}{\delta_{0}}\delta(t) \quad &\text{uniformly for } t\in [t_{1}, t_{2}] \text{ and } \delta(t)\geq \delta_{0},\\
	\label{EstimateE22}
|\EE(t)|\leq \epsilon \delta(t)\quad &\text{uniformly for } t\in [t_{1}, t_{2}] \text{ and } \delta(t)< \delta_{0}.
\end{align}

Assuming Claims I and II, integrating \eqref{VirilaX} over $[t_{1}, t_{2}]$ and applying estimates \eqref{EstimateV11}, \eqref{EstimateV22}, \eqref{EstimateE11}, and \eqref{EstimateE22} yields
\[
\int^{t_{2}}_{t_{1}}\delta(t)dt\lesssim 
\frac{\rho_{\epsilon}}{\delta_{0}}\sup_{t\in [t_{1}, t_{2}]}\frac{1}{\lambda(t)^{2}}(\delta(t_{1})+\delta(t_{2}))
+\Big(\frac{\epsilon}{\delta_{0}}+\epsilon\Big)\int^{t_{2}}_{t_{1}}\delta(t)dt.
\]
Choosing $\epsilon=\epsilon(\delta_{0})$ sufficiently small gives the estimate \eqref{BoundT1}.

To complete the proof, we now verify the claims.
\begin{proof}[Proof of Claim I] 
Note that if $\delta(t_{j})\geq \delta_{0}$, Hardy's inequality implies
\[
|I_{R}[\ve{u}(t)]|
\lesssim R^{2}\|\ve{u}\|^{2}_{L^{\infty}_{t}\dot{H}^{1}}\lesssim_{Q} \frac{R^{2}}{\delta_{0}}\delta(t_{j}),
\]
which proves \eqref{EstimateV11}. On the other hand, if $\delta(t_{j})< \delta_{0}$, using the fact that $Q$ is real, we obtain
\begin{align*}
|I_{R}[\ve{u}(t_{j})]|&=\left| 2\IM\int_{\R^{4}}\nabla w_{R}(\overline{\ve{u}}_{[\eta(t_{j}), \theta(t_{j}), \lambda(t_{j})]} \nabla \ve{u}_{[\eta(t_{j}), \theta(t_{j}), \lambda(t)]}
-\Q\nabla\Q)
dx\right|\\
&\lesssim
R^{2}[\|\ve{u}\|_{\Lm^{\infty}_{t}\dot{H}^{1}_{x}}+\|\Q\|_{\dot{H}^{1}}]
\|{\ve{u}}_{[\eta(t_{j}), \theta(t_{j}), \lambda(t_{j})]}-\Q\|_{\dot{H}^{1}}\\
&\lesssim_{Q}R^{2} \delta(t_{j}),
\end{align*}
where the last inequality follows from \eqref{EstimateOne}.
\end{proof}

\begin{proof}[Proof of Claim II]
Assume that $\delta(t) \geq \delta_{0}$. From \eqref{CompactX}, we infer that for each $\epsilon > 0$, there exists $\rho_{\epsilon} = \rho(\epsilon) > 0$ such that (recall $\ve{u} = (u, v, g)$)
\begin{equation}\label{CompactAgain}
\sup_{t \geq 0} \int_{|x| > \tfrac{C_{\epsilon}}{\lambda(t)}} \left[|\nabla u|^2 + |\nabla v|^2 + |\nabla g|^2 + |u|^4 + |v|^4 + |g|^4\right](t, x) \, dx \ll \epsilon.
\end{equation}
Let
\[
R := \rho_{\epsilon} \sup_{t \in [t_1, t_2]} \tfrac{1}{\lambda(t)}.
\]
From the argument in Lemma~\ref{ZeroVirial}, we have
\[
|F_R[\ve{u}(t)] - F_\infty[\ve{u}(t)]| \leq \epsilon \leq \tfrac{\epsilon}{\delta_0} \delta(t)
\quad \text{for all $t \in [t_1, t_2]$ with $\delta(t) \geq \delta_0$}.
\]
This establishes the estimate \eqref{EstimateE11}.

Now, suppose $\delta(t) < \delta_0$. By the definition of $\EE(t)$ in \eqref{Error11} and an argument analogous to that in Lemma~\ref{ZeroVirial}, we may write
\begin{align}
\EE(t) &\leq |F_R[\ve{u}(t)] - F_R[\Q_{[-\eta, -\theta, \lambda^{-1}]}]| 
        + |F_\infty[\ve{u}(t)] - F_\infty[\Q_{[-\eta, -\theta, \lambda^{-1}]}]| \nonumber \\
    &= |F_R[\ve{u}_{[\eta,\theta,\lambda]}(t)] - F_R[\Q]| 
        + |F_\infty[\ve{u}_{[\eta,\theta,\lambda]}(t)] - F_\infty[\Q]| \nonumber \\
    &\lesssim \Big[\|\ve{u}_{[\eta,\theta,\lambda]}(t)\|^2_{\dot{H}^1(|x|\geq R)} 
        + \|\Q\|^2_{\dot{H}^1(|x|\geq R)} \nonumber \\
        &\quad + \|\ve{u}(t)\|^2_{L^4(|x|\geq R)} 
        + \|\Q\|^2_{L^4(|x|\geq R)}\Big] 
        \|\ve{u}_{[\eta,\theta,\lambda]}(t) - \Q\|_{\dot{H}^1} \nonumber \\\label{Decomp44}
    &\lesssim \Big[\|\ve{u}_{[\eta,\theta,\lambda]}(t)\|^2_{\dot{H}^1(|x|\geq R)} 
        + \|\Q\|^2_{\dot{H}^1(|x|\geq R)} \nonumber \\
        &\quad + \|\ve{u}(t)\|^2_{L^4(|x|\geq R)} 
        + \|\Q\|^2_{L^4(|x|\geq R)}\Big] \delta(t),
\end{align}
for all $t \in [t_1, t_2]$. By \eqref{CompactAgain}, the term \eqref{Decomp44} is bounded as
\begin{align*}
|\eqref{Decomp44}| &\lesssim \Big[\|\ve{u}(t)\|^2_{\dot{H}^1(|x|\geq \rho_\epsilon/\lambda(t))} 
        + \|\Q\|^2_{\dot{H}^1(|x|\geq \rho_\epsilon)} \nonumber \\
        &\quad + \|\ve{u}(t)\|^2_{L^4(|x|\geq \rho_\epsilon/\lambda(t))} 
        + \|\Q\|^2_{L^4(|x|\geq \rho_\epsilon)}\Big] \\
    &\leq \epsilon \delta(t),
\end{align*}
provided $\rho_\epsilon$ is sufficiently large. This completes the proof of Claim II.
\end{proof}
\end{proof}


\begin{proposition}[Control of the variations of the parameter $\lambda(t)$]\label{Spatialcenter}
Let $[t_{1}, t_{2}]$ be an interval of $(0, \infty)$ satisfying $t_{1}+\tfrac{1}{\lambda(t_{1})}\leq t_{2}$. Then there exists a positive constant $C_{0}$ such that
\begin{equation}\label{BoundCenter}
\left|\frac{1}{\lambda(t_{2})^{2}}-\frac{1}{\lambda(t_{1})^{2}}\right|\leq C_{0}\int^{t_{2}}_{t_{1}}\delta(t)\, dt.
\end{equation}
\end{proposition}

\begin{proof}
The proof is divided into three steps.

\textsl{Step 1.} There exists a positive constant $C_{1}$ such that
\begin{equation}\label{step11}
\tfrac{\lambda(s)}{\lambda(t)}+\tfrac{\lambda(t)}{\lambda(s)}\leq C_{1} \quad \text{for all $t$, $s\geq 0$ such that $|t-s|\leq \tfrac{1}{\lambda(s)^{2}}$}.
\end{equation}
To prove this, suppose by contradiction that sequences $s_{n}$, $t_{n}$ satisfy 
\begin{align}\label{CBound}
|t_{n}-s_{n}|\leq \tfrac{1}{\lambda(s_{n})} \qtq{but}
\tfrac{\lambda(s_{n})}{\lambda(t_{n})}+\tfrac{\lambda(t_{n})}{\lambda(s_{n})}\to \infty.
\end{align}
Taking a subsequence if necessary, we may assume that
\[
\lim_{n\to \infty}\lambda(s_{n})^{2}(t_{n}-s_{n})=\tau_{0}\in[-1, 1].
\]
Consider the solution of \eqref{NLS}
\[
\ve{v}_{n}(\tau,y)=\lambda(s_{n})^{-1}\ve{u}\( \tfrac{\tau}{\lambda(s_{n})^{2}}+s_{n},\tfrac{y}{\lambda(s_{n})} \).
\]
By compactness, there exists $\ve{v}_{0}\in \dot{H}^{1}$ such that
\[
\ve{v}_{n}(0,y)\to \ve{v}_{0}(y) \qtq{in $\dot{H}^{1}$ as $n\to \infty$.}
\]
Since $E(\ve{v})=E(\Q)$ and $K(\ve{v}_{0})\leq K(\Q)$, the solution $\ve{v}$ of \eqref{NLS} with initial data $\ve{v}_{0}$ is globally defined (cf. Lemma~\ref{GlobalW}), and by stability theory (cf. Lemma~\ref{ConSNLS}) we conclude that
\[
\ve{w}_{n}(y)=\ve{v}_{n}(\lambda(s_{n})^{2}(t_{n}-s_{n}), y)=
\lambda(s_{n})^{-1}\ve{u}\( t_{n},\tfrac{y}{\lambda(s_{n})}\)\to \ve{v}(\tau_{0}, y).
\]
Moreover, by compactness we have
\[
\tfrac{1}{\lambda(t_{n})}\ve{u}\( t_{n},\tfrac{y}{\lambda(t_{n})}\)=
\tfrac{\lambda(s_{n})}{\lambda(t_{n})}\ve{w}_{n}\(\tfrac{\lambda(s_{n})}{\lambda(t_{n})}y\)\to \varphi\neq 0
\]
in $\dot{H}^{1}$, which implies the boundedness of $\tfrac{\lambda(s_{n})}{\lambda(t_{n})} + \tfrac{\lambda(t_{n})}{\lambda(s_{n})}$, contradicting \eqref{CBound}.

\textsl{Step 2.}
There exists $\delta_{1}>0$ such that either
\begin{equation}\label{MinMax}
\inf_{t\in [T, T+\tfrac{1}{\lambda(T)^{2}}]}\delta(t)\geq \delta_{1} \quad \text{or}\quad
\sup_{t\in [T,T+\tfrac{1}{\lambda(T)^{2}}]}\delta(t)<\delta_{0}\quad \text{for any $T\geq 0$}.
\end{equation}
Assume by contradiction that there exist $t_{n}^{\ast}\geq 0$ and sequences $t_{n}, t^{\prime}_{n}\in [t_{n}^{\ast}, t_{n}^{\ast}+\tfrac{1}{\lambda(t_{n}^{\ast})^{2}}]$ with
\begin{align}\label{ContraStep2}
&\delta(t_{n})\to 0 \quad \text{and}\quad \delta(t^{\prime}_{n})\geq \delta_{1} \quad \text{as $n\to \infty$}.
\end{align}
Step 1 implies $\tfrac{\lambda(t_{n})}{\lambda(t^{\ast}_{n})}\leq C$, so for a subsequence,
\begin{align}\label{ContraLimit}
\lambda(t_{n})^{2}(t_{n}-t^{\prime}_{n})\to t^{\ast}\in[-C,C].
\end{align}
Define
\[
\ve{v}_{n}(\tau,y)=\lambda(t_{n})^{-1}\ve{u}\( \tfrac{\tau}{\lambda(t_{n})^{2}}+t_{n},\tfrac{y}{\lambda(t_{n})} \).
\]
Since $\delta(t_{n})\to 0$, compactness yields (cf. Proposition~\ref{UBS}) parameters $\lambda_{0}>0$, $\theta_{1}, \theta_{2}\in \R$ with
\begin{equation}\label{Step2Conver}
\ve{v}_{n}(0, \cdot)\to \Q_{[\theta_{1}, \theta_{2}, \lambda_{0}]} \text{ strongly in $\dot{H}^{1}$}.
\end{equation}
Combining \eqref{ContraLimit} and \eqref{Step2Conver} via stability theory gives
\[
\lambda(t_{n})^{-1}\ve{u}\(t^{\prime}_{n},\tfrac{y}{\lambda(t_{n})} \)=
\ve{v}_{n}(\lambda(t_{n})^{2}(t_{n}-t^{\prime}_{n}), y)\to \Q_{[\theta_{1}, \theta_{2}, \lambda_{0}]},
\]
contradicting \eqref{ContraStep2}.

\textsl{Step 3.} We now establish
\begin{align}\label{acD}
0\leq t_{1}\leq \tilde{t_{1}}\leq \tilde{t_{2}}\leq t_{2}=t_{1}+\tfrac{1}{C^{2}_{1}\lambda(t_{2})^{2}}
\Rightarrow 
\left|\tfrac{1}{\lambda(\tilde{t_{2}})^{2}}-\tfrac{1}{\lambda(\tilde{t_{1}})^{2}}\right|\leq C\int^{t_{2}}_{t_{1}}\delta(t)\,dt.
\end{align}
By Step 2, either $\sup_{t\in [t_{1}, t_{2}]}\delta(t)<\delta_{0}$ or $\inf_{t\in [t_{1}, t_{2}]}\delta(t)\geq \delta_{1}$. In the first case, integrating $\left|\tfrac{\lambda^{\prime}(t)}{\lambda(t)^{3}}\right|\lesssim \delta(t)$ (cf. \eqref{EstimateFree}) yields \eqref{acD}. In the second case, note that $\int^{t_{2}}_{t_{1}}\delta(t)\,dt\geq \delta_{1}(t_{2}-t_{1})$ and
\[
|\tilde{t_{1}}-\tilde{t_{2}}|\leq \tfrac{1}{C^{2}_{1}\lambda(t_{1})^{2}}\leq \tfrac{1}{\lambda(\tilde{t_{1}})^{2}}.
\]
Thus Step 1 gives ($C_{1}\geq1$)
\[
\left|\tfrac{1}{\lambda(\tilde{t_{2}})^{2}}-\tfrac{1}{\lambda(\tilde{t_{1}})^{2}}\right|
\leq \tfrac{2C^{5}_{1}}{\delta_{1}}\int^{t_{2}}_{t_{1}}\delta(t)dt.
\]
Finally, dividing $[t_{1}, t_{2}]$ into subintervals and combining these inequalities proves \eqref{BoundCenter}.
\end{proof}

\begin{proof}[{Proof of Proposition~\ref{CompacDeca}}]
With Lemmas~\ref{ZeroVirial11} and \ref{Lemma11} and Proposition~\ref{Spatialcenter} at hand, the proof of the proposition follows along the same lines as in \cite[Proposition 6.1]{MIWUGU2015}. Here we outline the main steps.

First, Lemma~\ref{Lemma11} and Proposition~\ref{Spatialcenter} imply that $\tfrac{1}{\lambda(t)^{2}}$ is bounded on $[0,\infty)$; see \cite[Lemma 6.9]{MIWUGU2015} for details.

Using the boundedness of $\tfrac{1}{\lambda(t)^{2}}$, Lemma~\ref{Lemma11} yields
\[
\int^{s}_{T}\delta(t)\,dt\leq C\left\{\delta(T)+\delta(s)\right\}
\qtq{for $[T, s]\subset [0, \infty]$.}
\]
Applying this to a sequence $t_n\to\infty$ with $\delta(t_n)\to 0$ (cf. Lemma~\ref{ZeroVirial11}), we obtain $\int^{\infty}_{T}\delta(t)\,dt\leq C\delta(T)$ for all $T\geq 0$. Then Gronwall's lemma shows that
\begin{align}\label{PIN}
	\int^{\infty}_{T}\delta(t)\,dt\leq C e^{-c T}.
\end{align}
for some constants $C, c>0$.

Combining this inequality with estimate \eqref{alfaE} and employing the same argument as in Proposition~\ref{SupercriQ} (cf. \eqref{deltaN}) below, we obtain
\begin{align}\label{DeltaZ}
	\lim_{t\to \infty}\delta(t)=0.
\end{align}
In particular, the modulation parameters $\lambda(t)$, $\eta(t)$ and $\theta(t)$ are well-defined for $t\geq t_{0}$ for some $t_{0}\geq 0$.

From Lemma~\ref{BoundCenter} and \eqref{PIN}, it follows that $\lim_{t\to \infty} \lambda(t)=\lambda_{\infty}\in (0, +\infty)$.
See \cite[Section 6.2]{MIWUGU2015} for details. Moreover, Proposition~\ref{Spatialcenter} gives
\[
\left|\frac{1}{\lambda(t)^{2}}-\frac{1}{\lambda_{\infty}^{2}}\right|\leq C e^{-c t}.
\]
Since $|\alpha(t)|\sim |\delta(t)|$, \eqref{DeltaZ} implies $\lim_{t\to \infty}\alpha(t)=0$. Thus, from \eqref{EstimateFree} we derive
\[
\delta(t)+\|\ve{h}(t)\|_{\dot{H}^{1}}\sim |\alpha(t)|
\leq C\int^{+\infty}_{t}|\alpha^{\prime}(s)|ds\leq 
C\int^{+\infty}_{t}\lambda(s)^{2}\delta(s)ds\leq Ce^{-ct}.
\]
Finally, since $\int^{\infty}_{T}|\eta^{\prime}(t)|+ |\theta^{\prime}(t)|dt\leq Ce^{-cT}$ for sufficiently large $T$, there exist $\eta_{\infty}$ and $\theta_{\infty}\in \R$ such that
$|\eta(t)-\eta_{\infty}|+|\theta(t)-\theta_{\infty}|\leq Ce^{-ct}$. Combining all these estimates yields
\[
\delta(t)+|\alpha(t)|+\|\ve{h}(t)\|_{\dot{H}^{1}}+|\eta(t)-\eta_{\infty}|+|\theta(t)-\theta_{\infty}|
	+\left|\frac{1}{\lambda(t)^{2}}-\frac{1}{\lambda_{\infty}^{2}}\right|\leq C e^{-c T},
\]
which completes the proof of the proposition.
\end{proof}

\begin{proof}[{Proof of Corollary~\ref{ClassC}}]
Suppose, by contradiction, that $\ve{u}$ satisfies \eqref{PropCon11} and \eqref{Infinity10}. Following the same arguments as above, we can construct $\lambda(t)$ such that the set 
$\left\{\ve{u}_{[\lambda(t)]}(t)): t\in \R\right\}$ is pre-compact in $\dot{H}^{1}$. Moreover, using the same approach developed in this section, we can show that
\[
\lim_{t\to -\infty}\delta(t)=\lim_{t\to \infty}\delta(t)=0.
\]
Additionally, by modifying the proof of Lemma~\ref{Lemma11}, we obtain
\[
\int^{n}_{-n}\delta(t)\,dt\leq C(\delta(n)+\delta(-n))\quad \text{for all $n\in \mathbb{N}$}.
\]
Taking the limit as $n\to \infty$, we conclude that $\delta(t)\equiv 0$, which contradicts \eqref{PropCon11}.\end{proof}


\section{Convergence for supercritical threshold solutions}\label{S: SuperThres}

The main objective of this section is to prove the following result.

\begin{proposition}\label{SupercriQ}
Let $\ve{u}\in H^{1}$ be a radial solution to \eqref{NLS} such that
\begin{align*}
E(\ve{u}_{0})=E(\Q) \quad \text{and} \quad K(\ve{u}_{0})>K(Q),
\end{align*}
which is globally defined in positive time. Then there exist $\eta_{1}$, $\theta_{1}\in \R$, $\lambda_{0}>0$, and constants $c, C>0$ such that
\begin{align}\label{ExpoAbove}
\|\ve{u}(t)-Q_{[\eta_{1}, \theta_{1}, \lambda_{0}]}\|_{\dot{H}^{1}}\leq Ce^{-c t}
\quad \text{for all } t\geq 0.
\end{align}
Moreover, the negative time of existence is finite.
\end{proposition}

To prove this proposition, we establish the following two lemmas.

\begin{lemma}\label{BoundDQ}
Let $\ve{u}(t)$ be a solution to \eqref{NLS} satisfying the conditions of Proposition~\ref{SupercriQ}. Then there exists $R_{1}>0$ such that for $R\geq R_{1}$ we have
\begin{align}\label{DoDeri}
\tfrac{d}{dt}I_{R}[\ve{u}(t)]\leq -2\delta(t) \quad \text{for all } t\geq 0.
\end{align}
\end{lemma}

\begin{proof}
From Lemma~\ref{VirialIden} we can write
\begin{align*}
	\frac{d}{dt }I_{R}[\ve{u}]&=F_{\infty}[\ve{u}(t)] + F_{R}[\ve{u}(t)]-F_{\infty}[\ve{u}(t)].
\end{align*}
for some $R$ to be specified below. Since
\[
F_{\infty}[\ve{u}(t)]=4[K(\ve{u})- 4P(\ve{u})] = 4[K(\Q) - K(\ve{u})]=-4\delta(t),
\]
we obtain
\[
\frac{d}{d t}I_{R}[\ve{u}]=-4\delta(t)+[F_{R}[\ve{u}(t)]-F_{\infty}[\ve{u}(t)]],
\]
where (we set $\ve{u}=(u, v, g)$)
\begin{align*}
F_{R}[\ve{u}(t)] - F_{\infty}[\ve{u}(t)] 
&= \int_{|x|\geq R}\left(-\tfrac{1}{4} \Delta \Delta w_{R}\right)\left(\tfrac{1}{m_{1}}|u|^{2}+\tfrac{1}{m_{2}}|v|^{2}+\tfrac{1}{m_{3}}|g|^{2}\right)\,dx \\
&\quad - 2\RE\int_{|x|\geq R}\Delta[w_{R}(x)]\overline{u}(x)^{2}v(x)g(x)dx \\
&\quad -2\RE\int_{|x|\geq R}\left[\tfrac{1}{m_{1}}|\nabla u|^{2} + \tfrac{1}{m_{2}}|\nabla v|^{2} + \tfrac{1}{m_{3}}|\nabla g|^{2} - 8\overline{u}^{2}vg\right]dx \\
&\quad + \RE\int_{|x|\geq R} \left[\tfrac{1}{m_{1}}\overline{u_{j}} u_{k} + \tfrac{1}{m_{2}}\overline{v_{j}} v_{k} + \tfrac{1}{m_{3}}\overline{g_{j}} g_{k}\right]\partial_{jk}[w_{R}(x)]dx.
\end{align*}

\textbf{Step 1. General bound on $|F_{R}[\ve{u}(t)]-F_{\infty}[\ve{u}(t)]|$.}
By choosing $\phi$ appropriately such that $\partial^{2}_{r}w_{R}\leq 2$, we observe that
\begin{align*}
 &\RE\int_{|x|\geq R} \left[\tfrac{1}{m_{1}}\overline{u_{j}} u_{k} + \tfrac{1}{m_{2}}\overline{v_{j}} v_{k} + \tfrac{1}{m_{3}}\overline{g_{j}} g_{k}\right]\partial_{jk}[w_{R}(x)]dx \\
 &\quad -2\RE\int_{|x|\geq R}\left[\tfrac{1}{m_{1}}|\nabla u|^{2} + \tfrac{1}{m_{2}}|\nabla v|^{2} + \tfrac{1}{m_{3}}|\nabla g|^{2}\right]dx \\
 &= \RE\int_{|x|\geq R}\left[\tfrac{1}{m_{1}}|\nabla u|^{2} + \tfrac{1}{m_{2}}|\nabla v|^{2} + \tfrac{1}{m_{3}}|\nabla g|^{2}\right](\partial^{2}_{r}w_{R}- 2)dx \leq 0.
\end{align*}
Then, H\"older's inequality shows that
\[
|F_{R}[\ve{u}(t)]-F_{\infty}[\ve{u}(t)]|
\lesssim
\int_{|x|\geq R} \tfrac{1}{R^{2}}[|u|^{2}+|v|^{2}+|g|^{2}]dx
+\int_{|x|\geq R} [|u|^{4}+|v|^{4}+|g|^{4}]dx.
\]

By Strauss' lemma (see Lemma~\ref{STRau}), we see that for any $f \in H^{1}(\mathbb{R}^4)$,
\[
\int_{|x|\geq R} |f(x)|^{4}\, dx \leq \|f\|_{L^\infty_{\{|x|\geq R\}}}^{2} \|f\|_{L^2}^2 
\leq \tfrac{C}{R^{3}} \|\nabla f\|_{L^2} \|f\|_{L^2}^{3}.
\]
Therefore,
\[
\int_{|x|\geq R} [|u|^{4}+|v|^{4}+|g|^{4}]dx\leq \tfrac{C}{R^{3}}
[\|\nabla u\|_{L^{2}}+\|\nabla v\|_{L^{2}}+\|\nabla g\|_{L^{2}}],
\]
where the constant $C$ depends only on $\|\ve{u}_{0}\|_{L^{2}}$.
Combining the above estimates, we conclude
\begin{align}\label{GeneralB}
	|F_{R}[\ve{u}(t)]-F_{\infty}[\ve{u}(t)]|
\leq C_{0}\left[ \tfrac{1}{R^{2}} +\tfrac{1}{R^{3}} (\delta(t)+K(\Q))^{\frac{1}{4}} \right].
\end{align}
\textbf{Step 2. Bound on $|F_{R}[\ve{u}(t)]-F_{\infty}[\ve{u}(t)]|$ for sufficiently small $\delta(t)$.}

Using \eqref{DecomUFree}, we can write $\ve{u}_{[\eta(t), \theta(t), \mu(t)]} = \Q + \ve{V}$, where $\|\ve{V}\|_{\dot{H}^{1}} \sim \delta(t)$. 
First, we claim that
\begin{align}\label{InfP}
    \mu_{\text{inf}} := \inf\left\{{\mu(t) : t \geq 0, \delta(t) \leq \delta_{1}}\right\} > 0
\end{align}
for $\delta_{1}$ sufficiently small. Indeed, writing $\ve{V} = (v_{1}, v_{2}, v_{3})$, mass conservation gives
\begin{align*}
    \|\ve{u}_{0}\|^{2}_{L^{2}} &\gtrsim \int_{|x| \leq \tfrac{1}{\mu(t)}} [|u(x,t)|^{2} + |v(x,t)|^{2} + |g(x,t)|^{2}] dx \\
    &= \tfrac{1}{\mu(t)^{2}} \int_{|x| \leq 1} [|u_{[\eta(t), \theta(t), \mu(t)]}|^{2} + |v_{[\eta(t),\theta(t), \mu(t)]}|^{2} + |g_{[\eta(t),\theta(t), \mu(t)]}|^{2}] dx \\
    &\gtrsim \tfrac{1}{\mu(t)^{2}} \left( \int_{|x| \leq 1} Q^{2} dx - \int_{|x| \leq 1} [|v_{1}|^{2} + |v_{2}|^{2} + |v_{3}|^{2}] dx \right).
\end{align*}
Since
\[
\|\ve{V}(t)\|_{L^{2}(|x| \leq 1)} \lesssim \|\ve{V}(t)\|_{L^{4}(|x| \leq 1)} \lesssim \|\ve{V}(t)\|_{\dot{H}^{1}} \lesssim \delta(t),
\]
we obtain
\[
\|\ve{u}_{0}\|_{L^{2}} \gtrsim \tfrac{1}{\mu(t)^{2}} \left( \int_{|x| \leq 1} Q^{2} dx - C\delta^{2}(t) \right).
\]
Taking $\delta_{1}$ sufficiently small yields \eqref{InfP}.

Let us define
\[
A_{R}(\ve{u}(t)) := F_{R}[\ve{u}(t)] - F_{\infty}[\ve{u}(t)].
\]
A change of variables shows that
\[
|A_{R}(\ve{u}(t))| = |A_{R\mu(t)}(\ve{V}(t) + \Q)|.
\]
Moreover, since
\begin{align}\label{IdenQ}
    A_{R\mu(t)}(\Q) = 0, \quad \text{and} \quad \|Q\|_{\dot{H}^{1}(|x| \geq r)} \sim \|Q\|_{L^{3}(|x| \geq r)} \sim r^{-1} \quad \text{for } r \geq 1,
\end{align}
the H\"older, Hardy, and Sobolev inequalities (cf. \eqref{SobL}) combined with \eqref{IdenQ} imply that for $R \geq 1$ (recall that $\ve{V} = (v_{1}, v_{2}, v_{3})$),
\begin{align*}
|A_{R}(\ve{u}(t))| &= |A_{R\mu(t)}(\Q + \ve{V}(t))| \\
&= |A_{R\mu(t)}(\Q + \ve{V}(t)) - A_{R\mu(t)}(\Q)| \\
&\leq C \big[ \|\ve{V}\|^{2}_{\dot{H}^{1}} + \tfrac{1}{R\mu(t)}\|\ve{V}\|_{\dot{H}^{1}} \\
&\quad + \tfrac{1}{(R\mu(t))^{3}}\|\ve{V}\|_{\dot{H}^{1}} + \tfrac{1}{(R\mu(t))^{2}}\|\ve{V}\|^{2}_{\dot{H}^{1}} \\
&\quad + \tfrac{1}{R\mu(t)}\|\ve{V}\|^{3}_{\dot{H}^{1}} + \|\ve{V}\|^{4}_{\dot{H}^{1}} \big] \\
&\leq C_{\ast} \left[ \delta(t)^{2} + \tfrac{1}{R} \delta(t) \right],
\end{align*}
where the constant $C_{\ast}$ depends only on $\mu_{\text{inf}}$.

\textbf{Step 3. Conclusion}. To establish the estimate \eqref{DoDeri}, it suffices to show that
\begin{align}\label{ClaimDel}
    |F_{R}[\ve{u}(t)]-F_{\infty}[\ve{u}(t)]|\leq 2\delta(t).
\end{align}
By Step 2, there exists $\delta_{2}>0$ such that if $\delta(t)\leq\delta_{2}$ and $R\geq R_{1}$, then
\[
|F_{R}[\ve{u}(t)]-F_{\infty}[\ve{u}(t)]|\leq C_{\ast}\left[ \delta(t)^{2}+\tfrac{1}{R}\delta(t) \right] \leq 2\delta(t),
\]
for $R_{1}$ sufficiently large. 

Next, we consider the case $\delta(t)>\delta_{2}$. Define the function
\[
f_{R}(\delta):=C_{0}\left[ \tfrac{1}{R^{2}} +\tfrac{1}{R^{3}} (\delta+K(\Q))^{\frac{1}{2}} \right]-2\delta,
\]
where $C_{0}$ is the constant from \eqref{GeneralB}. Note that $f^{\prime\prime}_{R}(\delta)<0$ for all $\delta>0$.

For sufficiently large $R_{2}$, we observe that:
\begin{itemize}
\item $f_{R_{2}}(\delta_{2})\leq 0$,
\item $f^{\prime}_{R_{2}}(\delta_{2})\leq 0$.
\end{itemize}
Consequently, $f_{R}(\delta)\leq 0$ for all $\delta\geq \delta_{2}$ and $R\geq R_{2}$. Therefore, the bound \eqref{ClaimDel} holds for $R=\max\left\{R_{1}, R_{2}\right\}$.

This completes the proof of the result.
\end{proof}

\begin{lemma}\label{NeD}
Let $\ve{u}(t)$ be as in Proposition~\ref{SupercriQ}. Then there exist positive constants $c$ and $C$, and $R_{1}>0$ such that for $R\geq R_{1}$ we have
\begin{align}\label{DoQ11}
\int^{+\infty}_{t}\delta(s)ds \leq Ce^{-ct} \quad \text{for all } t\geq 0.
\end{align}
\end{lemma}

\begin{proof}
From Lemmas~\ref{BoundDQ} and \ref{VirialIden}, we deduce that
$\tfrac{d^{2}}{dt^{2}}V_{R}(t)=\tfrac{d}{dt}I_{R}[\ve{u}(t)]\leq -2\delta(t)$ for $R\geq R_{1}$.  Since $\tfrac{d^{2}}{dt^{2}}V_{R}(t)<0$ and $V_{R}(t)>0$ for all $t\geq0$, we conclude that $I_{R}[\ve{u}(t)]=\tfrac{d}{dt}V_{R}(t)>0$ for all $t\geq 0$. Therefore,
\[
2\int^{T}_{t}\delta(s)ds\leq -\int^{T}_{t}\tfrac{d}{ds}I_{R}[\ve{u}(s)]ds
=I_{R}[\ve{u}(t)]-I_{R}[\ve{u}(T)]\leq I_{R}[\ve{u}(t)]\leq CR^{2}\delta(t),
\]
where we have used the estimate $I_{R}[\ve{u}(t)]\leq CR^{2}\delta(t)$ for all $t\geq 0$ (see \eqref{EstimateV11} and \eqref{EstimateV22}). The Gronwall inequality then yields \eqref{DoQ11}.
\end{proof}

\begin{proof}[Proof of Proposition~\ref{SupercriQ}]
First, we show that 
\begin{align}\label{deltaN}
    \lim_{t \to \infty}\delta(t)=0.
\end{align}
Indeed, Lemma~\ref{NeD} guarantees the existence of a sequence $\left\{t_{n}\right\}_{n\in \mathbb{N}}$ with $t_{n}\to +\infty$ such that $\lim_{n \to \infty}\delta(t_{n})=0$. Fix such a sequence $\left\{t_{n}\right\}_{n\in \mathbb{N}}$.  

Now, assume by contradiction that \eqref{deltaN} fails. Then, there exists a sequence $\left\{t^{\prime}_{n}\right\}_{n\in \mathbb{N}}$ such that $\delta(t^{\prime}_{n})\geq \epsilon$ for some $\epsilon\in(0, \delta_{0})$. By passing to subsequences of $\left\{t_{n}\right\}_{n\in \mathbb{N}}$ and $\left\{t^{\prime}_{n}\right\}_{n\in \mathbb{N}}$ if necessary, we may assume that  
\[
t_{n}<t^{\prime}_{n}, \quad \delta(t^{\prime}_{n})=\epsilon, \quad \delta(t)<\epsilon \quad \text{for all $t\in [t_{n}, t^{\prime}_{n})$}.
\]
Note that on $[t_{n}, t^{\prime}_{n})$, the parameters $\alpha(t)$, $\theta(t)$, and $\mu(t)$ are well-defined, and (cf. \eqref{DecomUFree})  
\[
\ve{u}_{[\theta(t), \mu(t)]}(t)=(1+\alpha(t))\Q+\ve{h}(t).
\] 
Further, by taking a subsequence if necessary, we have  
\begin{align}\label{BounLa}
    \lim_{n\to \infty}\mu(t_{n})=\mu_{\infty}\in (0, +\infty).
\end{align}
Indeed, from the estimate $\left|\tfrac{\mu^{\prime}(t)}{\mu(t)^{3}}\right|\leq C\delta(t)$ and \eqref{DoQ11}, we deduce that  
\begin{equation}\label{BoLamda}
\left|\tfrac{1}{\mu(t)^{2}}-\tfrac{1}{\mu(t_{n})^{2}}\right|\leq C_{0}e^{-ct_{n}}
\qtq{for $t\in [t_{n}, t^{\prime}_{n})$.}
\end{equation}

Next, suppose $\mu_{\infty}=\infty$. Let $r_{0}>0$. By H\"older's, Hardy's, and Sobolev's inequalities, we obtain  
\begin{align*}
    |V_{R}(t_{n})|&\lesssim
r^{4}_{0}K(\Q)+\|\ve{u}(t_{n})\|^{2}_{{L}^{4}(|x|\geq r_{0})}.
\end{align*}
Since $\ve{u}_{[\eta(t_{n}), \theta(t_{n}), \mu(t_{n})]} \to \Q$ in $\dot{H}^{1}$, it follows that for every $r_{0}>0$,  
\begin{equation}\label{Limuv}
\|\ve{u}(t_{n})\|^{2}_{{L}^{4}(|x|\geq r_{0})}\to 0
\qtq{as $n\to \infty$.}
\end{equation}
Passing to the limit $n\to \infty$ and then $r_{0} \to 0$, we conclude  
\[
\lim_{n\to \infty}V_{R}(t_{n})=0.
\]
However, since $\tfrac{d}{dt}V_{R}>0$ for all $t\geq 0$ (cf. Lemma~\ref{NeD}), we have $V_{R}(t)<0$ for $t\geq 0$, which is a contradiction.  
Thus, $\mu_{\infty}<\infty$. In particular, \eqref{BoLamda} implies that $\mu(t)\leq 2\mu_{\infty}$ on $\cup [t_{n}, t^{\prime}_{n})$.  

Since $|\alpha^{\prime}(t)|\lesssim \mu(t)^{2}|\delta(t)|\lesssim|\delta(t)|$ on $\cup [t_{n}, t^{\prime}_{n})$ (cf. \eqref{EstimateFree}), estimate \eqref{DoQ11} yields  
\begin{equation}\label{limnn}
\lim_{n\to \infty}|\alpha(t_{n})-\alpha(t^{\prime}_{n})|=0.
\end{equation}
As $|\alpha| \sim |\delta|$ (cf. \eqref{EstimateOne}), we have  
\[
|\alpha(t_{n})| \sim |\delta(t_{n})|\to 0 \quad \text{and} \quad 
|\alpha(t^{\prime}_{n})| \sim |\delta(t^{\prime}_{n})|=\epsilon>0,
\]
which contradicts \eqref{limnn}. Therefore, $\lim_{t \to \infty}\delta(t)=0$. In particular, the parameters $\alpha(t)$, $\mu(t)$, $\eta(t)$, and $\theta(t)$ are well-defined for large $t$, and from \eqref{InfP}, we deduce that $\mu_{\infty}>0$.  

Moreover, since $\mu(t)\leq 2\mu_{\infty}$ for sufficiently large $t$, estimate \eqref{EstimateOne} implies  
\[
\delta(t)+\|\ve{h}(t)\|_{\dot{H}^{1}}\sim |\alpha(t)|
\leq C\int^{+\infty}_{t}|\alpha^{\prime}(s)|ds\leq 
C\int^{+\infty}_{t}\mu(s)^{2}\delta(s)\,ds\leq Ce^{-ct}.
\]
Additionally, by \eqref{EstimateFree}, we infer the existence of $\eta_{\infty}$ and $\theta_{\infty}$ such that  
\[
\lim_{t\to +\infty}|\eta(t)-\eta_{\infty}|=0, \quad \lim_{t\to +\infty}|\theta(t)-\theta_{\infty}|=0, \quad \lim_{t\to +\infty}|\mu(t)-\mu_{\infty}|=0,
\]
which, by stability theory, implies  
\[
\|\ve{u}(t)-\Q_{[\eta_{\infty},\theta_{\infty}, \lambda_{\infty}]}\|_{\dot{H}^{1}}\leq  Ce^{-ct}\qtq{for all $t\geq 0$.}
\]

Finally, we prove finite-time blow-up for negative times. Suppose by contradiction that $\ve{u}$ is globally defined for negative times. Define $\ve{v}(t,x):=\overline{\ve{u}(-t,x)}$. Then, Lemmas~\ref{BoundDQ} and \ref{NeD} also hold for negative times. In particular, we obtain  
\[
\lim_{t\to \pm\infty}\delta(t)=0.
\]
Furthermore, $\frac{d}{dt}V_{R}(t)\to 0$ as $t\to \pm\infty$, and $\frac{d}{dt}V_{R}(t)> 0$ for all $t\in \R$ (cf. Lemma~\ref{NeD}). Since $\frac{d^{2}}{dt^{2}}V_{R}(t)< 0$ for all $t\in\R$, we arrive at a contradiction.  

This completes the proof of the proposition.
\end{proof}


\section{Spectral properties of the linearized operator}\label{SpecLine}

Recall from \eqref{Decomh} that we have
\[
\mathcal{L} := \begin{pmatrix}
0 & -L_{I} \\
L_{R} & 0
\end{pmatrix},
\]
where $L_{I}$ and $L_{R}$ are defined in Section~\ref{EqLinea}. The main objective of this section is to establish some spectral properties that will be used in subsequent sections.

The primary goal of this section is to prove the following result.

\begin{lemma}\label{SpecLL}
Let $\sigma(\L)$ denote the spectrum of the operator $\L$, defined on the space $(L^{2}(\R^{4}:\R))^{6}$ with domain $(H^{2}(\R^{4}:\R))^{6}$. The operator $\L$ admits two simple eigenfunctions $e_{+} = (Y, Z, W)$ and $e_{-} = (\overline{Y}, \overline{Z}, \overline{W})$, both belonging to the Schwartz space $(\Sch(\R^{4}:\R))^{6}$, with corresponding real eigenvalues $\pm \lambda_{1}$, where $\lambda_{1} > 0$. Moreover, the real part of the spectrum satisfies
\[
\sigma(\L) \cap \R = \left\{-\lambda_{1}, 0, \lambda_{1}\right\},
\]
and the essential spectrum of $\L$ is given by
\[
\sigma_{\text{ess}}(\L) = \left\{i\xi: \xi \in \R\right\}.
\]
\end{lemma}
\begin{proof}
Note that the operator $\mathcal{L}$ is a compact perturbation of $(-i \tfrac{1}{m_{1}}\Delta, -i \tfrac{1}{m_{2}}\Delta, -i \tfrac{1}{m_{3}}\Delta)$. Indeed, $Q$ decays at infinity. Consequently, the essential spectrum of $\mathcal{L}$ satisfies $\sigma_{\text{ess}}(\mathcal{L}) = i \mathbb{R}$. In particular, the intersection $\sigma(\mathcal{L}) \cap (\mathbb{R} \setminus \{0\})$ consists solely of eigenvalues. 

Lemma~\ref{NegativeL} in Appendix~\ref{S:A2} shows that $\mathcal{L}$ has a negative eigenvalue $-\lambda_{1}$ (and, by conjugation, it also has the corresponding positive eigenvalue $\lambda_{1} > 0$). Thus, $\{\pm \lambda_{1}\} \subset \sigma(\mathcal{L})$. 

Furthermore, employing the same reasoning as developed in \cite[Subsection 7.2.2]{DuyMerle2009}, we deduce that the eigenfunctions $e_{\pm}$ belong to $(\Sch(\R^{4}:\R))^{6}$. Here, $e_{+} = (Y, Z, W)$ is the eigenfunction associated with the eigenvalue $\lambda_{1}$, and $e_{-} = \overline{e_{+}} = (\overline{Y}, \overline{Z}, \overline{W})$ is the eigenfunction associated with the eigenvalue $-\lambda_{1}$. 

\begin{remark}\label{PLf} 
A straightforward computation shows that for any $\ve{h}$, $\ve{u} \in \dot{H}^{1}$, the following properties hold:
\begin{align*}
&\F(e_{\pm}) = \F(i\Q_{p}) = \F(i\Q_{q}) = \F(\Lambda \Q) = 0, \quad \F(\Q) < 0, \\
&\F(\ve{h}, \ve{u}) = \F(\ve{u}, \ve{h}), \quad \F(\L \ve{h}, \ve{u}) = -\F(\ve{h}, \L \ve{u}), \\
&\F(\ve{h}, i\Q_{p}) = \F(\ve{h}, i\Q_{q}) = \F(\ve{h}, \Lambda \Q) = \F(\ve{h}, \partial_{j}\Q) =0,
\end{align*}
for $j = 1, \ldots, 4$.
\end{remark}

\begin{remark}\label{Nzero}
We observe that $\F(e_{+}, e_{-}) \neq 0$. Indeed, suppose for contradiction that $\F(e_{+}, e_{-}) = 0$. Define the subspace
\[
E = \mbox{span}\left\{i\Q_{p}, i\Q_{q}, e_{+}, e_{-}, \Lambda \Q, \partial_{j} \Q : j = 1, \ldots, 4\right\},
\]
which has codimension $9$. Then, using the identities established in Remark~\ref{PLf}, we see that $\F(h) = 0$ for all $h \in E$. However, this leads to a contradiction because $\F$ is positive definite on a co-dimension $8$ subspace (cf. Proposition~\ref{CoerQuadra}).
\end{remark}

It remains to prove that $\sigma(\mathcal{E}) \cap (\mathbb{R} \setminus \{0\}) = \{-\lambda_{1}, \lambda_{1}\}$. Before proceeding with the proof, we require the following result. Recall that $\mathcal{F}$ is the quadratic form defined in \eqref{Quadratic}.
\begin{proposition}\label{FCove}
There exists a constant $ C > 0$ such that for every $\ve{h} \in \tilde{G}^{\bot} $, the following inequality holds:
\[
\F(\ve{h}) \geq C \|\ve{h}\|^{2}_{\dot{H}^{1}},
\]
where the orthogonal complement $ G^{\bot}$ is defined as
\begin{align*}
\tilde{G}^{\bot} := \bigg\{ \ve{h} \in \dot{H}^{1} \, \bigg| \, & \F(\ve{h}, e_{+}) = \F(\ve{h}, e_{-}) = (i\Q_{p}, \ve{h})_{\dot{H}^{1}} = (i\Q_{q}, \ve{h})_{\dot{H}^{1}} = (\Lambda \Q, \ve{h})_{\dot{H}^{1}} = 0, \\
& (\partial_{j} \Q, \ve{h})_{\dot{H}^{1}} = 0 \, \text{ for } j = 1, \ldots, 4 \bigg\}.
\end{align*}
\end{proposition}
\begin{proof}[Proof of Proposition~\ref{FCove}] 
We first show that if $\ve{h} \in \tilde{G}^{\bot}$, then $\F(\ve{h}) > 0$. Suppose, for contradiction, that there exists $\ve{g} \in \tilde{G}^{\bot}$ with $\ve{g} \neq 0$ such that $\F(\ve{g}) \leq 0$. From Remarks~\ref{PLf} and~\ref{Nzero}, we have
\begin{equation}\label{FEE}
\F(e_{-}) = 0, \quad \F(e_{+}) = 0,
\quad \text{and} \quad 
\F(e_{+}, e_{-}) \neq 0.
\end{equation}
Define the subspace
\[
E_{-} := \mbox{span}\left\{ i\Q_{p}, i\Q_{q}, \Lambda \Q, e_{+}, \ve{g}, \partial_{j} \Q : j = 1, \ldots, 4\right\}.
\]
From \eqref{FEE}, it follows that $\F(\ve{h}) \leq 0$ for all $\ve{h} \in E_{-}$. Since $i\Q_{p}$, $i\Q_{q}$, $\Lambda \Q$, $\ve{g}$, and $\partial_{j} \Q$ are orthogonal in the real Hilbert space $\dot{H}^{1}$ and $\F(e_{-}) = 0, \quad \F(e_{+}) = 0$, we deduce that $\mbox{dim}_{\R} E_{-} = 9$.  However, Proposition~\ref{CoerQuadra} states that $\F$ is positive definite on a co-dimension $8$ subspace of $\dot{H}^{1}$, which leads to a contradiction. Therefore, $\F(\ve{h}) > 0$ for all $\ve{h} \in \tilde{G}^{\bot}$.
Finally, since $Q$ decays at infinity, a compactness argument ensures that coercivity holds on $\tilde{G}^{\bot}$.
\end{proof}

To complete the proof of Lemma~\ref{SpecLL}, we must show that $\sigma(\EE)\cap (\R \setminus \{0\}) = \{-\lambda_{1}, \lambda_{1}\}$. 
Assume for contradiction that there exists $\ve{f} \in H^{2}$ with $\ve{f} \neq 0$ such that $\L \ve{f} = -\lambda_{0}\ve{f}$, where $\lambda_{0} \in \R \setminus \{0, -\lambda_{1}, \lambda_{1}\}$. Using the identity $\F(\L \ve{g}, \ve{h}) = -\F(\ve{g}, \L \ve{h})$, we derive:
\[
(\lambda_{1} + \lambda_{0})\F(\ve{f}, e_{+}) = (\lambda_{1} - \lambda_{0})\F(\ve{f}, e_{-}) = 0
\quad \text{and} \quad 
\lambda_{0}\F(\ve{f}, \ve{f}) = -\lambda_{0}\F(\ve{f}, \ve{f}),
\]
which simplifies to:
\[
\F(\ve{f}, e_{+}) = \F(\ve{f}, e_{-}) = \F(\ve{f}, \ve{f}) = 0.
\]

Decompose $\ve{f}$ as:
\[
\ve{f} = i\beta_{0}\Q_{p} + i\beta_{1}\Q_{q} + \sum_{j=1}^{4} \alpha_{j}\partial_{j} \Q + \gamma \Lambda \Q + \ve{g},
\]
where $\ve{g} \in \tilde{G}^{\bot}$, and the coefficients are defined by:
\[
\beta_{0} = \frac{(\ve{f}, i\Q_{p})_{\dot{H}^{1}}}{\|\Q_{p}\|^{2}_{\dot{H}^{1}}}, \quad
\beta_{1} = \frac{(\ve{f}, i\Q_{q})_{\dot{H}^{1}}}{\|\Q_{q}\|^{2}_{\dot{H}^{1}}}, \quad
\alpha_{j} = \frac{(\ve{f}, \partial_{j} \Q)_{\dot{H}^{1}}}{\|\partial_{j} \Q\|^{2}_{\dot{H}^{1}}}, \quad
\gamma = \frac{(\ve{f}, \Lambda \Q)_{\dot{H}^{1}}}{\|\Lambda \Q\|^{2}_{\dot{H}^{1}}}.
\]

From Remark~\ref{PLf}, we observe that $\F(\ve{g}, \ve{g}) = \F(\ve{f}, \ve{f}) = 0$. By Proposition~\ref{FCove}, this implies:
\[
\|\ve{g}\|^{2}_{H^{1}} \lesssim \F(\ve{g}) = 0.
\]
Thus, $\ve{g} = 0$, and consequently $\lambda_{0}\ve{f} = \L \ve{f} = \L \ve{g} = 0$, which contradicts $\ve{f} \neq 0$. This completes the proof of Lemma~\ref{SpecLL}.

\end{proof}

\begin{remark}\label{KerLR}
As a direct consequence of Proposition~\ref{FCove}, we obtain
\begin{align}\label{Ker11}
	\text{Ker}(\L) = \text{span}\left\{i\Q_{p}, i\Q_{q}, \Lambda Q, \partial_{j}\Q : j = 1, \ldots, 4\right\}.
\end{align}
In particular, we deduce that
\begin{align}\label{Ker12}
	\text{Ker}(L_{R}) &= \text{span}\left\{ \Lambda Q, \partial_{j}\Q : j = 1, \ldots, 4\right\}, \\
	\text{Ker}(L_{I}) &= \text{span}\left\{\Q_{p}, \Q_{q}\right\}.
\end{align}
\end{remark}

\begin{remark}\label{PE1}
We observe that
\[
(e_{1}, \Q)_{K} \neq 0 \qtq{where} e_{1} = (\RE Y, \RE Z, \RE W),
\]
and $(\cdot,\cdot )_{K}$ denotes the inner product associated with the norm $K(\cdot)^{\frac{1}{2}}$ (see \eqref{HNN}).
To prove this, assume by contradiction that $(e_{1}, \Q)_{K} = 0$. Note that
\[
\lambda_{1}\F(e_{\pm}, \Q) = \pm\F(\L e_{\pm}, \Q) = \mp\F(e_{\pm}, \L \Q)
= \tfrac{1}{2}\lambda_{1}(e_{1}, \Q)_{K} = 0.
\]
Here, we have used the fact that $L_{I}e_{2} = -\lambda_{1} e_{1}$, where $e_{2} = (\IM Y, \IM Z, \IM W)$.
Additionally, since
\[
(i\Q_{q}, \Q)_{\dot{H}^{1}}=(i\Q_{p}, \Q)_{\dot{H}^{1}} = (\Lambda\Q, \Q)_{\dot{H}^{1}} = (\partial_{j}\Q, \Q)_{\dot{H}^{1}} = 0,
\]
Proposition~\ref{FCove} implies that $\F(\Q) > 0$, which leads to a contradiction (cf. Remark~\ref{PLf}). 
\end{remark}

\section{Construction of special solutions}\label{S:Existence}

We begin with some estimates that will be useful throughout this section. Recall that for $\ve{h}=(h_{1}, h_{2}, h_{3})$ (cf. Section~\ref{EqLinea}):

\[
\begin{aligned}
K(\ve{h}) &= (2\overline{h}_1 Q_2 Q_3 + 2\overline{Q}_1 h_2 Q_3 + 2\overline{Q}_1 Q_2 h_3, \\
&\quad 2h_1 Q_1 \overline{Q}_3 + Q_1^2 \overline{h}_3, 2h_1 Q_1 \overline{Q}_2 + Q_1^2 \overline{h}_2), \\
R(\ve{h}) &= (2\overline{h}_1 h_2 h_3 + 2\overline{h}_1 h_2 Q_3 + 2\overline{h}_1 Q_2 h_3 + 2\overline{Q}_1 h_2 h_3, \\
&\quad h_1^2 \overline{h}_3 + h_1^2 \overline{Q}_3 + 2h_1 Q_1 \overline{h}_3, h_1^2 \overline{h}_2 + h_1^2 \overline{Q}_2 + 2h_1 Q_1 \overline{h}_2).
\end{aligned}
\]

\begin{lemma}[Linear estimates]\label{LL1}
Let $I$ be a finite interval of length $|I|$, $\ve{h}\in S(I)$, and $\nabla \ve{h}\in Z(I)$. Then, there exists a positive constant $C$ independent of $I$ such that
\begin{equation}\label{Fi22}
\|\nabla K(\ve{h})\|_{N(I)} \leq |I|^{\frac{1}{3}} \|\nabla \ve{h}\|_{Z(I)}.
\end{equation}
Moreover, for $\ve{h}\in \text{L}^{3}$, we have
\begin{equation}\label{Fi1}
\|K(\ve{h})\|_{\text{L}_{x}^{\frac{4}{3}}} \leq C \|\ve{h}\|_{\text{L}_{x}^{4}}.
\end{equation}
\end{lemma}
\begin{proof}
First, note that by the Sobolev inequality,
\begin{equation}\label{SI1}
\left\|{f}\right\|_{L^{6}_{t}L^{6}_{x}} \lesssim \left\|\nabla {f}\right\|_{L^{6}_{t}L^{\frac{12}{5}}_{x}}.
\end{equation}
Additionally, H\"older's inequality shows that
\begin{equation}\label{Hnew}
\|fgh\|_{L^{\frac{4}{3}}} \leq \|f\|_{L^{4}} \|g\|_{L^{4}} \|h\|_{L^{4}}.
\end{equation}
The inequality \eqref{Fi1} is a direct consequence of \eqref{Hnew}.
On the other hand, H\"older's inequality also implies
\begin{equation}\label{H11}
\|f{g}h\|_{L^{2}_{t}L^{\frac{4}{3}}_{x}} \leq \|f\|_{L^{6}_{t}L^{\frac{12}{5}}_{x}} \|{g}\|_{L^{6}_{t}L^{6}_{x}} \|h\|_{L^{6}_{t}L^{6}_{x}}.
\end{equation}
Since $|\partial_{\alpha}Q| \lesssim |Q|$ for every multi-index $\alpha$, and $Q \in L^{4} \cap L^{\frac{12}{5}}$, combining \eqref{SI1} and \eqref{H11}, we obtain \eqref{Fi22}.
\end{proof}

\begin{lemma}[Nonlinear estimates]\label{LN1}
Let $\ve{h}$ and $\ve{g}$ be functions in $L^{4}$. We have that
\begin{equation}\label{Nol1}
\|R(\ve{h}) - R(\ve{g})\|_{L^{\frac{4}{3}}} \leq C \|\ve{h} - \ve{g}\|_{L^{4}} \left( \|\ve{h}\|_{L^{4}} + \|\ve{g}\|_{L^{4}} 
+\|\ve{h}\|^{2}_{L^{4}} + \|\ve{g}\|^{2}_{L^{4}}
\right).
\end{equation}
In addition, let $I$ be a finite interval of length $|I|$, $\ve{h}, \ve{g} \in S(I)$, and $\nabla\ve{h}, \nabla\ve{g} \in Z(I)$. There exists a positive constant $C$ independent of $I$ such that
\begin{equation}\label{Nol2}
\begin{split}
\|\nabla R(\ve{h}) - \nabla R(\ve{g})\|_{N(I)} \leq C \|\nabla \ve{h} - \nabla \ve{g}\|_{Z(I)} \times \\
 \left( \|\nabla \ve{h}\|_{Z(I)} + \|\nabla \ve{g}\|_{Z(I)} + \|\nabla \ve{h}\|^{2}_{Z(I)} + \|\nabla \ve{g}\|^{2}_{Z(I)} \right).
\end{split}
\end{equation}
\end{lemma}
\begin{proof}
Inequality \eqref{Nol1} is an immediate consequence of \eqref{Hnew}. Moreover, by combining \eqref{H11} and \eqref{SI1}, the inequality \eqref{Nol2} follows easily. 
\end{proof}

This result will be useful in this and the next section; see \cite{DuyMerle2009}.
\begin{lemma}\label{SumsE}
Let $a_{0}>0$, $t_{0}>0$, $p\in[1, \infty)$, $E$ a normed vector space, and $f\in L_{\text{loc}}^{p}((t_{0}, \infty); E)$. Suppose that there exist $\tau_{0}>0$ and $C_{0}>0$ so that
\[
\|f\|_{L^{p}(t, t+\tau_{0})}\leq C_{0}e^{-a_{0}t} \quad \text{for all $t\geq t_{0}$}.
\]
Then
\[
\|f\|_{L^{p}(t, \infty)}\leq \frac{C_{0}e^{-a_{0}t}}{1-e^{-a_{0}\tau_{0}}}.
\]
\end{lemma}

\begin{lemma}\label{BoundGra11}
Let $\ve{v}$ be a solution of \eqref{Decomh} satisfying
\begin{equation}\label{ExpH}
\|\ve{v}(t)\|_{{\dot{H}}^{1}}\leq Ce^{-c_{0}t}
\end{equation}
for some positive constants $C$ and $c_{0}$. Then for any admissible pair $(q, r)$ and sufficiently large $t$, we have
\begin{equation}\label{BoundCo}
\|\ve{v}\|_{S(t, +\infty)}
+\|\nabla \ve{v}\|_{\text{L}^{p}(t, +\infty; L^{q})}
\leq Ce^{-c_{0}t}.
\end{equation}
\end{lemma}
\begin{proof}
The estimate \eqref{BoundCo} follows from Strichartz estimates (cf. \eqref{Estriz}), Lemmas~\ref{LL1}, \ref{LN1} and \ref{SumsE}, and a continuity argument. See \cite[Lemma 5.7]{DuyMerle2009} for further details.
\end{proof}

\begin{proposition}\label{ApproxSo}
Let $a \in \mathbb{R}$. There exists a sequence $\{g^{a}_{j}\}_{j \geq 1}$ in $\mathcal{S}:=\big(\mathcal{S}(\mathbb{R}^{4})\big)^{3}$ satisfying the following properties: 
\begin{itemize}
    \item The first term is given by $g^{a}_{1} = a e_{+}$;
    \item For each $k \geq 1$, defining 
    \[
    U^{a}_{k}(t, x) := \sum^{k}_{j=1} e^{-j \lambda_{1} t} g^{a}_{j}(x),
    \]
    the approximation error satisfies
    \begin{equation}\label{Aproxlimi}
    \epsilon_{k} := \partial_{t} U^{a}_{k} + \mathcal{L} U^{a}_{k} - i {R}(U^{a}_{k}) = \mathcal{O}(e^{-(k+1) \lambda_{1} t}) \quad \text{in }\mathcal{S} \quad \text{as } t \to \infty.
    \end{equation}
\end{itemize}
Note that if $W^{a}_{k}:= (f^{a}_{k}, h^{a}_{k}, q^{a}_{k}) := U^{a}_{k} + \mathcal{Q}$, then error term becomes
    \[
    \epsilon_{k} := i\partial_{t}W^{a}_{k} + 
		(\tfrac{1}{2m_{1}}\Delta f^{a}_{k}, \tfrac{1}{2m_{2}}\Delta h^{a}_{k}, \tfrac{1}{2m_{3}}\Delta q^{a}_{k}) + {R}(f^{a}_{k}, h^{a}_{k}, q^{a}_{k}) = \mathcal{O}(e^{-(k+1) \lambda_{1} t}) \quad \text{in }\mathcal{S},
    \]
	as  $t \to \infty$.
\end{proposition}
\begin{proof}
The proof proceeds by induction. For the  case $k=1$, consider $U^{a}_{1} := a e^{-\lambda_{1} t}e_{+}$. We observe that:
$$
\partial_{t} U^{a}_{1} + \mathcal{L} U^{a}_{1} - i R(U^{a}_{1}) = -i R(U^{a}_{1}) = \mathcal{O}(e^{-2\lambda_{1}t}).
$$
This establishes \eqref{Aproxlimi} for $k=1$.

For the inductive step, assume there exist $g_{1}^{a},\ldots,g_{k}^{a}$ such that $U_{k}^{a}$ satisfies \eqref{Aproxlimi}. Then there exists $P_{k+1}^{a}\in\mathcal{S}$ such that as $t\to+\infty$:
\begin{align}\label{dua}
	\partial_{t}U_{k}^{a}+\mathcal{L} U_{k}^{a}
	=iR(U_{k}^{a})+e^{-(k+1)\lambda_{1}t}P_{k+1}^{a}+O\left(e^{-(k+2)\lambda_{1}t}\right)\ \mbox{ in }\  \mathcal{S}.
\end{align}
Since $(k + 1)\lambda_{1}$ is not in the spectrum of $\mathcal{L}$ (by Lemma~\ref{SpecLL}), we define:
$$
g_{k+1}^{a}:=-\left(\mathcal{L}-(k+1)\lambda_{1}\right)^{-1}P_{k+1}^{a}.
$$
Following the argument in \cite[Section 6.2]{CAPA2022}, we conclude $g_{k+1}^{a}\in  \mathcal{S}$. Let $U_{k+1}^{a}:=U_{k}^{a}+e^{-(k+1)\lambda_{1}t}g_{k+1}^{a}$. Then by construction and \eqref{dua}, $U_{k+1}^{a}$ satisfies:
$$
\partial_{t}U_{k+1}^{a}+\mathcal{L} U_{k+1}^{a}-iR(U_{k+1}^{a})=iR(U_{k}^{a})-iR(U_{k+1}^{a})+O\left(e^{-(k+2)\lambda_{1}t}\right)\mbox{ as }t\to+\infty.
$$
The explicit form of $R$ yields $R(U_{k}^{a})-R(U_{k+1}^{a})=O\left(e^{-(k+2)e_{0}t}\right)$ as $t\to+\infty$, which completes the inductive step and proves the proposition.
\end{proof}

\begin{proposition}\label{ContracA}
Let $a\in \R$. There exist constants $k_{0}>0$ and $t_{k}\geq 0$ such that for every $k\geq k_{0}$, the following holds:
\begin{enumerate}
    \item[\text(i)] There exists a  radial solution $W^{a}$ of \eqref{NLS} satisfying, for all $t\geq t_k$,
    \begin{equation}\label{Uniq}
    \|\nabla W^{a}(t) - \nabla W_k^a(t)\|_{Z(t, +\infty)}
    \leq e^{-(k+\frac{1}{2})\lambda_{1}t}.
    \end{equation}
    
    \item[\text(ii)] The radial solution $W^{a}$ is the unique solution to \eqref{NLS} satisfying \eqref{Uniq} for large $t$.
    
    \item[\text(iii)] The radial solution $W^{a}$ is independent of $k$ and satisfies, for large $t$,
    \begin{equation}\label{UniqVec}
    \|W^{a}(t)-\Q-ae^{-\lambda_{1} t}e_{+}\|_{\dot{H}^{1}}\leq e^{-\frac{3}{2}\lambda_{1}t}.
    \end{equation}
\end{enumerate}
\end{proposition}
\begin{proof}
The function $W^a$ is a solution of \eqref{NLS} if and only if $\ve{w}^a := W^a - \Q$ satisfies
\[
\partial_t \ve{w}^a + \L \ve{w}^a = iR(\ve{w}^a).
\]
From \eqref{Aproxlimi}, the approximation $\ve{v}_k^a := W_k^a - \Q$ fulfills the identity
\[
\partial_t \ve{v}_k^a + \L \ve{v}_k^a - iR(\ve{v}_k^a) = \varepsilon_k.
\]
Consequently, $W^a$ solves \eqref{NLS} precisely when $\ve{h} := W^a - W_k^a = \ve{w}^a - \ve{v}_k^a$ satisfies
\[
\partial_t \ve{h} + \L \ve{h} = i\big[R(\ve{v}_k^a + \ve{h}) - R(\ve{v}_k^a)\big] - \varepsilon_k.
\]
In component form (with $\ve{h}:=(h,g,r)$), this becomes
\[
i \partial_t \ve{h} + \left(\tfrac{1}{2m_{1}}\Delta h, \tfrac{1}{2m_{2}}\Delta g, \tfrac{1}{2m_{3}}\Delta r\right) = -K(\ve{h}) - \big[R(\ve{v}_{k}+ \ve{h}) - R(\ve{v}_{k})\big] - i \varepsilon_k.
\]

We therefore construct the solution $W^{a}$ to \eqref{NLS} via a fixed point argument. Define the operator
\begin{equation*}
[\text{M}_{k}(\ve{h})](t) := -\int^{\infty}_{t} U(t-s)\big[-iB(\ve{h}(s)) - i\big(R(\ve{v}_{k}(s)+\ve{h}(s)) - R(\ve{v}_{k}(s))\big) + \epsilon_{k}(s)\big]\,ds,
\end{equation*}
where the propagator $S_{P}(t)$ is given by
\[
U(t)=
\left(
\begin{array}{ccc}
e^{\frac{1}{2m_{1}}it\Delta} & 0 & 0 \\
0 & e^{\frac{1}{2m_{2}}it\Delta}& 0 \\
0 & 0 & e^{\frac{1}{2m_{3}}it\Delta}
\end{array}
\right).
\]
Fix $k>0$ and $t_{k}\geq 0$. We define the space
\begin{align*}
E^{k}_{Z}&:=\left\{ \ve{h} \in S(t_k, +\infty), \nabla \ve{h} \in Z(t_k, +\infty); \|h\|_{E_l^k} := \sup_{t \geq t_k} e^{(k + \frac{1}{2})\lambda_{1} t} \| \nabla \ve{h}\|_{Z(t, +\infty)} < \infty \right\},\\
L^{k}_{Z}&:=\left\{\ve{h} \in E^{k}_{Z}, \|\ve{h}\|_{E^{k}_{Z}}\leq 1\right\}.
\end{align*}
Note that $E_Z^k$ is a Banach space.

\begin{claim}\label{Est34}
There exists $k_{0}>0$ such that for all $k\geq k_{0}$, the following estimates hold:
\begin{enumerate}
    \item[(i)] For any $\ve{h}\in E^{k}_{Z}$,
    \begin{align}\label{Cla11}
    \| \nabla K (\ve{h})\|_{N(t, \infty)} &\leq \tfrac{1}{4 C^{\ast}}e^{-(k+\frac{1}{2})\lambda_{1}t}\|\ve{h}\|_{E_{Z}^{k}}.
    \end{align}
    
    \item[(ii)] There exists a constant $C_{k}$ (depending only on $k$) such that for all $\ve{h}, \ve{g}\in L^{k}_{Z}$ and $t\geq t_{k}$,
    \begin{align}\label{Cla22}
    \|\nabla (N(\ve{v}_{k}+\ve{g})-N(\ve{v}_{k}+\ve{h}))\|_{N(t, \infty)} &\leq C_{k}e^{-(k+\frac{3}{2})\lambda_{1}t}\|\ve{g}-\ve{h}\|_{E_{Z}^{k}},\\
		\label{Cla33}
    \|\epsilon_{k}\|_{N(t, \infty)} &\leq C_{k}e^{-(k+1)\lambda_{1}t}.
    \end{align}
 \end{enumerate}
\end{claim}
\begin{proof}[Proof of Claim~\ref{Est34}]
First, observe that \eqref{Cla33} follows directly from \eqref{Aproxlimi}.

Next, we establish estimate \eqref{Cla11}. Fix $\tau_{0}>0$. From \eqref{Fi22}, we derive
\[
\|\nabla K(\ve{h})\|_{N(t, t+\tau_{0})}\leq C_{1}\tau_{0}^{\frac{1}{3}}e^{-(k+\frac{1}{2})\lambda_{1}t}\|\ve{h}\|_{E_{Z}^{k}}.
\]
Hence, \eqref{Cla11} follows by applying Lemma~\ref{SumsE} for $k\geq k_{0}$, provided $\tau_{0}$ and $k_{0}$ are chosen appropriately.

Finally, we prove \eqref{Cla22}. By construction (see Proposition~\ref{ContracA}), we have the bound $\|\ve{v}^{a}_{k}\|_{Z(t, t+1)}\leq C_{k}e^{-\lambda_{1}t}$. Let $I:=[t,t+1]$. Using estimate \eqref{Nol2}, we obtain:
\begin{align*}
&\|\nabla (R(\ve{v}^{a}_{k}+\ve{g}) - R(\ve{v}^{a}_{k}+\ve{h}))\|_{N(I)} \\
&\quad \leq C_{1,2} \|\nabla \ve{h} - \nabla \ve{g}\|_{Z(I)} \Bigl( \|\nabla \ve{h}\|_{Z(I)} + \|\nabla \ve{g}\|_{Z(I)} \\
&\qquad + \|\nabla \ve{v}^{a}_{k}\|_{Z(I)} + \|\nabla \ve{h}\|^{2}_{Z(I)} + \|\nabla \ve{g}\|^{2}_{Z(I)} + \|\nabla \ve{v}^{a}_{k}\|^{2}_{Z(I)}\Bigr) \\
&\quad \leq C_{k,2} e^{-\lambda_{1}t} \|\nabla \ve{h} - \nabla \ve{g}\|_{Z(I)} \\
&\quad \leq C_{k,2} e^{-(k+\frac{3}{2})\lambda_{1}t} \|\ve{h} - \ve{g}\|_{E_{Z}^{k}}.
\end{align*}
Here, the constant $C_{k,2}$ depends only on $k$. An application of Lemma~\ref{SumsE} now yields \eqref{Cla22}. This completes the proof of the claim.
\end{proof}

With Claim~\ref{Est34} established and applying a fixed point argument, we can prove  the existence of a unique radial solution $W^{a}$ to \eqref{NLS} satisfying \eqref{Uniq}. By the uniqueness property in the fixed point argument, we conclude that $W^{a}$ is independent of the parameter $k$ 
(cf. \cite[Proposition 6.3, Step 2]{DuyMerle2009} for more details)
Finally, from estimates \eqref{Cla11} and \eqref{Cla22}, we obtain
\[
\|\nabla W^{a}(t) - \nabla W_k^a(t)\|_{\dot{H}^{1}}\leq C e^{-(k+\frac{1}{2})\lambda_{1}t}.
\]
Combining this with the asymptotic expansion $W^{a}(t)=\Q+ae^{-\lambda_{1}t}e_{+}+O(e^{-2\lambda_{1}t})$ (cf. Proposition~\ref{ApproxSo}), we derive \eqref{UniqVec}. This completes the proof of the proposition.
\end{proof}

\subsection{Construction of special solutions}
\begin{proof}[Proof of Theorem~\ref{Gcharc}]
From Proposition~\ref{ContracA} we see that
\[
K(W^{a}(t)) = K(\Q) + 2a e^{-\lambda_{1}t} (e_{1}, \Q)_{K} + O\left(e^{-\frac{3}{2}\lambda_{1}t}\right) \quad \text{as } t \to +\infty.
\]
We may assume that $(e_{1}, \Q)_{K} > 0$ (cf. Remark~\ref{PE1}), which implies that $K(W^{a}(t)) - K(\Q)$ has the same sign as $a$ for large times. In particular, by the variational characterization of $\Q$ (cf. Proposition~\ref{UBS}), we have that $K(W^{a}(t_{0})) - K(\Q)$ has the same sign as $a$. Defining
\[
\G^{+}(t,x) = W^{+1}(t+t_{0},x), \quad \G^{-}(t,x) = W^{-1}(t+t_{0},x),
\] 
for $t_{0}$ sufficiently large, we obtain two radial solutions $\G^{\pm}(t,x)$ of \eqref{NLS} that satisfy
\[
K(\G^{-}(0)) < K(\Q) \quad \text{and} \quad K(\G^{+}(0)) > K(\Q), 
\]
and such that
\[
\|\G^{\pm}(t) - \Q\|_{\dot{H}^{1}} \leq C e^{-\lambda_{1}t} \quad \text{for } t \geq 0.
\]
In particular, $E(\G^{\pm}) = E(\Q)$. Finally, Corollary~\ref{ClassC} shows that the solution $\G^{-}$ is defined for all $\mathbb{R}$ and scatters as $t \to -\infty$. 
This concludes the proof of the theorem.
\end{proof}

\section{A Uniqueness Result}\label{S:uniq}

The main objective of this section is to establish the following proposition and its corollary.

\begin{proposition}\label{UniqueU}  
Let $\ve{u}$ be a radial solution to \eqref{NLS} satisfying
\begin{equation}\label{UniqCon}
\|\ve{u}(t)-\Q\|_{\dot{H}^{1}}\leq Ce^{-ct} \qtq{for $t\geq 0$,}
\end{equation} 
for some positive constants $C$ and $c$. Then there exists a unique $a \in \R$ such that $\ve{u}=W^{a}$, where $W^{a}$ is the solution of \eqref{NLS} given in Proposition~\ref{ContracA}.
\end{proposition}

As a direct consequence of Propositions~\ref{UniqueU} and \ref{ContracA}, we obtain the following result.

\begin{corollary}\label{coroCla}
Let $a\neq0$. Then there exists $T_{a}\in \R$ such that 
\begin{equation}\label{C1s}
\begin{cases} 
W^{a}=W^{+1}(t+T_{a}) & \text{if } a>0, \\
W^{a}=W^{-1}(t+T_{a}) & \text{if } a<0.
\end{cases}
\end{equation}
\end{corollary}

Throughout this section, we introduce the linearized equation
\begin{equation}\label{CondiExp}
\partial_{t}\ve{v}+\L \ve{v}=g, \quad (t,x)\in [0, \infty)\times \R^{4},
\end{equation}
where $\ve{v}$ and $g$ are \textsl{radial} functions satisfying
\begin{align} \label{CondiExp22}
& \|\ve{v}(t)\|_{{ \dot{H}}^{1}}\leq Ce^{-c_{1}t},\\ \label{CondiExp33}
&\|\nabla g\|_{N(t, +\infty)}+\|g\|_{L_{x}^{\frac{4}{3}}}\leq Ce^{-c_{2}t},
\end{align}
for all $t\geq 0$, with $0 < c_{1} < c_{2}$.

By Strichartz estimates (cf. \eqref{Estriz}) and Lemma~\ref{SumsE}, and a continuity argument, we can obtain the following result (cf. \cite[Lemma 5.7]{DuyMerle2009}).

\begin{lemma}\label{AxuST11}
Under the assumptions \eqref{CondiExp}, \eqref{CondiExp22}, and \eqref{CondiExp33} with $0 < c_{1} < c_{2}$, we have
\begin{equation}\label{NewSt11}
\|\ve{v}\|_{\text{L}^{p}(t, +\infty; L^{q})}
\leq Ce^{-c_{1}t}
\end{equation}
for any admissible pair $(q, r)$.
\end{lemma}

In what follows, we will use the following notation: for a given $c > 0$, we denote by $c^{-}$ a positive number that is arbitrarily close to $c$ and satisfies $0 < c^{-} < c$.

\begin{proposition}\label{AxuST}
Consider $\ve{v}$ and $g$ radial functions satisfying \eqref{CondiExp}, \eqref{CondiExp22}, and \eqref{CondiExp33}. Then we have:
\begin{enumerate}[label=\rm{(\roman*)}]
\item If $\lambda_{1}\notin [c_{1}, c_{2})$, then
\begin{equation}\label{BoundHsec}
 \|\ve{v}(t)\|_{{ \dot{H}}_{1}}\leq Ce^{-c^{-}_{2}t}.
\end{equation}
\item If $\lambda_{1}\in [c_{1}, c_{2})$, then there exists $a\in \R$ so that
\begin{equation}\label{BoundHsec22}
 \|\ve{v}(t)-ae^{-\lambda_{1}t}e_{+}\|_{{ \dot{H}}^{1}}
\leq Ce^{-c^{-}_{2}t}.
\end{equation}
Recall that $\lambda_1 > 0$ represents the eigenvalue of the linearized operator $\L$, as defined in Lemma~\ref{SpecLL}.
\end{enumerate}
\end{proposition}
\begin{proof}
We closely follow the argument in \cite[Proposition 5.9]{DuyMerle2009} and \cite[Proposition 7.2]{MIWUGU2015}, which consider the scalar case.
Let 
\[
Y^{\bot}:= \left\{\ve{h} \in \dot{H}^{1}, \F(\ve{h}, e_{+}) =\F(\ve{h}, e_{-})= (i\Q_{p}, \ve{h})_{\dot{H}^{1}} = (i\Q_{q}, \ve{h})_{\dot{H}^{1}} =(\Lambda \Q, \ve{h})_{\dot{H}^{1}} =0\right\}.
\]

We write $\ve{v}$ as
\begin{equation}\label{DecompV}
\ve{v}(t)=\alpha_{+}(t)e_{+}+\alpha_{-}(t)e_{-}+\beta_{p}(t)i\Q_{p}+\beta_{q}(t)i\Q_{q}+\gamma(t)\Lambda \Q+v^{\bot}(t),
\end{equation}
where $v^{\bot}(t)\in Y^{\bot}\cap \dot{H}^{1}_{rad}$.

Recall that by Remark~\ref{Nzero}, we have $\F(e_{+}, e_{-})\neq 0$, so we can normalize the eigenfunctions $e_{\pm}$ such that $\F(e_{+}, e_{-})=1$. Then, Remark~\ref{PLf} implies
\begin{align*}
&	\alpha_{+}(t)=\F(\ve{v}(t), e_{-}), \quad	\alpha_{-}(t)=\F(\ve{v}(t), e_{+}),\\
& \beta_{p}(t)=\frac{1}{\|\Q_{p}\|_{\dot{H}^{1}}}(\ve{v}(t)-\alpha_{+}(t)e_{+}-\alpha_{-}(t)e_{-}, i\Q_{p})_{\dot{H}^{1}},\\
& \beta_{q}(t)=\frac{1}{\|\Q_{q}\|_{\dot{H}^{1}}}(\ve{v}(t)-\alpha_{+}(t)e_{+}-\alpha_{-}(t)e_{-}, i\Q_{q})_{\dot{H}^{1}}, \\
& \gamma(t)=\frac{1}{\|\Lambda\Q\|_{\dot{H}^{1}}}(\ve{v}(t)-\alpha_{+}(t)e_{+}-\alpha_{-}(t)e_{-}, \Lambda\Q)_{\dot{H}^{1}}.
\end{align*}

\textbf{Step 1. Differential equations:} First, we show that:
\begin{align}\label{FB1}
	&\frac{d}{dt}\F(\ve{v}(t))=2\F(g,\ve{v}),\\\label{alpaD}
	&\frac{d}{dt}\left(e^{-\lambda_{1}t}\alpha_{-}\right)=e^{-\lambda_{1}t}\F(g, e_{+}),\\\label{alpaD22}
	&\frac{d}{dt}\left(e^{\lambda_{1}t}\alpha_{+}\right)=e^{\lambda_{1}t}\F(g, e_{-}).
\end{align}

Indeed, note that by Remark~\ref{PLf}, we see that
\begin{align}\label{Faal}
\alpha_{-}^{\prime}(t)&=\F(\partial_{t}\ve{v}, e_{+})= \F(-\mathcal{L}\ve{v}, e_{+})+\F(g,e_{+})\\
&= \lambda_{1}\F(\ve{v}, e_{+})+\F(g, e_{+})=\lambda_{1}\alpha _{-}(t)+\F(g,e_{+}),
\end{align}
and
\begin{align}\label{Faal11}
\alpha_{+}^{\prime}(t)&=\F(\partial_{t}\ve{v}, e_{-})= \F(-\mathcal{L}\ve{v}, e_{-})+\F(g,e_{-})\\
&= -\lambda_{1}\F(\ve{v}, e_{-})+\F(g, e_{+})=-\lambda_{1}\alpha _{-}(t)+\F(g,e_{-}).
\end{align}
Combining \eqref{Faal} and \eqref{Faal11}, we obtain the equations \eqref{alpaD} and \eqref{alpaD22}.
On the other hand, from \eqref{CondiExp}, we get \eqref{FB1},
\[\tfrac{d}{dt}\F(\ve{v})=\tfrac{d}{dt}\F(\ve{v}, \ve{v})= 
2\F(\ve{v},\partial_{t}\ve{v})=2\F(\ve{v},-\mathcal{L}\ve{v})+2\F(\ve{v}, g)=2\F(\ve{v},\,g).\]

Next, we show that
\begin{align}\label{beta11}
&\frac{d}{dt}\beta_{p}(t)=\frac{\left(i\Q_{p},\ve{w}\right)_{\dot{H }^{1}}}{\left\|\Q_{p}\right\|_{\dot{H}^{1}}^{2}},\\\label{beta22}
&\frac{d}{dt}\beta_{q}(t)=\frac{\left(i\Q_{q},\ve{w}\right)_{\dot{H }^{1}}}{\left\|\Q_{q}\right\|_{\dot{H}^{1}}^{2}},\\\label{gamma11}
&\frac{d}{dt}\gamma(t)=\frac{ \left(\Lambda \Q,\ve{w}\right)_{\dot{H}^{1}}}{\left\|\Lambda \Q\right\|_{\dot{H}^{1}}^{2}},
\end{align}
where $\ve{w}:=g-\F(e_{-}, g)e_{+}-\F(e_{+},g) e_{-}-\L v^{\bot}$.
We will only prove equation \eqref{beta11}, as the proofs of \eqref{beta22} and \eqref{gamma11} are similar.

Indeed, by \eqref{CondiExp}, \eqref{DecompV}, \eqref{Faal}, and \eqref{Faal11}, we get
\begin{align*}
	\tfrac{d}{dt}\beta_{p}(t)&= \frac{1}{\left\|\Q_{p}\right\|_{\dot{H}^{1}}^{2}}(\partial_{t}\ve{v}-\alpha_ {+}^{\prime}(t)e_{+}-\alpha_{-}^{\prime}(t)e_{-},\,i \Q_{p})_{\dot{{H}}^{1}}\\
	&=\tfrac{1}{\left\|\Q_{p}\right\|_{\dot{H}^{1}}^{2}}(g-\mathcal{L}\ve{v}-\alpha_ {+}^{\prime}(t)e_{+}-\alpha_{-}^{\prime}(t)e_{-}, i \Q_{p}
	)_{\dot{{H}}^{1}}\\
	&=\tfrac{1}{\left\|\Q_{p}\right\|_{\dot{H}^{1}}^{2}}(g-\F(g,e_{-})e_{+}-\F(g, e_{+})e_{-}+{\L}{v}^{\bot}, i \Q_{p})_{\dot{{H}}^{1}}\\
	&=\tfrac{1}{\left\|\Q_{p}\right\|_{\dot{H}^{1}}^{2}}(\ve{w},i \Q_{p})_{\dot{{H}}^{1}},
\end{align*}
which shows \eqref{beta11}.

\textbf{Step 2. Decay estimates.} We will show that there exists a real number $a\in \R$ such that
\begin{align}\label{Exp11alfa00}
&|\alpha^{\prime}_{-}(t)|\leq C e^{-c_{2}t},\\ \label{Exp11alfa}
&|\alpha^{\prime}_{+}(t)|\leq C e^{-c_{2}t} \quad \text{if $\lambda_{1}\leq c_{1}$ or $c_{2}\leq \lambda_{1}$},\\\label{Exp11alfalamda}
&|\alpha_{+}(t)-ae^{-\lambda_{1}t}|\leq e^{-c_{2}t} \quad \text{if $c_{1}\leq \lambda_{1}<c_{2}$},
\end{align}
First, note that for any time interval $I$ with $|I|<+\infty$, we have
\begin{equation}\label{L1BoundF}
\begin{aligned}
\int_{I}|\F(\ve{f}(t), \ve{h}(t))|dt &\lesssim
\|\nabla \ve{f}\|_{N(I)} \|\nabla \ve{h}\|_{{L}^{2}(I:{L}^{4})}\\
&+|I| \| \ve{f}\|_{{L}^{\infty}(I: {L}^{\frac{4}{3}})}\| \ve{h}\|_{{L}^{\infty}(I: \text{L}^{4})}.
\end{aligned}
\end{equation}
Indeed, for any time interval $I$ with $|I|< \infty$, we observe that
\begin{equation*}
\begin{aligned}
&\int_{I}\left|\int_{\R^{4}} \nabla f(t)\nabla g(t) dx\right|\lesssim\|\nabla {f}\|_{L^{2}(I :L^{\frac{4}{3}})} \|\nabla {g}\|_{{L}^{2}(I: {L}^{4})}\\
&\int_{\R^{4}}  \left| f\,g\, Q^{2}\right|\, dx\lesssim \|f\|_{{L}_{x}^{\frac{4}{3}}}\|g\|_{{L}_{x}^{4}}\|Q\|^{2}_{{L}^{\infty}}.
\end{aligned}
\end{equation*}
Combining these inequalities with the definition of $\F$, we obtain \eqref{L1BoundF}.

Now, from \eqref{CondiExp33} and inequality \eqref{L1BoundF}, we obtain
\[
\int^{t+1}_{t}|e^{-\lambda_{1}s}\F(g(s), e_{+})|ds\leq Ce^{-(\lambda_{1}+c_{2})t}.
\]
In this case, Lemma~\ref{SumsE} yields
\[
\int^{\infty}_{t}|e^{-\lambda_{1}s}\F(g(s), e_{+})|ds\leq Ce^{-(\lambda_{1}+c_{2})t}.
\]
Since $\lim_{t\to+\infty}e^{-\lambda_{1}t}\alpha_{-}(t)=0$ (cf. \eqref{CondiExp22}), integrating equation \eqref{alpaD} between $t$ and $+\infty$ and applying the fundamental theorem of calculus, we establish \eqref{Exp11alfa00}.

Next, we prove \eqref{Exp11alfa}. First consider the case $\lambda_{1}<c_{1}$. Estimate \eqref{CondiExp22} implies that $\lim_{t\to+\infty}e^{\lambda_{1}t}\alpha_{+}(t)=0$. Using \eqref{L1BoundF} and following the same argument as above, we have
\[
\int^{\infty}_{t}|e^{\lambda_{1}s}\F(g(s), e_{-})|ds\leq Ce^{(\lambda_{1}-c_{2})t}.
\]
Integrating equation \eqref{alpaD22} between $t$ and $+\infty$ and applying the fundamental theorem of calculus again, we obtain 
\eqref{Exp11alfalamda}.

Next, we consider the case $c_{1}\leq \lambda_{1}<c_{2}$. 
Note that from \eqref{CondiExp33} and \eqref{L1BoundF} we obtain
\[\int_{t}^{t+1}|e^{\lambda_{0}s}\F(g(s), e_{-})|\,ds\leq Ce^{\lambda_{1}t}e^{-c_{2}t},\]
which together with Lemma~\ref{SumsE} implies that
\[\int_{t_{0}}^{+\infty}|e^{\lambda_{0}s}\F(g(s), e_{-})|\,ds\lesssim e^{\lambda_{1}t_{0}}e^{-c_{2}t_{0}}<\infty.\]
From the above estimate and \eqref{alpaD22}, we deduce that $\lim_{t\to+\infty}e^{\lambda_{1}t}\alpha_{+}(t)=a$ for some $a\in \R$ and 
\[|e^{\lambda_{1}t}\alpha_{+}(t)-a|\leq Ce^{\lambda_{1}t}e^{-c_{2}t},\]
which establishes \eqref{Exp11alfalamda}.

Finally, we consider the case $c_{1}<c_{2}\leq e_{0}$. Integrating equation \eqref{alpaD22} between $0$ and $t$ and applying the fundamental theorem of calculus, we obtain
\[\alpha_{+}(t)=e^{-\lambda_{0}t}\alpha_{+}(0)+e^{-\lambda_{0}t}\int_{0}^{t}e^{\lambda_{0}s}\F(g(s),e_{-})ds.\]

From estimate \eqref{CondiExp33} we deduce that
\[\left|\int_{0}^{t}e^{\lambda_{1}s}\F(g(s),e_{-})ds\right|\leq \left\{\begin{array}{cc}
Ce^{(\lambda_{1}-c_{2})t}, & \text{if $c_{2}<\lambda_{1}$,}\\
Ct, & \text{if $c_{2}=\lambda_{1}$,}
\end{array}\right.\]
which proves \eqref{Exp11alfa}.

\textbf{Step 3. Proof for the case $\lambda_1 \geq c_2$ or ($\lambda_1 < c_2$ and $a=0$).}
From the estimates in the previous step, we obtain
\begin{equation}\label{Al12}
|\alpha_{+}(t)| + |\alpha_{-}(t)| \leq Ce^{-c_{2}t}.
\end{equation}
We claim that
\begin{align}\label{bqBp}
\beta_{p}(t) \lesssim e^{-\frac{(c_{1}+c_{2})}{2}t},
\quad
\beta_{q}(t) \lesssim e^{-\frac{(c_{1}+c_{2})}{2}t},
\quad
\gamma(t) \lesssim e^{-\frac{(c_{1}+c_{2})}{2}t}.
\end{align}

To prove this, note that by \eqref{CondiExp33} and estimate \eqref{L1BoundF}, we have
\[
\int^{t+1}_{t} |\F(g(s), \ve{w}(s))| \, ds \leq Ce^{-(c_{1}+c_{2})t}.
\]
Lemma~\ref{SumsE} then implies
\[
\int^{\infty}_{t} |\F(g(s), \ve{w}(s))| \, ds \leq Ce^{-(c_{1}+c_{2})t}.
\]

From \eqref{CondiExp22}, it follows that $|\F(\ve{w}(t))| \lesssim \|\ve{w}(t)\|^{2}_{\dot{H}^{1}} \to 0$ as $t \to \infty$. Using \eqref{FB1}, we deduce
\[
|\F(\ve{w}(t))| \leq \int^{\infty}_{t} |\F(g, \ve{w}(t))| \, dt \leq Ce^{-(c_{1}+c_{2})t}.
\]

Since $\F(e_{+}, e_{-}) = 1$ and $\F(e_{+}) = \F(e_{-}) = 0$, Remark~\ref{PLf} yields
\[
\F(\ve{w}) = \F(v^{\bot}) + 2\alpha_{+}\alpha_{-}.
\]
By Proposition~\ref{FCove} and \eqref{Al12}, we conclude
\begin{equation}\label{EstimaV}
\|v^{\bot}(t)\|_{\dot{H}^{1}} \lesssim \sqrt{|\F(v^{\bot})|} \leq Ce^{-\tfrac{(c_{1}+c_{2})}{2}t}.
\end{equation}

Next, we establish the decay estimate for $\beta_{p}(t)$. First, observe from \eqref{Al12} that $\lim_{t \to +\infty} \beta_{p}(t) = 0$. Moreover, since
\begin{equation}\label{EstimaV123}
(i\Q_{p}, \mathcal{L} v^{\bot})_{\dot{H}^{1}} = (\mathcal{L}^{*}i\Delta \Q_{p}, v^{\bot})_{{L}^{2}} \lesssim \|\mathcal{L}^{*}i\Delta \Q_{p}\|_{L^{\frac{4}{3}}} \|v^{\bot}\|_{\dot{H}^{1}} \lesssim e^{-\tfrac{(c_{1}+c_{2})}{2}t},
\end{equation}
where we used $\mathcal{L}^{*}i\Delta \Q_{p} = L_{R}\Delta \Q_{p} \in L^{\frac{4}{3}}$, it follows from \eqref{beta11} that
\begin{align*}
\int_{t}^{t+1} |(\ve{w},i\Q_{p})_{\dot{H}^{1}}| \, ds &\lesssim e^{-c_{2}t} + \int_{t}^{t+1} |(i\Q_{p},\mathcal{L} v_{\bot}(s))_{\dot{H}^{1}}| \, ds \\
&\lesssim e^{-c_{2}t} + \int_{t}^{t+1} \int_{\mathbb{R}^{4}} |\mathcal{L}^{*}(i\Delta \Q_{p})\overline{v_{\bot}}(s)| \, dx \, ds \\
&\lesssim e^{-c_{2}t} + \|v_{\bot}(t)\|_{L^{\infty}\dot{H}^{1}} \lesssim e^{-\tfrac{c_{1}+c_{2}}{2}t}.
\end{align*}

Combining this estimate with Lemma~\ref{SumsE} and \eqref{beta11}, we obtain
\[
|\beta_{p}(t)| \lesssim e^{-\tfrac{c_{1}+c_{2}}{2}t}.
\]
A similar argument proves the estimates for $\beta_{q}(t)$ and $\gamma(t)$ in \eqref{bqBp}.

Finally, combining \eqref{Exp11alfa00}--\eqref{Exp11alfalamda} and \eqref{bqBp}, and recalling the decomposition \eqref{DecompV}, we conclude
\[
\|v(t)\|_{\dot{H}^{1}} \leq C e^{-c_{2}^{-}t}.
\]
This completes the proof for this case.

\textbf{Step 4: Proof of the case $c_{2} > \lambda_{1}$, and $a \neq 0$.}
By Step 2 and \eqref{CondiExp22}, if $c_{1} > \lambda_{1}$, we see that $a=0$. Therefore, in what follows we assume that 
$c_{1} \leq \lambda_{1}$, i.e., $\lambda_{1}\in [c_{1}, c_{2})$. Now, we set
\[\ve{w}(t) := \ve{v}(t) - a e^{-\lambda_{1}t} e_{+}.\]
Then 
\[\partial_{t}\ve{w}(t) + \L\ve{w}(t) = g(t), \quad \|\ve{w}(t)\|_{\dot{H}^{1}} \leq C e^{-c_{1}t}.\]
Writing $\overline{\alpha}_{+}(t) = \F(\ve{w}(t),e_{-})$, we see that $\overline{\alpha}_{+}(t) ={\alpha}_{+}(t)-ae^{-\lambda_{1}t}$. Thus, from \eqref{Exp11alfalamda},
\[\lim_{t \to +\infty} e^{\lambda_{1}t} \overline{\alpha}_{+}(t) = 0.\]
This implies that $\overline{\alpha}_{+}(t)$ and $g$ satisfy all the assumptions of Step 3, and we can conclude that
\[\|\ve{v}(t) - a e^{-\lambda_{1}t}e_{+}\|_{\dot{H}^{1}} \leq C e^{-c_{2}^{-}t}.\]
This completes the proof of the proposition.
\end{proof}

\begin{proof}[Proof of Proposition~\ref{UniqueU}]
Combining Lemmas~\ref{BoundGra11}, \ref{LN1}, \ref{AxuST11}, and \ref{SumsE} with Propositions~\ref{AxuST} and \ref{ContracA},
the proof follows the same lines as in \cite[Lemma 6.5]{DuyMerle2009}. We omit the details here. 
\end{proof}

\begin{proof}[Proof of Corollary~\ref{coroCla}]
Let $a\neq0$ and choose $T_{a}\in \R$ such that $|a|e^{-\lambda_{1}T_{a}}=1$. From \eqref{UniqVec} we obtain
\begin{equation}\label{W1}
\|W^{a}(t+T_{a})-\Q\mp e^{-\lambda_{1} t}e_{+}\|_{\dot{H}^{1}}
\leq e^{-\frac{3}{2}\lambda_{1}t}.
\end{equation}
Thus, $W^{a}(t+T_{a})$ satisfies the assumptions of Proposition~\ref{UniqueU}, which implies that there exists $\tilde{a}$
such that 
\[W^{a}(\cdot+T_{a})=W^{\tilde{a}}.\]
From \eqref{W1} and the uniqueness established in Proposition~\ref{ContracA}, we conclude that $\tilde{a}=1$ if $a>0$, and $\tilde{a}=-1$ if $a<0$, proving \eqref{C1s}. 
\end{proof}

\section{Proof of the main result}\label{S:proof}

\begin{proof}[Proof of Theorem~\ref{TH22}]

(i) Let $\ve{u}$ be a radial solution to \eqref{NLS} satisfying
\begin{equation}\label{Hdn}
E(\ve{u}_0) = E(\Q), \quad K(\ve{u}_0) < K(\Q).
\end{equation}
From Lemma~\ref{GlobalW}, we have that $\ve{u}$ is global. Suppose that $\ve{u}$ does not scatter, i.e., $\|\ve{u}\|_{{L}^{6}_{t, x}(\R \times \R^4)} = \infty$. 
Replacing $\ve{u}(t)$ with $\overline{\ve{u}}(-t)$ if necessary, Proposition~\ref{CompacDeca} and Corollary~\ref{ClassC} show that there exist $\eta_{0}$, $\theta_0 \in \mathbb{R}$, $\mu_0 > 0$, and constants $c, C > 0$ such that 
\[
\|\ve{u}_{[\eta_{0}, \theta_0, \mu_0]}(t) - \Q\|_{\dot{H}^1} \leq Ce^{-ct} \qtq{for $t\geq 0$.}
\]
Thus, $\ve{u}_{[\eta_{0}, \theta_0, \mu_0]}$ satisfies the assumptions of Proposition~\ref{UniqueU}. Therefore, by \eqref{Hdn}, Corollary~\ref{coroCla} implies the existence of $a < 0$ and $T_a$ such that
\[
\ve{u}_{[\eta_{0}, \theta_0, \mu_0]}(t) = W^{-1}(t + T_a).
\]
Consequently, $\ve{u}_{[\eta_{0}, \theta_0, \mu_0]}(t) = W^{-1}(t + T_a) = \G^{-}(t+t_{0})$ for some $t_{0}\in \R$, 
which completes the proof of part (i).

(ii) If $E(\ve{u}_0) = E(\Q)$ and $K(\ve{u}_0) = K(\Q)$, then by the variational characterization given in Proposition~\ref{UBS}, we deduce that $\ve{u}_0 = \Q$ up to the symmetries of the equation.

Finally, we prove part (iii). Let $\ve{u}$ be a radial solution to \eqref{NLS} defined on $[0, +\infty)$ (if necessary, replace $\ve{u}(t)$ with $\overline{\ve{u}}(-t)$) satisfying
\[
E(\ve{u}) = E(\Q), \quad K(\ve{u}_0) > K(\Q), \quad \text{and} \quad \ve{u}_0 \in {L}^2.
\]
Proposition~\ref{SupercriQ} guarantees that there exist $\eta_{0},\theta_0 \in \mathbb{R}$, $\mu_0 > 0$, and constants $c, C > 0$ such that
\[
\|\ve{u}_{[\eta_{0}, \theta_0, \mu_0]}(t) - \Q\|_{\dot{H}^1} \leq Ce^{-ct} \qtq{for $t\geq 0$.}
\]
Since $K(\ve{u}_0) > K(\Q)$, Corollary~\ref{coroCla} implies the existence of $a > 0$ and $T_a$ such that
\[
\ve{u}_{[\eta_{0}, \theta_0, \mu_0]}(t) = W^{+1}(t + T_a)=\G^{+}(t+t_{0}),
\]
for some $t_{0}\in \R$, which completes the proof of part (iii).

This concludes the proof of the theorem.

\end{proof}

\appendix
\section{Spectrum of the linearized operator}\label{S:A2}

This appendix is dedicated to showing that the operator $\L$ has at least one negative eigenvalue. 

Notice that since $\overline{\L \ve{v}} = -\L(\overline{\ve{v}})$, we infer that if $\lambda_{1} > 0$ is an eigenvalue of the operator $\L$ with eigenfunction $e_{+} = (Y, Z, W)$, then $-\lambda_{1}$ is also an eigenvalue of $\L$ with eigenfunction $e_{-} = \overline{e_{+}} = (\overline{Y}, \overline{Z}, \overline{W})$. Denoting $e_{1} = \RE e_{+}$ and $e_{2} = \IM e_{+}$, to show the existence of $e_{+}$, we must study the system
\begin{equation}\label{SiE}
\begin{cases} 
L_{R} e_{1} = \lambda_{1} e_{2}, \\
-L_{I} e_{2} = \lambda_{1} e_{1}.
\end{cases} 
\end{equation}

Lemma~\ref{CoerLi11} shows that $L_{I}$ on $L^{2}$ with domain $H^{2}$ is nonnegative. Consequently, since $L_{I}$ is self-adjoint, it follows that $L_{I}$ has a unique square root $(L_{I})^{\frac{1}{2}}$ with domain $H^{1}$. 

Now, consider the self-adjoint operator $\TT$ on $L^{2}$ with domain $H^{4}$, defined as
\[
\TT = (L_{I})^{\frac{1}{2}} L_{R} (L_{I})^{\frac{1}{2}}.
\]
Since
\[
\TT = (L_{I})^{2} - (L_{I})^{\frac{1}{2}}
\begin{pmatrix}
-4 Q_2 Q_3 & 0 & 0 \\
0 & 0 & -2 Q_1^2 \\
0 & -2 Q_1^2 & 0
\end{pmatrix}
(L_{I})^{\frac{1}{2}},
\]
and noting that $|\partial^{\alpha} Q_{j}(x)| \leq C_{\alpha} |Q_{j}(x)|$ for every multi-index $\alpha$, and $Q_{j}$ decays at infinity for $j=1$, $2$ ,$3$, it follows that $\TT$ is a relatively compact, self-adjoint perturbation of $((\tfrac{1}{2m_{1}}\Delta)^{2}, (\tfrac{1}{2m_{2}}\Delta)^{2}, (\tfrac{1}{2m_{3}}\Delta)^{2})$. By Weyl's theorem, this implies that $\sigma_{\text{ess}}(\TT) = [0, \infty)$.

Suppose there exists $\ve{g} \in H^{4}$ such that
\begin{equation}\label{egenT}
\TT \ve{g} = -\lambda^{2}_{1} \ve{g}.
\end{equation}
Defining
\[
e_{1} := (L_{I})^{\frac{1}{2}} \ve{g} \quad \text{and} \quad e_{2} := \frac{1}{\lambda_{1}} L_{R} (L_{I})^{\frac{1}{2}} \ve{g},
\]
we obtain a solution to \eqref{SiE}, which implies the existence of the eigenfunction $e_{+}$. 

Thus, to show the existence of $e_{+}$, we need to prove that the operator $\TT$ has at least one negative eigenvalue $-\lambda^{2}_{1}$, which is the content of the following result.

\begin{lemma}\label{NegativeL}
\[
\Pi(\TT) := \inf\left\{ (\TT \ve{g}, \ve{g})_{L^{2}} : \ve{g} \in H^{4}, \|\ve{g}\|_{L^{2}} = 1 \right\} < 0.
\]
\end{lemma}

\begin{proof}
Notice that since $L_{I}$ is self-adjoint on $L^{2}$ with domain $H^{2}$ and $\ker{L_{I}} = \{0\}$ (indeed, $\Q_{p} \notin L^{2}$ and $\Q_{q} \notin L^{2}$), it follows that the range of $L_{I}$ is dense in $L^{2}$. Using the same density argument developed in \cite[Claim 7.1]{DuyMerle2009}, it suffices to show that there exists $\mathbf{W} \in (\dot{H}^{2})^{3}$ such that
\begin{align}\label{PosL}
-(L_{R} \mathbf{W}, \mathbf{W})_{L^{2}} > 0.
\end{align}
In \cite[Claim 7.1]{DuyMerle2009}, it is shown that there exists a function $\varphi \in {H}^{2}$ with $-(L_{3} \varphi, \varphi)_{L^{2}} > 0$, where $L_{3} \varphi = -\Delta \varphi - 3Q^{2} \varphi$. We define
\[
\mathbf{W} = \left( \tfrac{\varphi}{2\sqrt{m_1}}, \tfrac{\varphi}{2\sqrt{2m_2}}, \tfrac{\varphi}{2\sqrt{2m_3}} \right).
\]
Then, we have (see proof of Lemma~\ref{Apx11})
\begin{align*}
\langle L_{R} \mathbf{W}, \mathbf{W} \rangle &= \langle L_{3} \varphi, \varphi \rangle  < 0,
\end{align*}
which implies \eqref{PosL}. 
\end{proof}

\section*{Acknowledgements}
The author wishes to express his sincere thanks to the referees for their valuable comments.

\bibliographystyle{siam}

\begin{thebibliography}{10}

\bibitem{AA2018}
{\sc A.~H. Ardila}, {\em Orbital stability of standing waves for a system of
  nonlinear schrödinger equations with three wave interaction}, Nonlinear
  Analysis, 167 (2018), pp.~1--20.

\bibitem{ArdilaCelyMeng}
{\sc A.~H. Ardila, L.~Cely, and F.~Meng}, {\em Threshold solutions for the
  energy-critical nls system with quadratic interaction}, \texttt{Preprint
  arXiv:2505.03124v1}.

\bibitem{AeDiFo2021}
{\sc A.~H. Ardila, V.~D. Dinh, and L.~Forcella}, {\em Sharp conditions for
  scattering and blow-up for a system of NLS arising in optical materials with
  $ \chi^3$ nonlinear response}, Comm. Partial Differ. Equ., 46 (2021),
  pp.~2134--2170.

\bibitem{SerraBadi}
{\sc M.~Badiale and E.~Serra}, {\em Critical nonlinear elliptic equations with
  singularities and cylindrical symmetry}, Rev. Mat. Iberoam., 20 (2004),
  pp.~33--66.

\bibitem{CamposFarahRoudenko}
{\sc L.~Campos, L.~G. Farah, and S.~Roudenko}, {\em Threshold solutions for the
  nonlinear {S}chr \"odinger equation}, Rev. Mat. Iberoam. 38 (2022), no. 5, pp. 1637--1708.

\bibitem{CAPA2022}
{\sc L.~Campos and A.~Pastor}, {\em Threshold solutions for cubic
  Schr\"{o}dinger systems}, \texttt{preprint, arXiv:2210.07369}.

\bibitem{CCC}
{\sc M.~Colin and T.~Colin}, {\em A numerical model for the raman amplification
  for laser-plasma interaction}, J. Comput. App. Math., 193 (2006),
  pp.~535--562.

\bibitem{COCOOh}
{\sc M.~Colin, T.~Colin, and M.~Ohta}, {\em Stability of solitary waves for a
  system of nonlinear Schrödinger equations with three wave interaction}, Ann.
  Inst. H. Poincaré Anal. Non Linéaire, 26 (2009), pp.~2211--2226.

\bibitem{CoDiSa}
{\sc M.~Colin, L.~D. Menza, and J.~C. Saut}, {\em Solitons in quadratic media},
  Nonlinearity, 29 (2016), pp.~1000--1035.

\bibitem{SCO}
{\sc M.~Colin and M.~Ohta}, {\em Bifurcation from semi-trivial standing waves
  and ground states for a system of nonlinear {\Large s}chrödinger equations},
  SIAM Journal on Math. anal., 44 (2012).

\bibitem{DuyMerle2009}
{\sc T.~Duyckaerts and F.~Merle}, {\em Dynamic of threshold solutions for
  energy-critical {NLS}}, Geom. funct. anal., 18 (2009), pp.~1787--1840.

\bibitem{FukayaHayashiInui2024}
{\sc N.~Fukaya, M.~Hayashi, and T.~Inui}, {\em Traveling waves for a nonlinear
  {Schr\"odinger} system with quadratic interaction}, Mathematische Annalen,
  388 (2024), pp.~1357--1378.

\bibitem{GaoMengXuZheng}
{\sc C.~Gao, F.~Meng, C.~Xu, and J.~Zheng}, {\em Scattering theory for
  quadratic nonlinear {S}chr\"odinger system in dimension six}, J. Math. Anal.
  Appl., 541 (2025), p.~128708.

\bibitem{HajaiejStuart2004}
{\sc H.~Hajaiej and C.~Stuart}, {\em On the variational approach to the
  stability of standing waves for the nonlinear Schr\"odinger equation},
  Advanced Nonlinear Studies, 4 (2004), pp.~469--501.

\bibitem{Lib2011}
{\sc Q.~Han and F.~Lin}, {\em Elliptic Partial Differential Equations}, vol.~1
  of Courant Lecture Notes in Mathematics, American Mathematical Society,
  2~ed., 2011.

\bibitem{NayaOzaTana}
{\sc N.~Hayashi, T.~Ozawa, and K.~Tanaka}, {\em On a system of nonlinear 
{S}chr\"odinger equations with quadratic interaction}, Ann. Inst. H. Poincar\'e
  Anal. Non Lin\'eaire, 30 (2013), pp.~661--690.

\bibitem{KenigMerle2006}
{\sc C.~E. Kenig and F.~Merle}, {\em Global well-posedness, scattering and
  blow-up for the energy-critical, focusing, non-linear {S}chr\"odinger
  equation in the radial case}, Invent. Math, 166 (2006), pp.~645--675.

\bibitem{KiiVisan2008}
{\sc R.~Killip and M.~Visan}, {\em Nonlinear {S}chr\"odinger equations at
  critical regularity}, in Lecture notes of the 2008 Clay summer school
  "Evolution Equations", 2008.


\bibitem{LiLiuTangXu2025}
{\sc X.~Li, C.~Liu, X.~Tang and G.~Xu}, {\em Dynamics of radial threshold solutions for generalized energy-critical Hartree equation}, Forum Mathematicum, 37 (2015), 1469--1502.

\bibitem{LiHayashi2014}
{\sc C.~Li and N.~Hayashi}, {\em Recent progress on nonlinear {Schr\"odinger}
  system with quadratic interactions}, The Scientific World Journal,  2014 (2014) 214821.
	
\bibitem{LiuYangZhang2024}
{\sc X.~Liu, K.~Yang and T.~Zhang}, {\em Dynamics of threshold solutions for the energy-critical inhomogeneous NLS}, \texttt{Preprint arXiv:2409.00073}.

\bibitem{HaInuiNi}
{\sc K.~Nishimura. M.~Hamano, T.~Inui}, {\em Scattering for the quadratic nonlinear
  Schrödinger system in $\R^{4}$ without mass-resonance condition}, Funkcialaj
  Ekvacioj, 64 (2021), pp.~261--291.

\bibitem{Masaki2022}
{\sc S.~Masaki}, {\em On scalar-type standing-wave solutions to systems of
  nonlinear {Schr\"odinger} equations}, \texttt{ Preprint arXiv 2212.00754v2}.

\bibitem{MasakiTsukuda2023}
{\sc S.~Masaki and R.~Tsukuda}, {\em Scattering below ground states for a class
  of systems of nonlinear {Schr\"odinger} equations}, \texttt{ Preprint arXiv
  2303.12351}.

\bibitem{MEXu}
{\sc F.~Meng and C.~Xu}, {\em Scattering for mass-resonance nonlinear
  Schrödinger system in 5d}, J. Differential Equations, 275 (2021),
  pp.~837--857.

\bibitem{MIWUGU2015}
{\sc C.~Miao, Y.~Wu, and G.~X.}, {\em Dynamics for the focusing,
  energy-critical nonlinear Hartree equation}, Forum Mathematicum, 27 (2015),
  pp.~373--447.

\bibitem{NogueraPastor2022}
{\sc N.~Noguera and A.~Pastor}, {\em Blow-up solutions for a system of
  {Schr\"odinger} equations with general quadratic-type nonlinearities in
  dimension five and six}, Calc. Var. Partial Differ. Equ., 61 (2022), p.~35.


\bibitem{Strauss1977}
{\sc W.~A.~Strauss}, {\em Existence of solitary waves in higher dimensions}, 
Comm. Math. Phys., 55 (1977), pp.~149--162.


\bibitem{OgaTsu2023}
{\sc T.~Ogawa and S.~Tsuhara}, {\em Global well-posedness for the sobolev
  critical nonlinear Schr\"{o}dinger system in four space dimensions}, J. Math.
  Anal. Appl., 524 (2023), p.~127052.


\bibitem{Tsu2025}
{\sc S.~Tsuhara}, {\em Global well-posedness for the sobolev-critical nonlinear
  Schrödinger system with general nonlinear terms}, \texttt{ Preprint},  (2025).
	
\bibitem{YangZengZhang2022}
{\sc K.~Yang, C.~Zeng and X.~Zhang}, {\em Dynamics of Threshold Solutions for Energy Critical NLS with Inverse Square Potential}, SIAM J. Math. Anal., 54 (2022), pp.~173--219.
	

\bibitem{Zhang2016}
{\sc H.~Zhang}, {\em Local well-posedness for a system of quadratic nonlinear
  {Schr\"odinger} equations in one or two dimensions}, Math. Methods Appl.
  Sci., 39 (2016), pp.~4257--4267.

\end{thebibliography}

\end{document}